\documentclass{article}

\usepackage{caption}
\usepackage{xspace}
\usepackage{enumitem}
\usepackage{tikz}
\usepackage{xcolor}
\usetikzlibrary{patterns}
\usepackage[section]{placeins}

\usepackage[latin1]{inputenc}
\usepackage{amsmath}
\usepackage{amsthm,amssymb,amsfonts}
\usepackage{graphicx}
\usepackage{verbatim}
\usepackage{color,cite}
\usepackage{latexsym}
\usepackage{lscape}
\usepackage[all,cmtip]{xy}
\usetikzlibrary{arrows, decorations.pathmorphing, positioning, backgrounds, fit, petri}
\usepackage{hyperref}

\theoremstyle{plain}
\newtheorem{theorem}{Theorem}[section]
\newtheorem{lemma}[theorem]{Lemma}

\newtheorem{proposition}[theorem]{Proposition}
\newtheorem{corollary}[theorem]{Corollary}

\newtheorem{MainThm}{Theorem}

\newtheorem{definition}[theorem]{Definition}

\newtheorem{remark}[theorem]{Remark}

\title{Rank $2$ Affine Manifolds in Genus $3$}
\author{David Aulicino\thanks{This material is based upon work supported by the National Science Foundation under Award Nos. DMS - 1204414, DMS - 1600360, DMS - 1738381 and PSC-CUNY Grant B $\sharp$~60571-00 48.} $\,$ and Duc-Manh Nguyen}
\date{}

\begin{document}

\newcommand{\splin}{$\text{SL}_2(\mathbb{R})$}
\newcommand{\spolin}{$\text{SO}_2(\mathbb{R})$}
\newcommand{\nc}{\newcommand}
\newcommand{\ann}[1]{\marginpar{\small{#1}}}

\def\Re{\operatorname{Re}}
\def\Im{\operatorname{Im}}

\def\R{\mathbb{R}}
\def\Z{\mathbb{Z}}
\def\Q{\mathbb{Q}}
\def\C{\mathbb{C}}

\def\Gal{\operatorname{Gal}}
\nc\Mod{\mathfrak M}

%\%\nc\bA{\mathbb{A}}
\nc\bB{\mathbb{B}}
\nc\bC{\mathbb{C}}
\nc\bD{\mathbb{D}}
\nc\bE{\mathbb{E}}
\nc\bF{\mathbb{F}}
\nc\bG{\mathbb{G}}
\nc\bH{\mathbb{H}}
\nc\bI{\mathbb{I}}
\nc{\bJ}{\mathbb{J}}
\nc\bK{\mathbb{K}}
\nc\bL{\mathbb{L}}
\nc\bM{\mathbb{M}}
\nc\bN{\mathbb{N}}
\nc\bO{\mathbb{O}}
\nc\bP{\mathbb{P}}
\nc\bQ{\mathbb{Q}}
\nc\bR{\mathbb{R}}
\nc\bS{\mathbb{S}}
\nc\bT{\mathbb{T}}
\nc\bU{\mathbb{U}}
\nc\bV{\mathbb{V}}
\nc\bW{\mathbb{W}}
\nc\bY{\mathbb{Y}}
\nc\bX{\mathbb{X}}
\nc\bZ{\mathbb{Z}}

% mathcal
\nc\cA{\mathcal{A}}
\nc\cB{\mathcal{B}}
\nc\cC{\mathcal{C}}
\nc\cD{\mathcal{D}}
\nc\cE{\mathcal{E}}
\nc\cF{\mathcal{F}}
\nc\cG{\mathcal{G}}
\nc\cH{\mathcal{H}}
\nc\cI{\mathcal{I}}
\nc{\cJ}{\mathcal{J}}
\nc\cK{\mathcal{K}}
\nc\cL{\mathcal{L}}
\nc\cM{\mathcal{M}}
\nc\cN{\mathcal{N}}
\nc\cO{\mathcal{O}}
\nc\cP{\mathcal{P}}
\nc\cQ{\mathcal{Q}}
%\rnc\cR{\mathcal{R}}
\nc\cS{\mathcal{S}}
\nc\cT{\mathcal{T}}
\nc\cU{\mathcal{U}}
\nc\cV{\mathcal{V}}
\nc\cW{\mathcal{W}}
\nc\cY{\mathcal{Y}}
\nc\cX{\mathcal{X}}
\nc\cZ{\mathcal{Z}}

\newcommand{\inv}{\tau}
\nc\ol{\overline}
\nc\ul{\underline}
\nc\id{\mathrm{id}}
\nc\sig{\sigma}
\nc\smin{\setminus}
\nc\eps{\epsilon}
\nc\Sig{\Sigma}
\nc\SL{\mathrm{SL}}
\nc\GL{\mathrm{GL}}
\nc\ra{\rightarrow}
\nc\Gr{\mathrm{G}}
\nc\Pres{{\rm Pres}}
\nc\Twist{{\rm Twist}}
\nc\del{\delta}
\nc\veps{\varepsilon}

\nc\dcoverodd{\tilde{\mathcal H}^{\rm odd}_{(2,2)}(2)}
\nc\dcoverhyp{\tilde{\mathcal H}^{\rm hyp}_{(2,2)}(2)}
\nc\dcoverprinc{\tilde{\cH}(1,1)}
\nc\prym{\tilde{\cQ}(4,-1^4)}
\nc\prymprinc{\tilde{\cQ}(2^2,-1^4)}
\nc\prymmin{\tilde{\mathcal{Q}}(3,-1^3)}
\nc\prymthreezero{\tilde{\mathcal{Q}}(2,1,-1^3)}
\nc\onedouble{\mathcal{H}(2,1^2)}
\nc\principal{\mathcal{H}(1^4)}
\nc\bk{\mathbf{k}}

\maketitle

\begin{abstract}
We complete the classification of rank two affine manifolds in the moduli space of translation surfaces in genus three.  Combined with a recent result of Mirzakhani and Wright, this completes the classification of higher rank affine manifolds in genus three.% and affirms a conjecture of Mirzakhani in genus three that all higher rank affine manifolds arise from covering constructions of quadratic or Abelian differentials in lower genus.%\footnote{Version of \date{\today}}
\end{abstract}

\tableofcontents

\section{Introduction} \label{sec:intro}

%The works of \cite{EskinMirzakhaniInvariantMeas, EskinMirzakhaniMohammadiOrbitClosures, Filip2} prove that \splin ~orbit closures in the moduli space of Abelian differentials are quasi-projective subvarieties that admit a finite ergodic \splin -invariant measure.  The (cylinder) rank of an orbit closure was introduced in \cite{WrightCylDef}, and it counts half the degrees of freedom in absolute periods of points (or translation surfaces) in the orbit closure.  Mirzakhani conjectured that if the rank is at least two, then the orbit closure covers a stratum of Abelian or quadratic differentials.  Though this conjecture is not true in full generality by \cite{McMullenMukamelWrightGothicLocus} and forthcoming work of Eskin, McMullen, Mukamel, and Wright, the exceptions appear to be extremely rare, and the conjecture should be true in most cases \cite{MirzakhaniWrightFullRank, ApisaHypAISClass}.  In this paper we confirm that all rank two affine manifolds in genus three arise from covering constructions.

A translation surface is a Riemann surface with a flat geometry given by a holomorphic $1$-form on the surface.  It is natural to consider the moduli space of translation surfaces, which is the moduli space of Riemann surfaces carrying the bundle of holomorphic $1$-forms, also called Abelian differentials.  This moduli space admits an action by \splin.
The works of \cite{EskinMirzakhaniInvariantMeas, EskinMirzakhaniMohammadiOrbitClosures, Filip2} prove that \splin ~orbit closures are {\em affine submanifolds}  admitting  a finite ergodic \splin-invariant measure, and are also quasi-projective subvarieties of the moduli space of Abelian differentials.  However, a complete classification of all quasi-projective subvarieties of moduli space that are \splin-invariant is beyond the scope of current techniques.  Nevertheless, such a classification was obtained in genus two prior to the aforementioned results \cite{McMullenGenus2}.

%*****************************************
The purpose of this paper is to contribute to the classification of the orbit closures in higher genus. Specifically, we complete the classification of rank two  affine submanifolds in genus three (see below for a brief introduction to the notion of cylinder rank).  Combined with the recent result \cite[Th. 1.1]{MirzakhaniWrightFullRank}, a consequence of our result is the following

\begin{MainThm}
\label{thm:classify:orbits:in:g3}
Let $M=(X,\omega)$ be a translation surface in a stratum $\cH(\kappa)$ in genus three. Then either the closure of  the $\GL^+(2,\R)$-orbit of $M$ is one of the following: the  component of $\cH(\kappa)$ that contains $M$, the intersection of this component with the hyperelliptic locus, with the Prym locus, or with the intersection of these two loci, or $M$ is completely periodic in the sense of Calta, and the ratio of the circumferences of any pair of parallel cylinders belongs to a finite set.
\end{MainThm}

%***********************************************
The (cylinder) rank of an orbit closure was introduced in \cite{WrightCylDef}, and it counts half the degrees of freedom in absolute periods of points (or translation surfaces) in the orbit closure.  By definition, the cylinder rank of an orbit closure of surfaces in genus $g$ cannot be greater than $g$. Any stratum of translation surfaces in genus $g$ is of rank $g$.
On the other hand, closed \splin -orbits are examples of rank one affine submanifolds as well as the Prym eigenform loci discovered by \cite{McMullenPrym}. Following  a result of \cite{WrightCylDef}, every surface in a rank one orbit closure is {\em completely periodic} (in the sense of Calta), meaning that if the surface has a regular closed geodesic in some direction, then any other trajectory in the same direction is either a saddle connection or a closed (regular) geodesic.  Orbit closures of rank at least two are said to be of {\em higher rank}.

%The main purpose of this paper is to complete the classification of rank two affine submanifolds  (orbit closures) in genus three.
The works of \cite{NguyenWright, AulicinoNguyenWright, AulicinoNguyenGen3TwoZeros} established the classification of rank two orbit closures in strata in genus three with at most two zeros.  This paper exclusively concerns rank two orbit closures in $\cH(2,1,1)$ and $\cH(1,1,1,1)$.

%The result of this paper, combined with the result of \cite{MirzakhaniWrightFullRank}, shows that  the only higher rank quasi-projective subvarieties of the moduli space of genus three translation %surfaces that admit an action by \splin ~are the ones that are already well-known to algebraic geometers: namely, the Prym locus, the hyperelliptic locus, and their intersections (see Theorem %\ref{thm:main:rk2:g3}).

All of the previous works heavily relied on ``cylinder proportions'' to establish the symmetry required to prove that a translation surface admitted an involution.  However, this approach seems to be unrealistic for the last two strata because of the large number of cylinder diagrams that must be analyzed.  (There are 190 3-cylinder diagrams, 92 4-cylinder diagrams, and 26 5-cylinder diagrams to consider.\footnote{Computed in Sage using the surface\_dynamics package.  The results in this paper do not rely on any Sage computations.}) On the other hand, for translation surfaces satisfying most cylinder diagrams in a stratum with several zeros, it is possible to deform the surface by collapsing some cylinders  to get a translation surface in
a lower stratum.  We developed new tools based on this observation that rely on \cite{MirzakhaniWrightBoundary}.

%This observation prompted us to develop new tools based on degenerating  surfaces by collapsing cylinders to produce surfaces in a lower stratum, and then use our previous classifications to conclude.

%3 cyls: H(2,1,1) = 119, H(1,1,1,1) = 71
%4 cyls: H(2,1,1) = 50, H(1,1,1,1) = 42
%5 cyls: H(2,1,1) = 10, H(1,1,1,1) = 16
%6 cyls: H(1,1,1,1) = 4

%The techniques in the proof of this result build upon the previous classification in genus three.
While it will be necessary to compute a few cylinder proportions, it is degeneration techniques that will take center stage in the proofs in this paper. A posteriori, all rank two affine manifolds in these two strata contain rank two affine manifolds in lower strata of genus three in their boundary.
Eventually, we will show that every surface in any rank two affine manifold in genus three admits a Prym involution (see the definitions below).
Some affine manifolds consist exclusively of hyperelliptic Riemann surfaces, that is, they have a hyperelliptic involution in addition to the Prym involution. The existence of those  involutions will be established by observing that they exist on the surfaces in the boundary, and with the appropriate assumptions, they can be extended to surfaces in the interior of the affine manifold (see Proposition \ref{DblCovExtSimpCyl}).  Combined with a dimension count, this allows us to get the complete list of all rank two affine manifolds in the remaining strata.

Another key ingredient is Proposition \ref{CylStDens}, which may be interesting in its own right.  This proposition generalizes the results of Masur and Kontsevich-Zorich on the density of the set of Jenkins-Strebel differentials with a single cylinder in any  stratum of translation surfaces (see also \cite{LanneauComponents} for related results in the space of quadratic differentials).
The essential observation in its proof is the flat surface implication of the result of \cite{EskinMirzakhaniMohammadiOrbitClosures} that the upper triangular orbit closure is equal to the \splin ~orbit closure.

Mirzakhani conjectured that if the rank is at least two, then the orbit closure covers a stratum of Abelian or quadratic differentials.
The result of this  paper thus  confirms the conjecture in genus three. It is also verified in other contexts.  In \cite{MirzakhaniWrightFullRank}, Mirzakhani and Wright prove that the only orbit closures of maximal rank are hyperelliptic loci and connected components of the moduli spaces of translation surfaces with specified orders of zeros, known as strata.
In~\cite{ApisaHypAISClass}, it is proven that all higher rank orbit closures in hyperelliptic connected components of strata arise from covering
constructions.
Though this conjecture is not true in full generality by \cite{McMullenMukamelWrightGothicLocus} and forthcoming work of Eskin, McMullen, Mukamel, and Wright, the exceptions appear to be extremely rare.

Together with the result of \cite{MirzakhaniWrightFullRank}, our results complete the classification of higher rank orbit closures in genus three.  We hope that this classification facilitates results in genus three concerning higher rank affine manifolds, e.g. \cite[Thm. 2.8]{AulicinoZeroExpGen3} follows easily from the main result of this paper and the Forni Geometric Criterion \cite{ForniCriterion}.  Furthermore, we hope that it inspires ideas that lead to classifications in higher genus.

Finally, we remark that we believe that a classification of rank three affine manifolds in genus three should be relatively easy to accomplish using our techniques.  However, given the general nature of the result announced in \cite{MirzakhaniWrightBoundary}, we refrain from attempting such a classification with our methods.

\subsection{Statement of the Main Result}

Let $M=(X,\omega)$ be a translation surface in genus three. Throughout this paper, by a {\em Prym involution} of $M$, we will mean an automorphism $\inv$ of the Riemann surface $X$ such that
\begin{itemize}
 \item[a)] $\inv^2=\id_X$,

 \item[b)] $\inv^*\omega=-\omega$,

 \item[c)] $\inv$ has exactly four fixed points in $X$.
\end{itemize}
Remark that condition b) means that $\inv$ is isometric for the flat metric structure whose derivative is given by $-\id$ at regular points.

Let $Y:=X/\langle \inv \rangle$ be the quotient of $X$ by the action of a Prym involution $\inv$. By definition, there exists a double cover $\pi: X \ra Y$ ramified at four points (the fixed points of $\inv$). It follows from the  Riemann-Hurwitz formula that $Y$ is a Riemann surface of genus one. Condition b) implies that there exists a meromorphic quadratic differential $\eta$ on $Y$ such that $\pi^*\eta=\omega^2$.

We will call the subset of $\cH_3=\Omega\cM_3$ consisting of surfaces admitting a Prym involution the {\em Prym locus} and denote it by $\cP$. As usual, the subset of $\cH_3$ consisting of pairs $(X,\omega)$ where $X$ is a hyperelliptic surface is called the {\em hyperelliptic locus}, and we denote it by $\cL$.

Naturally, the intersection of $\cP$ with each connected component  $\cH^{*}(\kappa)$ of a stratum $\cH(\kappa)$ (here $*$ is either ``hyp'' or ``odd'')  consists of standard double covers of quadratic  differentials in some stratum in genus one.

It follows from Lemma~\ref{lm:2inv:dblcover} below that the intersection $\cP\cap\cL\subset \cH_3$ consists of unramified double covers of translation surfaces in genus two.  Actually, it is not difficult to show that any unramified double cover of a surface in $\cH_2$ must be contained in $\cP\cap\cL$. Our main result can be stated as follows

\begin{theorem}\label{thm:main:rk2:g3}
Let $\cM$ be a rank two affine submanifold of a connected component of a stratum $\cH(\kappa)$ in genus three. Then either $\cM$ is a component of $\cP\cap\cH^*(\kappa)$, or $\cM$ is a component of $\cM=\cP\cap\cL\cap\cH^*(\kappa)$. In the latter case $\cM$ is a locus consisting of unramified double covers of surfaces in a stratum of $\cH_2$.
\end{theorem}

Theorem~\ref{thm:main:rk2:g3} was proved for strata  $\cH(\kappa)$ such that $|\kappa|\leq 2$ by our  previous classifications (see \cite{NguyenWright, AulicinoNguyenWright, AulicinoNguyenGen3TwoZeros}). Namely, in $\cH(4)$ we have two components $\cH^{\rm odd}(4)$ and $\cH^{\rm hyp}(4)$, the Prym locus does not intersect $\cH^{\rm hyp}(4)$, and $\cP\cap\cH^{\rm odd}(4)=\tilde{\cQ}(3,-1^3)$. The stratum $\cH(3,1)$ does not intersect $\cP$, hence there are no rank two affine submanifolds in $\cH(3,1)$.
The stratum $\cH(2,2)$ has two components $\cH^{\rm odd}(2,2)$ and $\cH^{\rm hyp}(2,2) \subset \cL$. We have

\begin{eqnarray*}
\cP\cap\cH^{\rm hyp}(2,2) & = & \dcoverhyp=\tilde{\cQ}(1^2,-1^2),\\
\cP \cap \cH^{\rm odd}(2,2) & = & \prym,\\
\cP \cap \cL \cap \cH^{\rm odd}(2,2) & = & \dcoverodd.
\end{eqnarray*}

\begin{remark}
Let $M$ be a surface in $\cH(2,2)\cap\cP$. If $M\in \cH^{\rm odd}(2,2)$, then the Prym involution exchanges the zeros (cone points) of $M$, but if $M \in \cH^{\rm hyp}(2,2)$, then the Prym involution fixes each of the zeros of $M$.
\end{remark}

Let  $M=(X,\omega)$ be a translation surface that admits a Prym involution  $\inv$.  Let $M \in  \cH(2,1^2)$.  Since $\inv^*\omega=-\omega$, the double zero of $\omega$ must be fixed, and the two simple zeros must be exchanged by $\inv$.  By assumption, $\inv$ has three  regular fixed points. Therefore,  $\cP \cap \cH(2,1^2)=\tilde{\cQ}(2,1,-1^3)$. If $M\in \cH(1^4)$, then $\inv$ must exchange two pairs of simple zeros and has four regular fixed points. Therefore, $\cP\cap \cH(1^4)=\prymprinc$.

Assume in addition that $M$ admits a hyperelliptic involution.  Then $M$ is an unramified  double cover of a translation surface in genus two by Lemma~\ref{lm:2inv:dblcover}. It follows in particular that $M\not\in \cH(2,1^2)$. If $M\in \cH(1^4)$, then $M$ is an unramified double cover of a surface in $\cH(1,1)$.  Denote the locus of such surfaces by $\dcoverprinc$.   Then, $\allowbreak \cP\cap\cL\cap\cH(1^4)=\dcoverprinc$. By Proposition~\ref{H11CoverConnProp}, this locus is a connected affine submanifold of $\cH(1^4)$.

Note that the loci $\prymthreezero$ and $\prymprinc$ are connected by a result of Lanneau~\cite[Th. 1.2]{LanneauComponents}.  From the observations above, to prove Theorem~\ref{thm:main:rk2:g3}, it suffices to show

\begin{theorem}\label{thm:rk2:H211:H1111}
 Let $\cM$ be a rank two affine submanifold in $\cH_3=\Omega\cM_3$.
 \begin{itemize}
  \item[$\bullet$] If $\cM\subset \cH(2,1^2)$, then $\cM=\tilde{\cQ}(2,1,-1^3)$,

  \item[$\bullet$] If $\cM\subset \cH(1^4)$, then either $\cM=\prymprinc$, or $\cM=\dcoverprinc$.
 \end{itemize}
\end{theorem}

Figure~\ref{fig:rk2:mnfd:gen3} gives the list of all rank two affine manifolds in genus three and the relations between them.

\begin{figure}[htb!]
% \centering
%\begin{tikzpicture}[scale=0.40]
%\draw(-2,6) node[above] {-};
%\draw(2,6) node[above] {$\tilde{\mathcal{Q}}(3,-1^3)$};
%\draw(-4,4) node[below,left] {$\dcoverhyp = \tilde{\mathcal{Q}}(1^2,-1^2)$};
%\draw(0,4) node[below] {-};
%\draw(4,4) node[below,right] {$\tilde{\mathcal{Q}}(4,-1^4) \supset \dcoverodd$};
%\draw(0,1) node[below] {$\tilde{\mathcal{Q}}(2,1,-1^3)$};
%\draw(0,-2) node[below] {$\tilde{\mathcal{Q}}(2^2,-1^4) \supset \tilde{\mathcal{H}}(1,1)$};
%\draw (-3.5,4.5)--(-2.5,5.5);
%\draw (-.5,4.5)--(-1.5,5.5);
%\draw (.5,4.5)--(1.5,5.5);
%\draw (3.5,4.5)--(2.5,5.5);
%\draw (-.5,1.5)--(-4.5,2.5);
%\draw (0,1.5)--(2,5.5);
%\draw (0,-1.5)--(0,-0.5);
%\draw (-3.5,-1.5)--(-7,2.5);
%\draw (3.5,-1.5)--(7,2.5);
%\end{tikzpicture}
\centering
 \begin{tikzpicture}[scale=0.7]
\node (Q1) at (0,2) {$\widetilde{\mathcal{Q}}(2^2,-1^4)$};
 \node (Q2) at (-2,4) {$\widetilde{\mathcal{Q}}(2,1,-1^3)$};
 \node (Q3) at (2,4) { $\widetilde{\mathcal{Q}}(4,-1^4)$};
 \node (Q4) at (0,6)  { $\widetilde{\mathcal{Q}}(3,-1^3)$};
 \node (D1) at (0,0.5) { $\widetilde{\cH}(1,1)$};
 \node (D2h) at (-7,4) {$\dcoverhyp=\tilde{\cQ}(1^2,-1^2)$};
 \node (D2o) at (5.5,4) {$\dcoverodd$};
 \node (Sub1) at (3.8,4) {$\supset$};
 \node (Sub2) at (0,1.3) {$\cup$};

\path (Q4)  edge[->, >=angle 60] (Q2)
         (Q4)  edge[->, >=angle 60] (Q3)
         (Q2) edge[->, >=angle 60] (Q1)
         (Q3) edge[->, >=angle 60] (Q1)
         (D2h) edge[->, >=angle 60] (Q2)
         (D2h) edge[->, >=angle 60] (D1)
         %(D2o) edge[left hook->, >=angle 60] (Q3)
         (D2o) edge[->, >=angle 60] (D1);
         %(D1) edge[left hook->, >=angle 60] (Q1);
  \end{tikzpicture}
\caption{Rank two affine submanifolds of $\cH_3$: $X \ra Y$ means that $X \subset \partial Y$, and $X$ has codimension $1$  in $\ol{Y}$.}
\label{fig:rk2:mnfd:gen3}
\end{figure}
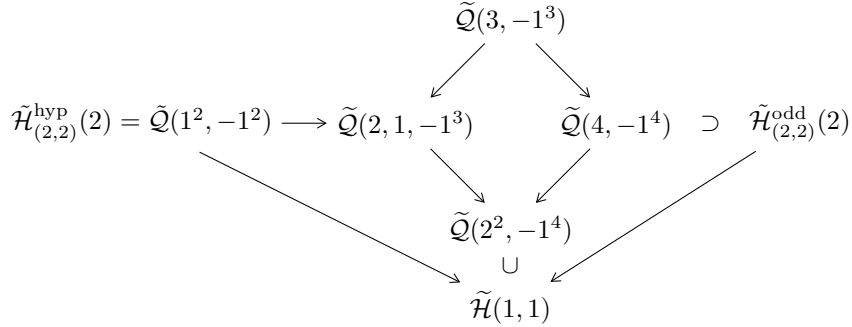

We close this section by indicating how Theorem~\ref{thm:main:rk2:g3} and \cite[Th. 1.1]{MirzakhaniWrightFullRank} imply Theorem~\ref{thm:classify:orbits:in:g3}. Let $\cM$ be the closure of $\GL^+(2,\R)\cdot M$ in $\cH(\kappa)$. If $\cM$ is of rank three (that is of full rank), then by \cite[Th. 1.1]{MirzakhaniWrightFullRank},  $\cM$ is  a component of  $\cH(\kappa)$ or a component of $\cH(\kappa)\cap\cL$.
If $\cM$ is of rank two, then by Theorem~\ref{thm:main:rk2:g3}, $\cM$ is a component of $\cH(\kappa)\cap\cP$ or a component of $\cH(\kappa)\cap\cP\cap\cL$. 
Finally, if $\cM$ is of rank one, then $M$ must be completely periodic by\cite[Th. 1.5]{WrightCylDef}, and the ratio of the circumferences of any pair of parallel cylinders belongs to a finite set by~\cite[Th. 1.4]{MirzakhaniWrightBoundary}.

\subsection{Outline}

The paper is organized as follows: in Section~\ref{sec:preliminaries} we recall essential definitions and important results needed for our proofs. Our strategy is to degenerate surfaces in a given rank two affine manifold $\cM\subset \cH(2,1^2)\cup\cH(1^4)$ by collapsing a family of $\cM$-parallel cylinders, to get surfaces in another rank two affine manifold $\cM'$ contained in some lower stratum.
%The list of possible $\cM'$ is known by the results in  \cite{NguyenWright,AulicinoNguyenWright,AulicinoNguyenGen3TwoZeros}.
The key point is that in some situations, we have $\dim \cM=\dim\cM'+1$ (see Propositions~\ref{prop:collapse:free:sim:cyl} and \ref{prop:collapse:similar:cyl}).
Moreover, we can derive some important properties of surfaces in $\cM$, namely the existence of involutory automorphisms, from the properties of surfaces in $\cM'$ (see Proposition~\ref{DblCovExtSimpCyl}).
We will also prove that the intersection $\cP\cap \cL$ in $\cH_3$ is precisely the locus of unramified double covers of translation surfaces of genus two (see Lemma~\ref{lm:2inv:dblcover}). In~\cite{AulicinoNguyenGen3TwoZeros}, we showed that $\cH(2)$ gives rise to two loci of unramified double covers in $\cH_3$, namely $\dcoverodd$ and $\dcoverhyp$. Interestingly, we will show that the locus of unramified double covers of surfaces in $\cH(1,1)$ is connected (see Proposition~\ref{H11CoverConnProp}). This follows from the fact that the mapping class group acts transitively on the set of non-zero cohomologies with coefficients in $\Z/(2\Z)$.

Section~\ref{sec:C3I:collapse} implements our strategy in a special situation, where $\cM$ contains a horizontally periodic surface with three horizontal cylinders, whose core curves span a Lagrangian in homology.

In Section~\ref{sec:get4cyl}, we show that $\cM$ must contain a horizontally periodic surface with at least four cylinders. For this, we  improve some technical lemmas in~\cite{AulicinoNguyenGen3TwoZeros} and use the results of \cite{NguyenWright,AulicinoNguyenWright, AulicinoNguyenGen3TwoZeros}.

Section~\ref{sec:4cyl} addresses the case in which $\cM$ contains a horizontally periodic surface with four cylinders. This case turns out to be the most involved in our analysis due to the various situations that may occur. Our main result in this section is Proposition~\ref{prop:4cyl}. For the proof, we split this case into four subcases following the topological type of the cylinder decomposition (see Lemma~\ref{lm:4CylDeg}), and each subcase is handled differently. In order to keep the focus on the main ideas of the proofs, we defer some technical lemmas to the appendix.

In Section~\ref{sec:5cyl}, we address the case in which $\cM$ contains a horizontally periodic surface with five cylinders. Employing essentially the strategy of collapsing, we come to the conclusion that if $\cM \subset \cH(2,1^2)$, then $\cM=\prymthreezero$, and if $\cM \subset \cH(1^4)$, then either $\cM=\prymprinc$, or $\cM$ contains a horizontally periodic surface with six cylinders (see Propositions~\ref{Case5IH211}, \ref{Case5IH1111}, \ref{prop:C5II}). This allows us to conclude the first part of Theorem~\ref{thm:rk2:H211:H1111}.

Finally, in Section~\ref{sec:6cyl} we consider the case in which $\cM$ contains a horizontally periodic surface with six cylinders.  Necessarily $\cM \subset \cH(1^4)$. By some elementary combinatorial arguments, we see that in this case there are only four possible cylinder diagrams (see Proposition~\ref{prop:6CylDiags}). Each cylinder diagram will be handled independently to show that either $\cM=\prymprinc$ or $\cM=\dcoverprinc$. This allows us to complete the proof of Theorem~\ref{thm:rk2:H211:H1111}.

\medskip

\noindent \textbf{Acknowledgements:} The authors warmly thank Alex Wright for helpful discussions and for suggesting the formulation of Theorem~\ref{thm:classify:orbits:in:g3}. They are also grateful to the Centre International de Rencontres  Math\'ematiques in Marseille for its hospitality and to Vincent Delecroix for providing the list of cylinder diagrams that inspired this work.

\section{Preliminaries}\label{sec:preliminaries}

We give a brief summary of the essential definitions and important results needed for this paper.  Since this paper is very much a sequel to \cite{AulicinoNguyenGen3TwoZeros}, all of the notation is consistent between the two papers, and we encourage the reader to refer to \cite[Sect. 2]{AulicinoNguyenGen3TwoZeros} for more detailed definitions and background.

\

\noindent \textbf{Strata and Their Structure}: A \emph{translation surface} $M = (X, \omega)$ is a pair of a Riemann surface of genus $g \geq 2$ carrying a non-zero Abelian differential $\omega$.  The set $\cH(\kappa)$ is the moduli space of translations surface where $\kappa$ specifies the orders of the zeros of the differential.  Strata admit an action by $\text{GL}_2(\bR)$ given by multiplying the real and imaginary foliations of $\omega$ by elements of the group.  There is a natural local system of coordinates on $\cH(\kappa)$ given by integrating $\omega$ over a basis of $H_1(X, \Sigma, \bZ)$, where $\Sigma \subset X$ is the set of zeros of $\omega$.  These are called \emph{period coordinates}.

\

\noindent \textbf{Orbit Closures and Their Structure}:  It was proven in \cite{EskinMirzakhaniMohammadiOrbitClosures}, that the $\text{GL}_2(\bR)$ orbit closure of a translation surface is an (immersed) affine manifold $\cM$ (after passing to a suitable finite cover) and that locally $\cM$ is a linear subspace of $H^1(X, \Sigma, \bC)$ in period coordinates.  The \emph{field of (affine) definition}, denoted by $\bk(\cM)$, is the smallest subfield of $\bR$ containing the coefficients of the linear equations defining $\cM$.  It is shown in \cite{WrightFieldofDef} that this field is of degree at most $g$ over $\Q$, where $g$ is the genus of a surface in $\cM$.

The \emph{rank of an affine manifold} $\cM$ is half the dimension of $\cM$ after applying the projection $H^1(X, \Sigma) \rightarrow H^1(X)$. We denote this invariant by ${\rm rk}(\cM)$.

\begin{theorem}[\cite{WrightFieldofDef}]\label{thm:Wright:def:field:deg:rk}
We have
$$
{\rm rk}(\cM)\,{\rm deg}_\Q\bk(\cM) \leq g.
$$
In particular, if $\cM$ is a rank two affine submanifold in $\cH_3$, then $\bk(\cM)=\Q$.
\end{theorem}

\begin{remark}\label{rk:sq:tiled:dense}
If $\bk(\cM)=\Q$, then the subset of square-tiled surfaces is dense in $\cM$.
\end{remark}

\noindent \textbf{Flat Structure}: A \emph{cylinder} on a translation surface is a maximal set of closed trajectories on $M$ that are pairwise homotopic and do not pass through singularities.  A \emph{saddle connection} is a flat trajectory that emanates from a zero and terminates at a not necessarily distinct zero.  A cylinder is \emph{simple} if each of its boundaries consist of exactly one saddle connection, and it is \emph{semi-simple} if at least one of its boundaries consists of exactly one saddle connection.

\

\noindent \textbf{Cylinder Decompositions:}  We say that a translation surface $M$ is {\em periodic in a direction $\theta \in \mathbb{RP}^1$}, if every geodesic in this direction is either  periodic, or  a saddle connection. Equivalently, $M$  decomposes into a union of open cylinders and saddle connections in this direction. Therefore, we also say that $M$ admits a cylinder decomposition in direction $\theta$. It follows from a result of Smillie-Weiss~\cite{SmillieWeissMinSets} that every $\GL^+_2(\R)$-orbit closure contains a horizontally  periodic surface.

\

\noindent \textbf{Cylinder Deformations}: If two parallel cylinders on $M$ remain parallel on all translation surfaces in a local neighborhood of $M \in \cM$, then we say that they are $\cM$-parallel.  A cylinder is called \emph{free} if it does not share this property with any other cylinder on $M$.  The relation of being $\cM$-parallel is an equivalence relation.

Let $\cC=\{C_1,\dots,C_k\}$ be a family of horizontal cylinders on  a surface $M \in \cM$. For any $t,s \in \R$, let $a_s=\left(\begin{smallmatrix} 1 & 0 \\ 0 & e^s \end{smallmatrix}\right)$, and $u_t:=\left(\begin{smallmatrix} 1 & t \\ 0 & 1 \end{smallmatrix}\right)$.  We denote by $a_s^\cC(M)$ (resp. $u_t^\cC(M)$) the surface obtained by applying $a_s$ (resp. $u_t$) to every cylinder in $\cC$, while the rest of $M$ remains unchanged.  Applying $a_s^\cC$ is called  {\em stretching}, and applying $u_t^\cC$ is called {\em shearing} the cylinders in $\cC$.

\begin{theorem}[\cite{WrightCylDef}, Thm. 5.1]
\label{thm:Wright:Cyl:Def}
Let $\cM$ be an affine manifold.    If $\cC$ is an equivalence class of $\cM$-parallel horizontal cylinders on $M \in \cM$, then for all $s,t \in \mathbb{R}$, $a_s^{\cC}(u_t^{\cC}(M)) \in \cM$.
\end{theorem}

\noindent \textbf{Twist and Preserving Space}: Let $M$ be a horizontally periodic translation surface in an affine manifold $\cM$.  The \emph{cylinder preserving space} $\Pres(M, \cM)$ is the largest subspace of the real tangent space to $\cM$ at $M$ whose elements evaluate to zero on all core curves of the horizontal cylinders of $M$.  The \emph{twist space} $\Twist(M, \cM) \subseteq \Pres(M, \cM)$ consists of all elements that evaluate to zero on all horizontal saddle connections of $M$.  The following definition is motivated by the lemma below.  In the case of rank one affine manifolds it aligns with the definition of $\cM$-stably periodic from \cite{LanneauNguyenWrightFinInNonArithRkOne}.  See \cite[Rmk. 2.8]{LanneauNguyenWrightFinInNonArithRkOne}.

\begin{definition}
Given a horizontally periodic translation surface $M \in \cM$, we say $M$ is \emph{$\cM$-cylindrically stable} if $\Twist(M, \cM) = \Pres(M, \cM)$.
\end{definition}

\begin{lemma}[\cite{WrightCylDef}, Lem. 8.6]
\label{lm:WrightTwistPresLem}
Let $M$ be a horizontally periodic translation surface in an affine manifold $\cM$.  If $M$ is \emph{not} $\cM$-cylindrically stable, then there exists a horizontally periodic translation surface in $\cM$ with more horizontal cylinders than $M$.
\end{lemma}

\noindent \textbf{Cylinder Proportions}:  Let $\cC$ be an equivalence class of $\cM$-parallel cylinders on a translation surface $M \in \cM$.  Let $X \subset M$ be any cylinder in another direction on $M$.  The \emph{cylinder proportion} of $C$ in $\cC$ is given by
$$P(X, \cC) = \frac{\text{Area}(X \cap (\cup_{C \in \cC} C))}{\text{Area}(X)}.$$

\begin{proposition}[Cylinder Proportion Lemma~\cite{NguyenWright}]
\label{CylinderPropProp}
Let $X$ and $Y$ be $\cM$-parallel cylinders on a translation surface $M \in \cM$.  Let $\cC$ be an equivalence class of $\cM$-parallel cylinders on $M$.  Then $P(X, \cC) = P(Y, \cC)$.
\end{proposition}

\noindent {\bf Cylinder collapsing:} We recall that by ``collapsing'' a cylinder we mean deforming the translation surface by decreasing the height of the cylinder to zero while keeping the rest of the surface unchanged. For a more precise description of this operation, we refer to \cite[Sect. 2.4]{AulicinoNguyenGen3TwoZeros}. We first notice

\begin{lemma}[\cite{AulicinoZeroExpGen3}, Lem. 5.4]
\label{lm:s:cyl:dist:sim:zeros}
Let $C$ be a simple cylinder on a translation surface $M$.  If the zeros (of the holomorphic $1$-form) contained in the boundary of $C$ are simple, then they must be distinct.
\end{lemma}

The following proposition can be proven without too much effort using the results from \cite{AulicinoNguyenGen3TwoZeros}, and in particular, Proposition 2.16 contained therein.  However, it is much quicker to use the more general and developed machinery of \cite{MirzakhaniWrightBoundary}.  Since all degenerations in this paper will occur over a compact subset of the moduli space of Riemann surfaces of fixed genus, the results will \emph{not} rely on the multicomponent EMM conjecture.

\begin{proposition}
\label{prop:collapse:free:sim:cyl}
Let $\cM$ be an affine manifold, and let $M \in \cM$.   Suppose  that $M$ has a free simple cylinder $C$ with two distinct zeros on its boundary components. Let $M'$ be the surface obtained by collapsing $C$ so that the two zeros collide, and let $\cH(\kappa')$ where $|\kappa'|=|\kappa|-1$ be the stratum of $M'$.
 Then $M' $ is contained in an affine submanifold  $\cM' \subset \cH(\kappa')$ such that ${\rm rk}(\cM')={\rm rk}(\cM)$ and  $\dim\cM'=\dim\cM-1$.

 Moreover, let $U$ be a neighborhood of $M$ in $\cM$ such that for any surface in $U$, $C$ persists and remains simple. Let $\varphi: U \ra \cH(\kappa')$ be the map consisting of collapsing $C$ such that the two zeros in its boundary are identified. Then $\varphi(U)$ is a neighborhood of $M'$ in $\cM'$.
\end{proposition}

\begin{proof}
Let $\sig$ be the saddle connection in $C$ that is reduced to a point in $M'$, and $V=\C\cdot\sig \subset H_1(X,\Sigma,\C)$.  By \cite[Thm. 2.7]{MirzakhaniWrightBoundary}, the tangent space $T_{M'}(\cM')$ is isomorphic to $T_M(\cM) \cap \text{Ann}(V)$, where $\allowbreak {\rm Ann}(V)= \{\xi \in H^1(X,\Sig,\C)\ | \ \xi(\sig)=0\}$.  By assumption, there is a single saddle connection that vanishes at the boundary, so $\text{Ann}(V)$ has codimension one in $T_M(\cM)$.  It follows that $\dim\cM'=\dim\cM-1$. The claim about the equality of the ranks follows from \cite[Prop. 2.16]{AulicinoNguyenGen3TwoZeros}.

For the final claim, it is enough to remark that in some appropriate period coordinates of  $\cH(\kappa)$ and $\cH(\kappa')$, $\varphi$ is just the projection from $T_M(\cM)$ onto $T_M(\cM)\cap{\rm Ann}(V)$.
\end{proof}

\noindent {\bf Similar cylinders:} Let $C_1$ and $C_2$ be two simple cylinders in $M$. Recall that $C_i, \,i=1,2$, is the quotient of an infinite horizontal strip $\tilde{C}_i:=\R\times[0,h_i]$ by a $\Z$-action generated by $(x,y) \mapsto (x+\ell_i,y)$, where $h_i$ and $\ell_i$ are respectively the height and the circumference of $C_i$. Note that the lines $\R\times\{0\}$ and $\R\times\{h_i\}$ are mapped to the boundary components of $C_i$. We can always assume that $(0,0)$ is mapped to the zero in a boundary component of $C_i$. The inverse image of the zero in the other component is given by $(a_i,h_i)+\Z(\ell_i,0)$.

We will call a parallelogram in $\tilde{C}_i$ whose set of vertices is $\{(0,0), (\ell_i,0), (a_i+m\ell_i,h_i), (a_i+(m+1)\ell_i,h_i)\}, \, m\in \Z$, a {\em normalized fundamental domain} of $C_i$. We will say that $C_1$ and $C_2$ are {\em similar} or {\em proportional}, if there exist two normalized fundamental domains $P_1, P_2$ of $C_1, C_2$ respectively such that $P_2$ is the image of $P_1$ by a homothety $z \mapsto \lambda z$, with $\lambda > 0$.  In particular, if $\lambda=1$, then $C_1$ and $C_2$ are said to be \emph{isometric}.  This agrees with the definition of isometric in \cite{AulicinoNguyenWright, AulicinoNguyenGen3TwoZeros}.

Since $C_1$ and $C_2$ are simple, they persist and remain simple on every surface in a sufficiently small neighborhood of $M$ in its stratum. The cylinders $C_1$ and $C_2$ are said to be {\em $\cM$-similar}, if they are $\cM$-parallel and remain similar on every surface in a neighborhood of $M \in \cM$. Note that in this case, there exists a constant $\lambda$ and normalized fundamental domains such that for any surface $M'$ in a neighborhood of $M$ in $\cM$, we have $P^{M'}_1=\lambda P^{M'}_2$, where $P^{M'}_i$ is a normalized fundamental domain of the cylinder corresponding to  $C_i$ in $M'$. By a slight abuse of notation, we will write $C_1=\lambda C_2$.

\begin{proposition}\label{prop:collapse:similar:cyl}
Assume that $\{C_1,C_2\}$ is an equivalence class of $\cM$-similar simple cylinders on $M$. Assume that the boundary of $C_1$ (resp. $C_2$) contains two distinct zeros, and the pairs of zeros contained in $\partial C_1$ and $\partial C_2$ are not the same.
Let $M'$ be the  surface obtained by twisting and collapsing  $C_1, C_2$ simultaneously such that the zeros in the boundary of $C_1$ (resp. $C_2$) collide.   Then $M'$ is contained in an affine submanifold $\cM'$ of a stratum $\cH(\kappa')$, where $|\kappa'|=|\kappa|-2$, such that $\dim \cM' = \dim \cM-1$, and $\mathrm{rk}(\cM')=\mathrm{rk}(\cM)$.

Moreover, let $U$ be a neighborhood of $M \in \cM$ such that for any surface in $U$, $C_1$ and $C_2$ persist and remain simple. Let $\varphi: U \ra \cH(\kappa')$ be the map consisting of  collapsing  $C_1$ and $C_2$ such that the two zeros in the boundary of each cylinder are identified. Then $\varphi(U)$ is a neighborhood of $M'$ in $\cM'$.
\end{proposition}

\begin{proof}
For $i = 1,2$, let $\sigma_i$ be the unique saddle connection in $C_i$ that is collapsed to a point under the degeneration in the assumption of the proposition, and $\allowbreak V:=\C\cdot\sig_1\oplus\C\cdot\sig_2 \subset H_1(X,\Sigma,\C)$. By \cite{MirzakhaniWrightBoundary}, we have $T_{M'}(\cM') \simeq  T_M(\cM) \cap \text{Ann}(V)$.
By definition,
$$T_M(\cM)\cap{\rm Ann}(V)=\{\xi \in T_M(\cM) \ | \ \xi(\sig_1)=\xi(\sig_2)=0\}.$$
But by the similarity assumption, there exists a constant $\lambda \in \R_{>0}$ such that  for every $\xi \in T_M(\cM), \xi(\sig_1)=\lambda\xi(\sig_2)$. Thus $\allowbreak T_M(\cM)\cap{\rm Ann}(V)=\{\xi \in T_M(\cM) \ | \ \xi(\sig_1)=0\}$. It follows that $\dim T_{M'}(\cM')=\dim T_M(\cM)-1$. The final claim follows as in Proposition \ref{prop:collapse:free:sim:cyl}.
% Then $\text{Ann}(V)$ is supported on $M \setminus \{C_1, C_2\}$.  the codimension of $\text{Ann}(V)$ is given by the dimension of the deformation space of $\{C_1, C_2\}$.  By \cite[Thm. 4.7]
%{MirzakhaniWrightBoundary}, $\text{codim}_{\bR} \text{Ann}(V) = 1 + \dim W$, where one dimension is given by the standard twist, and $W$ represents the space of non-standard twists.  Since $C_1$
%and $C_2$ are $\cM$-similar, $\dim W = 0$ and after complexifying the proposition follows.
\end{proof}

The following proposition shows that under some assumptions, an involution on the surface obtained from a cylinder collapsing does extend to an involution on the original surface with the same number of fixed points.

\begin{proposition}
\label{DblCovExtSimpCyl}
Let $M$ be a translation surface and $\cC=\{C_1,\dots,C_k\}$ a family of pairwise $\cM$-similar simple cylinders on $M$ in the horizontal direction. We assume that in each $C_i$ there exists a vertical saddle connection $\del_i$ joining the singularities in its boundary, and the graph $\Gr:=\bigcup_{i=1}^k\del_i$  contains no loops (that is $\Gr$ is a disjoint union of topological trees).

Collapse the cylinders in $\cC$ simultaneously so that the saddle connections $\del_i$ are all reduced to points, and let $M'$ be the resulting surface. Let $\Pi$ be the union of distinguished saddle connections resulting from the degeneration of $\cC$ on $M'$. If  $M'$ admits an involution whose derivative is $-\id$ that preserves  $\Pi$, then this involution of $M'$ extends to an involution of $M$ which has the same number of fixed points.

In particular, if $M'$ is contained in the Prym locus or in the hyperelliptic locus  of $\cH_3$, then so is $M$.
\end{proposition}
\begin{proof}
We first notice that $M'$ is a surface of the same genus as $M$. To see this, remark that the collapsing of a simple cylinder with distinct zeros (singularities) in its boundary does not change the topology of the surface. Using the assumption that $\Gr$ contains no loops, by induction, we derive that $M'$ has the same genus as $M$.

Let $\inv'$ be the involution of $M'$. Consider the case $k=1$, that is $\cC$ consists of a single simple cylinder $C$. In this case $\Pi$ is a saddle connection $\sig$ joining a zero $x'_0$ of $M'$ to itself where $x'_0$ is the collision of two zeros in $M$.

By assumption $\inv'$ preserves $\sig$, hence  $\sig$ contains two fixed points of $\inv'$, one of which is $x'_0$ the other one is the midpoint of $\sig$.
By construction, $M'\setminus\sig$ is identified with $M\setminus\ol{C}$. Since $\inv'$ maps $M'\setminus \sig$ to itself, we can consider $\inv'$ as an involution of $M\setminus\ol{C}$. Note that every simple cylinder admits an involution that exchanges its boundary components and fixes two points in its interior. Therefore, the involution $\inv'$  extends to an involution $\inv$ of $M$ which fixes the cylinder $C$. Clearly, $\inv'$ and $\inv$ have the same number of fixed points.

For the general case, let $\sig_i$ denote the degeneration of $C_i$ on $M'$. Since $\Pi$ is preserved by $\inv'$, each $\sig_i$ is either invariant or permuted with another $\sig_j$.
Let $I \subset \{1,\dots,k\}$ be the subset of indices defined by the condition: $i\in I$ if and only if  $\sig_i$ is invariant by $\inv'$.

Let $\tilde{M}'$ be the surface obtained by reinserting the family of cylinders $\{C_i, | i \in I\}$ to $M'$.  By the argument of the previous case, we conclude that $\inv'$ extends to an involution $\tilde{\inv}'$ of $\tilde{M}'$ with the same number of fixed points.
By construction, the family $\{\sig_i,  | i \in I^c\}$ persists on $\tilde{M}'$, and any saddle connection in this family is exchanged with another one by $\tilde{\inv}'$. Since $\tilde{\inv}'$ is an isometry for the flat metric, if $\sig_i$ and $\sig_j$ are exchanged, then they have the same length. By the assumption of similarity, this means that $C_i$ and $C_j$ are isometric. Thus $\tilde{\inv}'$ extends to an involution  $\inv$ of $M$ that exchanges $C_i$ and $C_j$. Clearly, $C_i$ and $C_j$ do not contain any fixed point of $\inv$ in their interior. Thus $\inv$ and $\tilde{\inv}'$ have the same number of fixed points as do $\tau$ and $\tau'$.
\end{proof}

\begin{remark}
 If $M$ is a genus three Riemann surface, an involution of $M$ is hyperelliptic if and only if it has $8$ fixed points.
\end{remark}

\medskip
%****************************************************************************
\noindent {\bf Topological type of cylinder decompositions:} Let $M$ be horizontally periodic, and let $\cC$ be the family of all horizontal cylinders of $M$.  By {\em topological type} of the cylinder decomposition of $M$, we will  mean the topological surface underlying the stable holomorphic $1$-form that is the limit $a_t^{\cC}(M)$ as $t\rightarrow +\infty$. Equivalently, this is also the surface one obtains after ``pinching'' all of the core curves of the horizontal cylinders. Note that all of the topological types of cylinder decompositions with three or four cylinders of surfaces in genus three are given in \cite[Lem. 3.1]{AulicinoNguyenGen3TwoZeros} and \cite[Lem. 6.1]{AulicinoNguyenGen3TwoZeros}. The topological types of the $5$-cylinder diagrams are given in Lemma~\ref{5CylDeg}.

%***********************************************************************************************
\begin{proposition}\label{prop:many:cyl:rktwo}
Let $\cM$ be one of the following loci
$$
\{ \tilde{\cQ}(3,-1^3), \tilde{\cQ}(1^2,-1^2), \tilde{\cQ}(4,-1^4), \allowbreak
\tilde{\cH}^{\rm odd}_{(2,2)}(2), \tilde{\cQ}(2,1,-1^3)\}.
$$

\begin{itemize}
 \item[(a)] If $\cM=\tilde{\cQ}(3,-1^3)$, then there exists a surface admitting a cylinder decomposition with three cylinders of topological type given by Case 3.I).

 \item[(b1)] If $\cM\in \{\cQ(1^2,-1^2)\simeq \tilde{\cH}^{\rm hyp}_{(2,2)}(2)\}$, then there exists a surface admitting a cylinder decomposition with four cylinders of topological type given by Case 4.I).

 \item[(b2)] If $\cM \in \{\tilde{\cH}^{\rm odd}_{(2,2)}(2), \tilde{\cQ}(4,-1^4)\}$, then there exists a surface admitting a cylinder decomposition with four cylinders of topological type given by Case 4.II).

  \item[(c)] If $\cM=\tilde{\cQ}(2,1,-1^3)$, then there exists a surface admitting a cylinder decomposition with five cylinders of topological type given by Case 5.I).
\end{itemize}
\end{proposition}

\begin{proof}
Claims (a), (b1), and (b2) follow from \cite[Fig. 7.1]{AulicinoNguyenWright}, \cite[Fig. 18]{AulicinoNguyenGen3TwoZeros}, and \cite[Fig. 21]{AulicinoNguyenGen3TwoZeros}, respectively.  For Claim (c), see Figure~\ref{fig:C5II:H211:cyl:diagram}.
\end{proof}
%*********************************************************************************************

\medskip

We now show
\begin{lemma}
\label{kCylsInBdLem}
Let $M$ be a surface in a rank $k$ affine manifold $\cM$ such that $M$ contains a free simple cylinder $C$ with distinct zeros on its boundary.  Let $M'$ be the surface obtained from $M$ by collapsing $C$ so that the two zeros in its boundary collide.  Then $M'$ is contained in an affine manifold $\cM'$ in the same genus such that $\text{rank}(\cM') = \text{rank}(\cM) = k$, and $\dim \cM' = \dim \cM - 1$.  Moreover, if $\cM'$ contains a dense subset $S$ such that every surface in $S$ admits a cylinder decomposition of the same topological type, then $\cM$ also contains a surface admitting a cylinder decomposition of this topological type.
\end{lemma}

\begin{proof}
The first claims concerning the rank and dimension follow from \cite[Thm. 2.7]{MirzakhaniWrightBoundary} or \cite[Prop. 2.16]{AulicinoNguyenGen3TwoZeros}.

Next, we claim that there is an open neighborhood $W$ of $M' \in \cM'$ such that every surface in $W$ is obtained from a surface in $\cM$ by collapsing a free simple cylinder.  To see this we observe that the tangent space $T_{M'}(\cM')$ is isomorphic to $\text{Ann}(V) \cap T_M(\cM)$ by \cite{MirzakhaniWrightBoundary}, where $V$ is the vanishing space, which in this case is generated by the saddle connection $\sigma \subset C$ that collapses to a point.  Since each deformation of $M$ that fixes $C$ corresponds to a deformation of $M'$ and vice versa, we see that $W$ has positive measure in $\cM'$.

Let $M_1' \in S \cap W$.  Let $c_1$ be the saddle connection on $M_1'$ that is the degeneration of a simple cylinder in a surface $M_1 \in \cM$.  Observe $c_1$ is a saddle connection from a zero to itself.  By assumption, $M_1'$ admits a cylinder decomposition of the given topological type in some direction $\theta$.  We split the remainder of the argument into two cases.

First, assume that $c_1$ does \emph{not} lie in direction $\theta$.  Cut $M_1'$ along the saddle connection $c_1$, and insert a simple cylinder $C$.  Since every saddle connection between the two zeros in the boundaries of $C$ differs by a Dehn twist, it suffices to choose a shortest and denote it by $\sigma$.  The foliation of $M_1'$ in direction $\theta$ naturally extends into a neighborhood of the boundary of $C$.  If necessary, twist $C$ so that $\sigma$ lies in direction $\theta$.  This can be accomplished because the directions of $\sigma$ and $c_1$ are transverse.
We claim that this surface, which we call $M_1 \in \cM$ has the same topological type as $M_1'$.  First we observe that $M_1$ is periodic in direction $\theta$ by construction.  Indeed, this construction added a line segment of length equal to that of $\sigma$ to every leaf of the foliation passing through $c_1$.  Secondly, if we consider the homotopy classes of the core curves of the cylinders in direction $\theta$, then these are preserved on $M$.
This follows from the observation that there is no leaf of the foliation in direction $\theta$ that passes through one zero on the boundary of $\sigma$ without passing through the zero on the other boundary of $\sigma$.\footnote{We remark that the intersection of the homology class of any closed leaf of the foliation in direction $\theta$ on $M_1$ with the relative homology class represented by $\sigma$ is zero.  This is a defining property of the construction we used to add a simple cylinder to $M_1'$.}

Second, assume that $c_1$ does lie in direction $\theta$.  In this case, we once again cut $M_1'$ along $c_1$, glue in a simple cylinder $C$, twist $C$ if necessary so that it does not admit a vertical saddle connection, and collapse the cylinder.  The resulting surface will have the same topological type as $M_1'$ for the same reason as above.
\end{proof}

\begin{definition}
Let $\mathfrak{c}$ be a cylinder diagram.  We say that an affine manifold $\cM$ \emph{admits a cylinder diagram} $\mathfrak{c}$ if there exists a periodic translation surface $M \in \cM$ such that $M$ has cylinder diagram $\mathfrak{c}$.
\end{definition}

\begin{proposition}
\label{CylStDens}
If an affine manifold $\cM$ admits a cylinder diagram $\mathfrak{c}$, then there exists a dense subset $S \subset \cM$ of periodic translation surfaces admitting $\mathfrak{c}$.
\end{proposition}

\begin{proof}
Let $M' \in \cM$ admit cylinder diagram $\mathfrak{c}$, and without loss of generality, assume that $M'$ admits $\mathfrak{c}$ in the horizontal direction.  Let $P$ denote the subgroup of upper triangular matrices in \splin .  By \cite[Thm. 2.1]{EskinMirzakhaniMohammadiOrbitClosures}, $\overline{P \cdot M'} = \overline{\text{SL}_2(\bR) \cdot M'}$.  Observe that every translation surface in $P \cdot M'$ is horizontally periodic admitting cylinder diagram $\mathfrak{c}$.  Hence, it suffices to produce $M \in \cM$ such that $M$ admits cylinder diagram $\mathfrak{c}$ and $\overline{\text{SL}_2(\bR) \cdot M} = \cM$.

To produce such an $M$, consider $T_{M'}^{\bR}(\cM)$.  All deformations in this space preserve all horizontal saddle connections, whence they preserve $\mathfrak{c}$.  Since there are at most countably many affine manifolds in $\cM$, there exists a real tangent vector $v \in T_{M'}^{\bR}(\cM)$ such that $\overline{\text{SL}_2(\bR) \cdot (M' + v)} = \cM$.  Let $M = M' + v$.
\end{proof}

\subsection{Unramified Double Covers}
\begin{lemma}\label{lm:2inv:dblcover}
Let  $M=(X,\omega)$ be a translation surface in genus three. The surface $M$ admits a Prym involution and a hyperelliptic involution if and only if there exists a translation surface $M'=(X',\omega')$ in genus two, and an unramified double cover $p: X \ra X'$ such that $p^*\omega'=\omega$.
\end{lemma}
\begin{proof}
First assume $X$ admits a Prym involution $\inv$ and a hyperelliptic involution $\iota$. Since $\iota$ commutes with all automorphisms of $X$, $\rho=\inv\circ\iota$ is also an involution of $X$ which satisfies $\rho^*\omega=\omega$. Let $X':=X/\langle \rho \rangle$ be the quotient of $X$ by $\rho$.

 For any involution $f$ of $X$, let
$$
\Omega^+(X,f):=\{\xi \in \Omega(X)| f^*\xi=\xi\}, \text{ and } \Omega^-(X,f):=\{\xi \in \Omega(X)| f^*\xi=-\xi\}.
$$
By definition, $\dim \Omega(X')=\dim \Omega^+(X,\rho)$.  Since  $\iota$ acts by $-\id$ on $\Omega(X)$, we have $\dim \Omega^+(X,\rho)=\dim \Omega^-(X,\inv)=2$. Thus $X'$ is a surface of genus two. The Riemann-Hurwitz formula then implies that the  double cover $p :X \ra X'$ is unramified. Since $\omega \in \Omega^+(X,\rho)$, there exists a holomorphic $1$-form $\omega'$ on $X'$ such that $\omega=p^*\omega'$.

Conversely, if there exists an unramified double cover of translation surface $p: M \ra M'$, then   $M'$ must be a surface of genus two, and  $M$ admits an automorphism $\rho$ such that $p\circ \rho=p$ and $\rho^2=\id$. The automorphism $\rho$ is induced by any element of $\pi_1(M')$ that is not contained in $p_*\pi_1(M)$. Since $M'$ is of genus two, it has a hyperelliptic involution which lifts to a hyperelliptic involution $\iota$ of $M$. The composition $\iota\circ\rho$ is then a Prym involution. The details are left to the reader.
\end{proof}

The following proposition shows that the locus of unramified double covers of translation surfaces of genus two in $\principal$ is connected,  thus it consists of a single  rank two affine submanifold of $\principal$. Note that this locus is also the intersection $\allowbreak \cP\cap\cL\cap\principal$  by Lemma~\ref{lm:2inv:dblcover}.
\begin{proposition}
\label{H11CoverConnProp}
The  locus  $\dcoverprinc$ of pairs $(X,\omega) \in \principal$ such that there exist a pair $(X',\omega') \in \cH(1,1)$ and an unramified double cover $\pi :X \ra X'$ satisfying $\pi^*\omega'=\omega$ is connected.
\end{proposition}

\begin{proof}
Let $\Mod_g$ be the moduli space of Riemann surfaces of genus $g$. Let $\widetilde{\Mod}_2 \subset \Mod_3$ denote the locus of Riemann surfaces of genus three that are unramified double covers of some surface of genus two.  We first show that $\widetilde{\Mod}_2$ is connected.

Let us fix a topological closed surface of genus two  $S$. Assume that we have a topological covering of degree two $p: \hat{S} \ra S$.  We then have $\chi(\hat{S})=2\chi(S)=-4$, hence $\hat{S}$ must be a surface of genus three.

By definition $p_*(\pi_1\hat{S})$ is a subgroup of index two of $\pi_1S$.  Thus there exists a group homomorphism $\varepsilon: \pi_1S \ra \Z/(2\Z)$ such that $p_*(\pi_1\hat{S})=\ker \varepsilon$. Since $\Z/(2\Z)$ is abelian,  $\varepsilon$ can be written as $h\circ {\bf p}$, where ${\bf p}: \pi_1S \ra H_1(S,\Z)$ is the natural projection, and $h: H_1(S,\Z) \ra \Z/(2\Z)$ is a homomorphism of abelian groups. Note that we can consider $h$ as an element of $H^1(S,\Z/(2\Z)) \setminus \{0\}$.

Conversely, given an element $h\in H^1(S,\Z/(2\Z)) \setminus \{0\}$, then $\Gamma={\bf p}^{-1}(\ker h)$ is a (normal)  subgroup of index two in $\pi_1S$. Thus $p: \widetilde{S}/\Gamma \ra S$ is a (topological) double cover, where $\widetilde{S}$ is the universal cover of $S$. In particular, $\widetilde{S}/\Gamma$  is a closed surface of genus three.
From classical results on covering spaces, we know that if $p_1: \hat{S}_1 \ra S$ and $p_2: \hat{S}_2 \ra S$ are two double covers which correspond to the same element of $H^1(S,\Z/(2\Z))$, then $p_1$ and $p_2$ are isomorphic, that is there exists a homeomorphism $f : \hat{S}_1 \ra \hat{S}_2$ such that $p_1=p_2\circ f$.
Thus we have shown the following

\

\noindent \underline{\em Claim 1:} There is a bijection between the set of topological double covers of $S$ up to isomorphism and the set $H^1(X,\Z/(2\Z))\setminus\{0\}$.

\medskip

Let us now fix a topological double covering $p: \hat{S} \ra S$ and denote by $h$ the element of $H^1(X,\Z/(2\Z))$ associated to $p$.  Let $r_0: X_0 \ra X'_0$ and $r_1: X_1 \ra X'_1$ be two unramified double covers of (compact) Riemann surfaces, where $X_i$ is of genus three and $X'_i$ is of genus two. Our goal is to show that there is a path in $\widetilde{\Mod}_2$ from $X_0$ to $X_1$. We first show

\

\noindent \ul{\em Claim 2:} There are two homeomorphisms $\phi_i: S \ra X'_i, \, i=0,1,$ such that the topological covering  $\phi_i^{-1}\circ r_i$ is isomorphic to $p$.
\begin{proof}
It is enough to show the existence of $\phi_0$.  Let $f_0: S \ra X'_0$ be any homeomorphism and consider the double cover $p_0=f_0^{-1}\circ r_0: X_0 \ra S$.  Let $h_0$ be the element of $H^1(S,\Z/(2\Z))$ associated to $p_0$. Since the action of the Mapping Class Group $MCG(S)$ on $H^1(S,\Z/(2\Z))\setminus\{0\}$ is transitive (see \cite[Chap. 6]{FarbMargalitMCG}), there exists a homeomorphism  $\varphi_0: S \ra S$ such that $(\varphi_0^{-1})^*h_0=h$. Setting  $\phi_0:=f_0\circ\varphi_0: S \ra X'_0$, from Claim 1, we see that the covers $p$ and $\phi_0^{-1}\circ r_0$ are isomorphic.
\end{proof}

Since $p$ and $\phi_i^{-1}\circ r_i$ are isomorphic, there exists a homeomorphism $\widetilde{\phi}_i: \hat{S} \ra X_i$ that satisfies $p=(\phi_i^{-1}\circ r_i)\circ\widetilde{\phi}_i$, or equivalently $\phi_i\circ p = r_i \circ\widetilde{\phi}_i$. Remark that if we equip $S$ with the conformal structure of $X'_i$ via $\phi_i$, we then get an induced conformal structure on $\hat{S}$ and $\widetilde{\phi}_i: \hat{S} \ra X_i$ becomes an isomorphism of Riemann surfaces.

\begin{center}
 \begin{tikzpicture}[scale=0.3]
    \node (A) at (0,4) {$\hat{S}$};
    \node (B) at (6,4) {$X_i$};
    \node (C) at (0,0) {$S$};
    \node (D) at (6,0) {$X'_i$};

    \path[->, font=\scriptsize, >= angle 60, dashed] (A) edge node[above]{$\widetilde{\phi}_i$} (B);
    \path[->, font=\scriptsize, >= angle 60]
    (A) edge node[left]{$p$} (C)
    (C) edge node[above]{$\phi_i$} (D)
    (B) edge node[right]{$r_i$} (D);
 \end{tikzpicture}
\end{center}

We now notice that the pairs $(X'_i,\phi_i), \, i=0,1,$ represent two points in the Teichm\"uller space $\cT_2$. Since $\cT_2$ is connected, there exists a path $[X'_t,\phi_t], \, t \in [0,1]$ connecting those two points (here $X'_t$ is a Riemann surface of genus two, $\phi_i: S \ra X_t$ is a homeomorphism, and $[X'_t,\phi_t]$ is the equivalence  class of $(X'_t,\phi_t)$).  Since $\phi_t\circ p: \hat{S} \ra X'_t$ is a double cover, the conformal structure of $X'_t$ induces a conformal structure on $\hat{S}$. Let $X_t$ denote the corresponding Riemann surface. By construction $X_t$ is an unramified double cover of $X'_t$, which means that $X_t \in \widetilde{\Mod}_2$. Thus we have found a path in $\widetilde{\Mod}_2$ from $X_0$ to  $X_1$, which shows that $\widetilde{\Mod}_2$ is connected.

\medskip

Recall that the stratum $\cH(1,1)$ is a subset of the Abelian differential bundle  $\Omega\Mod_2$ over $\Mod_2$. Each fiber of $\Omega \Mod_2$ is the space of holomorphic $1$-forms on a Riemann surface $X$ of genus two, thus  can be identified with $\C^2$. The intersection of this fiber with $\cH(1,1)$ is the set of holomorphic $1$-forms on $X$ with two simple zeros.  Remark that the double zero of a holomorphic $1$-form on $X$ must be a Weierstrass point, and every genus two Riemann surface has exactly $6$ Weierstrass points. Therefore, $\cH(1,1)\cap \Omega(X)$ can be identified with $\C^2$ minus $6$ complex lines. Hence we can realize $\cH(1,1)$ as a bundle over $\Mod_2$ whose fibers are $\C^2$ minus $6$ complex lines.

\begin{center}
\begin{tikzpicture}[scale=0.4]

 \node (A) at (0,3) {$\dcoverprinc$};

 \node (B) at (6,3) {$\cH(1,1)$};

 \node (C) at (0,0) {$\widetilde{\Mod}_2$};

 \node (D) at (6,0) {$\Mod_2$};

 \path[->, font=\scriptsize, >=angle 90, dashed] (A) edge (B);
 \path[->, font=\scriptsize, >=angle 90]
 (A) edge (C)
 (C) edge (D)
 (B) edge (D);
\end{tikzpicture}
\end{center}

By definition, $\dcoverprinc$ is the pullback of this bundle to $\widetilde{\Mod}_2$. Since $\widetilde{\Mod}_2$ is connected and the fibers of this bundle are connected, we conclude that $\dcoverprinc$ is connected.
\end{proof}

\section{A Special Case of Cylinder Collapsing}\label{sec:C3I:collapse}
Throughout this section, $\cM$ will be a rank two affine submanifold of either $\cH(2,1^2)$ or $\cH(1^4)$. Using the tools provided in Section~\ref{sec:preliminaries}  and the classification of rank two affine submanifolds in the strata  $\cH(\kappa) \subset \cH_3$ where $|\kappa| \leq 2$, we will show that in a special case one can get immediately the desired conclusions about $\cM$. Recall that a cylinder decomposition in Case 3.I) means that the cylinder decomposition consists of three cylinders such that the three core curves span a Lagrangian in homology.

\begin{proposition}\label{prop:C3I:2sim:cyl:classify}
Assume that $\cM$ contains a horizontally periodic surface $M$ satisfying Case 3.I) such that two of the cylinders are simple and there are at least two equivalence classes of cylinders. Then
 \begin{itemize}
  \item[$(a)$] If $\cM \subset \cH(2,1^2)$, then $\cM=\tilde{\cQ}(2,1,-1^3)$.

  \item[$(b)$] If $\cM \subset \cH(1^4)$, then $\cM=\tilde{\cH}(1,1)$ or $\cM=\tilde{\cQ}(2^2, -1^4)$.
 \end{itemize}
\end{proposition}
\begin{remark}
It can be shown that if $M$ is a horizontally periodic satisfying Case 3.I) in a rank two affine manifold, then the horizontal cylinders must fall into two equivalence classes.
\end{remark}

\begin{proof}
Let $C_1,C_2,C_3$ denote the horizontal cylinders of $M$, where $C_1,C_2$ are simple. By \cite[Lem. 2.11]{AulicinoNguyenGen3TwoZeros}, we know that none of $C_1,C_2$ is $\cM$-parallel to $C_3$. By \cite[Lem. 2.15]{AulicinoNguyenGen3TwoZeros}, $C_1,C_2,C_3$ cannot all be free. Therefore, we can conclude that $C_1,C_2$ are $\cM$-parallel, and $C_3$ is free.
The arguments in \cite[Lem. 5.3]{AulicinoNguyenGen3TwoZeros} allow us to conclude that $C_1$ and $C_2$ are actually isometric. Moreover, after twisting $C_3$, we can assume that any vertical ray exiting $C_i, \, i=1,2,$ from its top border reenters $C_i$ through the bottom border after crossing the core curves of $C_3$ once.

Let $\sig_i,\sig'_i$ be respectively the top and bottom borders of $C_i$, then the condition above means that there is a pair of homologous vertical saddle connections $\del_i,\del'_i$ contained in $C_3$ joining the left endpoint (resp. right endpoint) of $\sig_i$ to the left endpoint (reps. right endpoint) of $\sig'_i$. Let $M_i$ denote the subsurface of $M$ cut out by $\del_i,\del'_i$ that contains $C_i$. Remark that $M_i$ is a slit torus, and $M_1,M_2$ are isometric.

\medskip

\noindent \ul {\em Case $\cM \subset \cH(2,1^2)$.} Let $x_0$ denote the unique double zero of $M$, and $x_1,x_2$ the simple ones.

\

\noindent \ul{\rm Claim:} {\em The boundary of $C_i, \, i=1,2,$ must contain two distinct zeros.}
 \begin{proof}
 Without loss of generality, let us suppose on the contrary that the boundary of $C_1$ contains only one zero. By Lemma~\ref{lm:s:cyl:dist:sim:zeros},  this zero must be $x_0$. A simple computation shows that the total angle at $x_0$  inside $M_1$ is $4\pi$. Therefore, the angle at $x_0$ outside of $M_1$ is $2\pi$. If we remove $M_1$ from $M$ and glue $\del_1, \del'_1$ together such that the points corresponding to $x_0$ in $\del_1$ and $\del'_1$ are identified,  we will obtain a surface $M'_1$ in the stratum $\cH(1,1)$ which admits a cylinder decomposition with two cylinders  in the horizontal direction.

 Note that $x_0$ gives rise to a regular point in $M'_1$, and $C_2$ can be considered as a (simple) cylinder in $M'_1$. The pair $\{\del_1,\del'_1\}$ now corresponds to a vertical simple closed geodesic on $M'_1$. Remark that there is a unique diagram for 2-cylinder decompositions of surfaces in $\cH(1,1)$ such that one of the cylinders is simple. We then observe that the condition that the larger cylinder contains vertical simple closed geodesic, and  a pair of vertical saddle connections that cut out a slit torus cannot be satisfied. Therefore we get a contradiction.
 \end{proof}

 It is also easy to see that a simple zero cannot occur in the boundaries of both $C_1$ and $C_2$ by an angle count. Therefore, we can assume that the boundary of $C_1$ contains $x_0$ and $x_1$, and the boundary of $C_2$ contains $x_0$ and $x_2$. As a consequence collapsing simultaneously $C_1$ and $C_2$ so that all the zeros collide yields a surface $M'$ in $\cH(4)$. From Proposition~\ref{prop:collapse:similar:cyl}, we know that $M'$ is contained in a rank two affine submanifold $\cM'$ of $\cH(4)$  which satisfies
 $$
 \dim \cM'=\dim \cM-1
 $$
 From the results of \cite{AulicinoNguyenWright} and \cite{NguyenWright}, we must have $\cM'=\tilde{\cQ}(3,-1^3)$. By construction, $M'$ is horizontally periodic with a unique horizontal cylinder $C$. Since $M \in \tilde{\cQ}(3,-1^3)$, $M'$ admits a Prym involution $\inv$.

 Let $\tilde{\sig}_1$ (resp. $\tilde{\sig}_2$) denote the horizontal saddle connection in $M'$ that is the degeneration of $C_1$ (resp. of $C_2$). We claim that $\tilde{\sig}_1$ and $\tilde{\sig}_2$ are exchanged by $\inv$. If they are not exchanged by $\inv$, then in any neighborhood of $M'$ in $\tilde{\cQ}(3,-1^3)$ we can find a surface on which $\tilde{\sig}_1$ and $\tilde{\sig}_2$ remain but the corresponding holonomy vectors are not equal. Since such a surface is obtained from a surface in $\cM$ by collapsing $\{C_1,C_2\}$,  this contradicts the condition that $C_1$ and $C_2$ are isometric.

 Since $\inv$ exchanges $\tilde{\sig}_1$ and $\tilde{\sig}_2$, by Proposition~\ref{DblCovExtSimpCyl}, we see that $\inv$ extends to a Prym involution on $M$. As a consequence, $\allowbreak M\in \cH(2,1,1)\cap\cP= \tilde{\cQ}(2,1,-1^3)$. Since the same is true for all surfaces in $\cM$ close to $M$ (see Proposition~\ref{prop:collapse:similar:cyl}), we draw that $\cM \subset \tilde{\cQ}(2,1,-1^3)$.  Notice that we have
 $$
 \dim \cM =\dim\tilde{\cQ}(3,-1^3)+1=\dim\tilde{\cQ}(2,1,-1^3)=5.
 $$
 Using the ergodicity of the action of $\SL(2,\R)$ on $\cM$, we conclude that $\cM= \tilde{\cQ}(2,1,-1^3)$.

 \medskip

 \noindent \ul{\em  Case $\cM \subset \cH(1^4)$.} By Lemma~\ref{lm:s:cyl:dist:sim:zeros}, we know that the boundary of $C_i, \, i=1,2,$ must contain two distinct zeros. By computing the angles at the zeros, it is also easy to check that a simple zero cannot be contained in the boundaries of both $C_1$ and $C_2$. Therefore, we can conclude that the boundaries of $C_1$ and $C_2$ contain two different pairs of simple zero. Thus collapsing simultaneously $C_1,C_2$ so that the zeros in each pair collide, we obtain a surface $M'$ in $\cH(2,2)$. Let $\tilde{\sig}_1$ and $\tilde{\sig}_2$ be the horizontal saddle connections in $M'$ that are the degenerations of $C_1$ and $C_2$ respectively.

 By Proposition~\ref{prop:collapse:similar:cyl}, we know that $M'$ is contained in some rank two affine submanifold $\cM'$ of $\cH(2,2)$ such that $\allowbreak \dim \cM=\dim \cM'+1$. By the results of \cite{AulicinoNguyenGen3TwoZeros}, we must have
 $$
 \cM' \in \{\dcoverhyp,\dcoverodd,\prym\}.
 $$
 In all cases, let $\inv$ be the Prym involution of $M'$.
 \begin{enumerate}
 \item  Assume that $\cM'=\dcoverhyp$.  In this case, $\inv$ fixes each of the zeros of  $M'$,  and there is a hyperelliptic involution $\iota$ which exchanges the two zeros of $M'$. By definition, $\iota$ has $8$ fixed points. Note that  two fixed points of $\iota$ are contained in the interior of $C_3$ (which is the unique horizontal cylinder in $M'$).

 The hyperelliptic involution $\iota$  induces a permutation on the set of horizontal saddle connections of $M'$. Since $\iota$ permutes the zeros of $M'$, a saddle connection fixed by $\iota$ must join one zero to the other one. In particular, each saddle connection fixed by $\iota$ contains exactly one fixed point.  We now remark that each $\tilde{\sig}_i$ is a saddle connection joining a zero of $M'$ to itself (this zero is the collision of two simple zeros in $M$). In particular, $\tilde{\sig}_i$ is not invariant by $\iota$.  Since $M'$ has  $6$ horizontal saddle connections,  this implies that $\iota$ has at most $4$ fixed points in the union of the horizontal saddle connections.  Thus $\iota$ has at most $6$ fixed points, which is a contradiction, and we can conclude that $\cM'\neq \dcoverhyp$.

 \item Assume now that $\cM'=\dcoverodd$. In this case, $\inv$ exchanges the zeros of $M'$, and there is a hyperelliptic involution $\iota$ that fixes each of the zeros of $M'$. It follows that $\iota$ has $6$ regular fixed points in $M'$. Recall that two fixed points of $\iota$ are contained in the interior of $C_3$. Hence, $\iota$ has $4$ regular fixed points in the union of the horizontal saddle connections. Remark that each fixed point must be contained in a saddle connection which joins a zero of $M'$ to itself.
 Since there are $6$ horizontal saddle connections, and at least two of them have distinct endpoints, it follows that every  saddle connection that joins a zero of $M'$ to itself is invariant by $\iota$. In particular, each of $\tilde{\sig}_1,\tilde{\sig}_2$ is invariant by $\iota$.

 We claim that  $\inv$ exchanges $\tilde{\sig}_1$ and $\tilde{\sig}_2$. This is because otherwise we can deform $M'$ slightly in $\dcoverodd$ such that the holonomy vectors associated to $\tilde{\sig}_1$ and $\tilde{\sig}_2$ are not equal, which would contradict the condition that $C_1$ and $C_2$ are isometric.

 Now, the observations above mean that the set $\tilde{\sig}_1\cup\tilde{\sig}_2$ is preserved by both $\inv$ and $\iota$. We can now use Proposition~\ref{DblCovExtSimpCyl} to conclude that $\iota$ and $\inv$ extend to two  involutions $\hat{\iota}$ and $\hat{\inv}$  of $M$ with the same number of fixed points respectively. In particular, $\hat{\iota}$ must be a (the) hyperelliptic involution, and $\hat{\inv}$ a Prym involution of $M$. We  thus have $\allowbreak M \in \cH(1^4)\cap\cP\cap\cL=\dcoverprinc$.
 Since the same is true for any surface in $\cM$ close enough to $M$ (see Proposition~\ref{prop:collapse:similar:cyl}), we derive that $\cM \subseteq \dcoverprinc$.
Since  we have
 $$
 \dim\cM=\dim \dcoverodd +1=5=\dim \dcoverprinc,
 $$
it follows that $\cM=\dcoverprinc$.

 \item Consider finally the case $\cM'=\prym$.  By the same argument as the previous case, we see that $\inv$ must permute $\tilde{\sig}_1$ and $\tilde{\sig}_2$.
     Thus, $\inv$ gives rise to a Prym involution of $M$ by Proposition~\ref{DblCovExtSimpCyl}, which means  that $M \in \prymprinc$. Since the same is true for any surface in $\cM$ close enough to $M$, we derive that $\cM \subseteq \prymprinc$. Finally, since we have
 $$
 \dim \cM=\dim \prym+1=6=\dim \prymprinc,
 $$
 it follows that $\cM=\prymprinc$.
 \end{enumerate}
\end{proof}

\section{Getting Four Cylinders}\label{sec:get4cyl}

The goal of this section is to prove that every rank two affine manifold in the strata $\cH(2,1,1)$ and $\cH(1^4)$ contain a translation surface with at least four cylinders.
However, this cannot be done all at once.  Due to our argument below, we can only prove this result for $\cH(2,1,1)$.
Once the classification of rank two affine manifolds in $\cH(2,1,1)$ is established, the desired result for the principal stratum will follow automatically.
We state the main result of the section here.

\begin{proposition}
\label{Min4CylProp}
Let $\cM$ be a rank two affine manifold in genus three.
\begin{itemize}
 \item[(1)] If $\cM \subset \cH(2,1,1)$,  then $\cM$ contains a horizontally periodic surface with at least four horizontal cylinders.

 \item[(2)] Assume that $\tilde \cQ(2,1,-1^3)$ is the only rank two affine manifold in $\cH(2,1,1)$. If $\cM \subset \cH(1^4)$,  $\cM$ contains a horizontally periodic surface with at least four horizontal cylinders.
\end{itemize}
\end{proposition}

By \cite[Lem. 3.2]{AulicinoNguyenGen3TwoZeros}, we know that $\cM$ always contains a horizontally periodic surface with at least three cylinders.
The following lemma is a generalization of \cite[Lem. 3.3]{AulicinoNguyenGen3TwoZeros}.

\begin{lemma}
\label{ANLem33Gen}
Let $\cM$ be a rank two affine manifold in genus three in a stratum with $k \geq 2$ zeros.  Assume that every rank two affine manifold in genus three with at most $k-1$ zeros admits an involution with four fixed points whose derivative is $-\id$.\footnote{For example, this is true of all surfaces in the Prym locus.}  If $\cM$ contains a horizontally periodic translation surface with two cylinders, one of which is simple, then $\cM$ contains a horizontally periodic surface with at least {\em three} cylinders, one of which is simple and not free.
\end{lemma}

\begin{proof}
Let $M\in \cM$ be a horizontally periodic surface with two horizontal cylinders $C_1$ and $C_2$, where $C_1$ is simple.  If $C_1$ and $C_2$ are $\cM$-parallel, then we are done by \cite[Lem. 2.14]{AulicinoNguyenGen3TwoZeros}. Thus let us suppose that $C_1$ is free. We claim that given any two zeros in $M$, there always exists a path between them consisting of horizontal saddle connections. This is because if we cut $M$ along a core curve of $C_1$ and a core curve of $C_2$, then the resulting surface is connected.  Otherwise, $C_1$ and $C_2$ are homologous, thus they cannot be free.

Note that each boundary component of $C_1$ contains a single zero of $M$.  Let $x_1$ be a zero of highest order in $M$. Observe that there must exist a horizontal saddle connection $\sig$ connecting $x_1$ to another zero $x_2$. Since $\sig$ is not one of the boundary components of $C_1$, it must be contained in both sides of $C_2$, thus we have a simple cylinder $D$ contained in $C_2$ whose boundary contains $x_1$ and $x_2$. We consider the following cases:

\begin{itemize}
\item[$\bullet$]  $x_1$ is of order $\geq 2$:  We claim that $D$ is not free. Indeed, if this is the case, then we can collapse $D$ to get a surface $M'$ in a stratum with $k-1$ zeros, one of the zeros of $M'$ is of order at least $3$. Thus $M'$ belongs to $\cH(3,1)$ or $\cH(4)$.    Since there is no rank two affine submanifold in $\cH(3,1)$, we only need to consider the case $M'\in \cH(4)$. In this case we must have $M'\in \tilde{\cQ}(3,-1^3)$.  In particular, $M'$ has an involution $\inv$ with four fixed points whose derivative is $-\id$. Note that the unique zero of $M$ must be a fixed point of $\inv$.  By construction, $M'$ has two horizontal cylinders, one of which is simple, the other one is not. Thus, they are both fixed by $\inv$.  But a cylinder fixed by $\inv$ must contain two fixed points of $\inv$  in its interior. Therefore, $\inv$ must have at least $5$ fixed points, which is a contradiction.

Since $\cM$ is defined over $\Q$, we can assume that $D$ is vertical and $M$ is a square-tiled surface. Since $D$ is not free, it is $\cM$-parallel to another vertical cylinder $D'$, which must be entirely contained in the closure of $C_2$. In particular, $D$ and $D'$ do not fill $M$. Thus there exists at least another vertical cylinder, which means that we have at least $3$ vertical cylinders, one of which is simple and not free.

\item[$\bullet$]  $x_1$ is a simple zero, i.e. $M\in \cH(1^4)$: if $D$ is free, then we can collapse it to get a surface  $M'\in \cH(2,1^2)$.  Since the involution of $M'$ must fix the double zero,   by the same argument as above we get a contradiction. Thus   $D$ is not free, and we also get the desired conclusion.
\end{itemize}
\end{proof}

Recall that in \cite[Lem. 4.1]{AulicinoNguyenGen3TwoZeros}, we have divided $3$-cylinder diagrams in genus three into three Cases 3.I), 3.II), 3.III). The following is a slight generalization of \cite[Prop. 5.5]{AulicinoNguyenGen3TwoZeros}.

\begin{proposition}
\label{2SimpCylImp4CylsProp}
Let $\cM$ be a rank two affine manifold in genus three with at least two zeros.  If $M \in \cM$ is a horizontally periodic translation surface satisfying Case 3.I) and two of the horizontal cylinders are simple, then there is a horizontally periodic surface in $\cM$ with at least four cylinders.
\end{proposition}

\begin{proof}
By \cite[Lem. 5.3]{AulicinoNguyenGen3TwoZeros}, the two simple cylinders in $M$ are $\cM$-parallel and isometric.  Furthermore, they can be twisted so that there is a vertical trajectory passing exactly once through each.  This yields either Case (A) or (B) in Figure \ref{fig:2SimpCylPf1}.  Next, consider the vertical direction after perturbing to a nearby square-tiled surface, we see that each of the simple cylinders must be contained in (the closure of) a vertical cylinder. Therefore, there must exist at least three vertical cylinders.

If there are four or more cylinders, then we are done.  Otherwise, there is a vertical cylinder $D$ which is contained in the closure $C_3$.  Since no cylinder parallel to $D$ is entirely contained in $\ol{C}_3$, $D$ is free by \cite[Prop. 3.3(b)]{NguyenWright}.  After rotating the surface $M$ by $\pi/2$ and redrawing, we get the horizontally periodic surfaces in Figure~\ref{fig:2SimpCylPf2}.  In both cases, we twist the horizontal cylinder $D$ so that saddle connection $c$ lies where it does in both figures.  By applying \cite[Lem. 2.14]{AulicinoNguyenGen3TwoZeros} or \cite[Cor. 6]{SmillieWeissMinSets} to the vertical direction yields a translation surface with four or more parallel cylinders.

\begin{figure}[htb]
\centering
\begin{minipage}[t]{0.49\linewidth}
\centering
\begin{tikzpicture}[scale=0.5]
\draw[thin] (0,4) -- (0,0) -- (8,0) -- (8,2) -- (6,2) -- (6,4) -- (4,4) -- (4,2) -- (2,2) -- (2,4) -- cycle;
\draw[thin] (0,2) -- (2,2) (4,2) -- (6,2);

\draw (1,4) node[above] {\tiny $a$} (1,2) node[above] {\tiny $a'$} (1,0) node[below] {\tiny $a$} (5,4) node[above] {\tiny $b$} (5,2) node[above] {\tiny $b'$} (5,0) node[below] {\tiny $b$};

\foreach \x in {(0,4), (0,2), (0,0), (2,4),(2,2), (2,0), (8,2), (8,0)} \filldraw[fill=white] \x circle (3pt);
\foreach \x in {(4,4), (4,2), (4,0), (6,4),(6,2), (6,0)} \filldraw[fill=white] \x circle (3pt);

\draw (4,-1) node {(A)};

\end{tikzpicture}
\end{minipage}
\begin{minipage}[t]{0.49\linewidth}
\begin{tikzpicture}[scale=0.5]
\draw[thin] (0,6) -- (0,2) -- (2,2) -- (2,0) -- (4,0) -- (4,2) -- (8,2) -- (8,4) -- (2,4) -- (2,6) -- cycle;
\draw[thin] (0,4) -- (2,4) (2,2) -- (4,2);

\draw (1,6) node[above] {\tiny $a$} (1,4) node[above] {\tiny $a'$} (1,2) node[below] {\tiny $a$} (3,4) node[above] {\tiny $b'$} (3,2) node[above] {\tiny $b$} (3,0) node[below] {\tiny $b'$} ;

\foreach \x in {(0,6), (0,2), (2,6), (2,2), (4,2), (8,2)} \filldraw[fill=white] \x circle (3pt);

\foreach \x in {(0,4), (2,4), (2,0), (4,4), (4,0), (8,4)} \filldraw[fill=white] \x circle (3pt);
\draw (4,-1) node {(B)};
\end{tikzpicture}
\end{minipage}
\caption{$3$-cylinder diagrams with two simple cylinders}
\label{fig:2SimpCylPf1}
\end{figure}
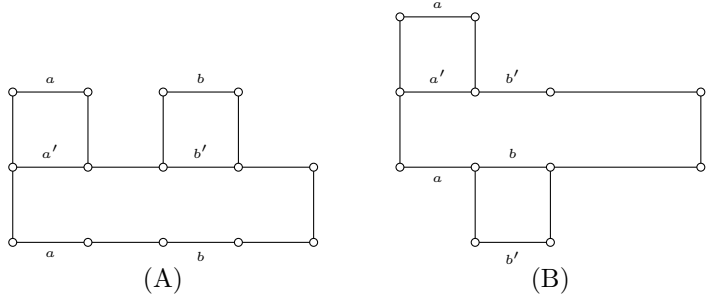

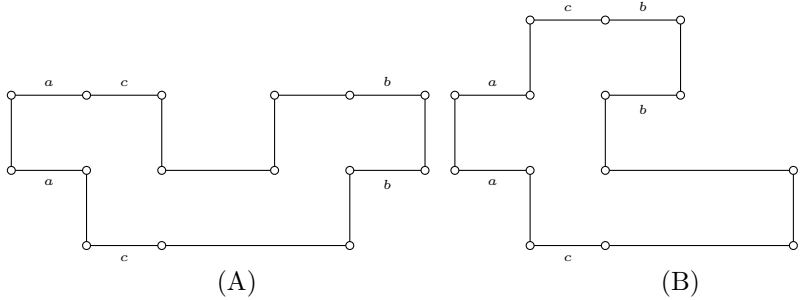
\begin{figure}[htb]
\centering
\begin{minipage}[t]{0.49\linewidth}
\centering
\begin{tikzpicture}[scale=0.5]
\draw[thin] (0,0) -- (0,2) -- (-2,2) -- (-2,4) -- (2,4) -- (2,2) -- (5,2) -- (5,4) -- (9,4) -- (9,2) -- (7,2) -- (7,0) -- cycle;
\draw (-1,4) node[above] {\tiny $a$} (-1,2) node[below] {\tiny $a$} (8,4) node[above] {\tiny $b$} (8,2) node[below] {\tiny $b$} (1,4) node[above] {\tiny $c$} (1,0) node[below] {\tiny $c$};
\foreach \x in {(-2,2), (-2,4), (0,4), (0,2), (0,0), (2,4),(2,2), (2,0), (5,2), (5,4), (7,4), (7,2), (7,0), (9,2), (9,4)} \filldraw[fill=white] \x circle (3pt);

\draw (4,-1) node {(A)};
\end{tikzpicture}
\end{minipage}
\begin{minipage}[t]{0.49\linewidth}
\begin{tikzpicture}[scale=0.5]
\draw[thin] (0,0) -- (0,2) -- (-2,2) -- (-2,4) -- (0,4) -- (0,6) -- (4,6) -- (4,4) -- (2,4) -- (2,2) -- (7,2) -- (7,0) -- cycle;
\draw (-1,4) node[above] {\tiny $a$} (-1,2) node[below] {\tiny $a$} (3,6) node[above] {\tiny $b$} (3,4) node[below] {\tiny $b$} (1,6) node[above] {\tiny $c$} (1,0) node[below] {\tiny $c$};
\foreach \x in {(-2,2), (-2,4), (0,4), (0,6), (2,6), (4,6), (4,4), (0,2), (0,0), (2,4),(2,2), (2,0), (7,2), (7,0)} \filldraw[fill=white] \x circle (3pt);

\draw (4,-1) node {(B)};
\end{tikzpicture}
\end{minipage}
\caption{$3$-cylinder diagrams with at least four vertical cylinders}
\label{fig:2SimpCylPf2}
\end{figure}
\end{proof}

Let $M$ be a horizontally periodic surface in $\cH(2,1^2)\cup\cH(1^4)$. Let  $\Gr$ be the graph which is the union of all horizontal saddle connections in $M$. This graph is called the {\em separatrix diagram} in the literature and has a {\em ribbon structure} (see \cite[Sec. 4]{KontsevichZorichConnComps}).
If $M \in \cH(2,1,1)$, then $G$ has $3$ vertices and $7$ edges.  If $M \in \cH(1^4)$, then $G$ has $4$ vertices and $8$ edges. Note that the valency of a simple zero is $4$ and of a double zero is $6$. Since each edge of $\Gr$ is a horizontal saddle connection in $M$, we can equip it with the orientation from the left to the right.

Let $U$ be a neighborhood of $\Gr$ in $M$ consisting of the points whose distance to $\Gr$ is at most $\eps$, with $\eps >0$ small enough.  Each component of $\partial U$ is a core curve of a horizontal cylinder, and also homotopic to a cycle of edges of $\Gr$. We say that two boundary components are {\em adjacent} if the corresponding cycles have a common edge.

We color a component of $\partial U$ red if its orientation (which is induced by the orientation of $U$) agrees with the orientation of the corresponding cycle in $\Gr$, otherwise we color it blue.  A red boundary component corresponds to the upper side of a cylinder, while a blue one corresponds to the lower side of a cylinder.  Clearly, we have a pairing between the set of red boundary components and the set of blue ones, two boundary components are paired if they belong to the same cylinder. Note that two adjacent boundary components must have different colors because a saddle connection cannot be contained in the tops (resp. bottoms) of two different cylinders.

\begin{proposition}
\label{NoExceptCaseProp}
Let $M \in \cH(2,1^2) \cup \cH(1^4)$ be a horizontally periodic translation surface satisfying Case 3.I).  Then at least one of the following occurs
\begin{itemize}
\item[(a)] One of the cylinders is semi-simple,
\item[(b)] There is a horizontal saddle connection contained in both the top and bottom of the same cylinder.
\end{itemize}
\end{proposition}
\begin{proof}
Consider the separatrix diagram $\Gr$ and its neighborhood $U$ described above.  The hypothesis implies that $\Gr$ is connected and  $U$ is homeomorphic to a sphere  with six open discs removed.  As a consequence,  $G$ is a planar graph.

A loop in $\Gr$ is an edge that joins a vertex to itself. If there is a component of $\partial U$ that is homotopic to a loop in $\Gr$, then one of the cylinders is semi-simple. Since $\Gr$ is planar, and using the hypothesis on the number of edges and vertices of $\Gr$, one can easily check that if there are some loops in $\Gr$, then there must exist a loop which bounds a disc. Hence, in this case we have  a semi-simple cylinder.

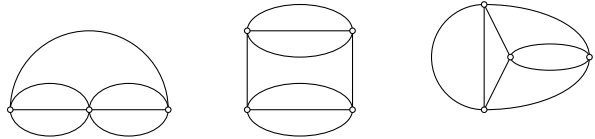
\begin{figure}[htb]
\centering
\begin{tikzpicture}[scale=0.35]
 \draw (6,0) arc (0:180:3);
 \foreach \x in {(3,0), (6,0)} \draw \x arc (0:180:1.5 and 1);
 \foreach \x in {(0,0), (3,0)} \draw \x arc (180:360:1.5 and 1);
 \draw (0,0) -- (6,0);
 %***************************
 \draw (9,0) -- (13,0) -- (13,3) -- (9,3) -- cycle;

 \foreach \x in {(13,3), (13,0)} \draw \x arc (0:180:2 and 1);
 \foreach \x in {(9,3), (9,0)} \draw \x arc (180:360:2 and 1);

 %***************************
 \draw (18,4) arc (90:270:2);
 \draw (18,0) arc (-90:90:4 and 2);
 \draw (22,2) arc (0:180:1.5 and 0.5);
 \draw (19,2) arc (180:360: 1.5 and 0.5);
 \draw (18,4)-- (18,0) -- (19,2) -- cycle;

 \foreach \x in {(0,0),(3,0), (6,0), (9,3), (9,0), (13,3), (13,0), (18,4), (18,0), (19,2), (22,2)} \filldraw[fill=white] \x circle (3pt);
\end{tikzpicture}

 \caption{Admissible configurations for the graph of saddle connections with no loops in Case 3.I)}
 \label{fig:config:graph:s:c:3I}
\end{figure}

Assume from now on that there are no loops in $\Gr$. There are three admissible configurations for $\Gr$, which are shown in Figure~\ref{fig:config:graph:s:c:3I}, one for $\cH(2,1^2)$ and two for $\cH(1^4)$. Observe that in all cases, the outer boundary component of $U$ is adjacent to three other boundary components. Therefore, the outer component must be paired with one of the adjacent ones. This implies immediately that there is an edge of $\Gr$ that is contained in both the top and the bottom sides of the corresponding cylinder. The proposition is then proved.
\end{proof}

\subsection*{Proof of Proposition \ref{Min4CylProp}}
\begin{proof}
By \cite[Lem. 3.2]{AulicinoNguyenGen3TwoZeros}, there exists a horizontally periodic surface $M \in \cM$ with at least three cylinders.  By \cite[Lem. 4.1]{AulicinoNguyenGen3TwoZeros}, $M$ satisfies one of three possible cases.

\begin{enumerate}
\item[(a)] If $M$ satisfies Case 3.II), then by the assumption and \cite[Lem. 4.3]{AulicinoNguyenGen3TwoZeros}, $M$ is $\cM$-cylindrically unstable. Thus there exists $M' \in \cM$ that is horizontally periodic with at least four horizontal cylinders.
  %If we also have $\cM \subset \cH(2,1,1)$, then we are done by \cite[Thm. 1.1]{AulicinoNguyenGen3TwoZeros} without any need for the covering assumption.

\item[(b)] If $M$ satisfies Case 3.III), then denote the homologous cylinders by $C_1,C_2$, and the remaining one by $C_3$. If we cut $M$ along a core curve in each of $C_1,C_2$, then glue the boundary components of the new surface after exchanging the pairings, we will obtain two translation surfaces of genus two, both of which are horizontally periodic. One of the new surfaces has two horizontal cylinders one of which is $C_3$. We denote this surface  $M^1$, and the other one $M^2$. Note that since $M^1$ is a genus two translation surface, $C_3$ is either simple or contains a horizontal saddle connection in both of its sides.

We have several possibilities. Assume that $C_3$ contains a simple cylinder $C$. If the boundary of $C$ contains only simple zeros, then the simple zeros are distinct by Lemma~\ref{lm:s:cyl:dist:sim:zeros}. It is easy to check that there is no cylinder parallel to $C$ that is entirely contained in $C_3$.  Hence, $C$ is free and can be collapsed.  Note that in this case $M$ degenerates to a surface $M'$ in $\cH(2,2)$ or $\cH(2,1^2)$. By Proposition~\ref{prop:collapse:free:sim:cyl}, $M$ is contained in a rank two affine submanifold $\cM'$ in $\cH(2,2)$ or in $\cH(2,1^2)$ such that $\dim \cM'=\dim\cM-1$. By the results of \cite{AulicinoNguyenGen3TwoZeros} and the hypothesis of the proposition, $\cM'$ is one of the following loci
$$
\{\dcoverhyp, \dcoverodd, \tilde{\cQ}(4,-1^4),\tilde{\cQ}(2,1,-1^3)\}.
$$
By Proposition~\ref{prop:many:cyl:rktwo}, there exists $M' \in \cM'$ admitting a cylinder decomposition with four or more cylinders.  We conclude by Proposition~\ref{CylStDens} and  Lemma~\ref{kCylsInBdLem}.

If the boundary of $C_3$ contains a double zero, then the two zeros in its boundary are the same, and we have a cylinder diagram similar to \cite[Lem. 4.8]{AulicinoNguyenGen3TwoZeros}. But in this case it is easy to check that the proof of \cite[Prop. 4.8]{AulicinoNguyenGen3TwoZeros} goes through without any challenge even though the top of $C_1$ and the bottom of $C_2$ contain four saddle connections instead of three.

Finally, if $C_3$ is itself a simple cylinder, we  apply \cite[Lem. 4.7]{AulicinoNguyenGen3TwoZeros} to reduce to the previous cases.

\item[(c)] If $M$ satisfies Case 3.I), by Proposition~\ref{NoExceptCaseProp}, we know that either one of the horizontal cylinders is semi-simple or contains a simple cylinder. If the latter occurs, since we can always suppose that the simple cylinder is vertical and $M$ is a square-tiled surface, it follows that $\cM$ contains a vertically periodic surface with one simple vertical cylinder. Using Lemma~\ref{ANLem33Gen}, we derive that $\cM$ contains a horizontally periodic surface with at least $3$ cylinders one of which is simple.
If the cylinder diagram of this surface satisfies Case 3.II or Case  3.III, then we conclude as above. Thus, we are left to consider the case $M$ is horizontally periodic satisfying Case 3.I, and one of the horizontal cylinders is semi-simple.

We only need to consider the case $M$ is $\cM$-cylindrically stable. Since the horizontal cylinders of $M$ cannot be all free (see \cite[Lem. 2.15]{AulicinoNguyenGen3TwoZeros}), they must fall into two equivalence classes. Let us denote these cylinders by $C_1,C_2,C_3$, where $C_1$ and $C_2$ are $\cM$-parallel, while $C_3$ is free.
Let us first consider the case one of the horizontal cylinders is simple. By Proposition~\ref{2SimpCylImp4CylsProp}, we can assume that only one of $C_1,C_2,C_3$ is simple.
If one of $C_1$ and $C_2$ is simple, then the other one is not, and we conclude by \cite[Prop. 5.6]{AulicinoNguyenGen3TwoZeros}. If $C_3$ is simple, then we conclude by \cite[Prop. 5.9]{AulicinoNguyenGen3TwoZeros}, and Proposition~\ref{2SimpCylImp4CylsProp}. Finally, in the case where none of $C_1,C_2,C_3$ is simple, and one of them is strictly semi-simple, we conclude by \cite[Prop. 5.14]{AulicinoNguyenGen3TwoZeros} and  Proposition~\ref{2SimpCylImp4CylsProp}.
\end{enumerate}
\end{proof}

% By using Lemma~\ref{ANLem33Gen} and Proposition~\ref{2SimpCylImp4CylsProp} instead of  \cite[Lem. 3.3]{AulicinoNguyenGen3TwoZeros} and \cite[Prop. 5.5]{AulicinoNguyenGen3TwoZeros}, respectively,
%we can improve the statements of Propositions 5.8, 5.9, 5.10 in \cite{AulicinoNguyenGen3TwoZeros} to cover all of the strata in $\cH_3$.  Therefore, we can then use the same sequence of arguments
%presented in the proofs of \cite[Prop. 5.11]{AulicinoNguyenGen3TwoZeros} and \cite[Prop. 5.14]{AulicinoNguyenGen3TwoZeros} to get the desired conclusion. The details are left to the reader.
%\end{enumerate}
%\end{proof}

\section{Four Cylinders}\label{sec:4cyl}

We recall \cite[Lem 6.1]{AulicinoNguyenGen3TwoZeros} that enumerates all topological types of $4$-cylinder decompositions in genus three.

\begin{lemma}\cite{AulicinoNguyenGen3TwoZeros}
\label{lm:4CylDeg}
If a translation surface $M$ in genus three decomposes into four cylinders, then pinching the core curves of those cylinders degenerates the surface to one of four possible surfaces:
\begin{itemize}
\item 4.I) Two spheres joined by four pairs of simple poles.
\item 4.II) Two spheres joined by two pairs of simple poles such that each sphere has a pair of simple poles.
\item 4.III) Two spheres joined by three pairs of simple poles such that one sphere carries an additional pair of simple poles.
\item 4.IV) Two spheres and a torus such that the spheres have three simple poles and the torus has two simple poles.
\end{itemize}
\end{lemma}

In what follows we will individually consider each of those topological types of $4$-cylinder decomposition. The final result is the following.

\begin{proposition}\label{prop:4cyl}
Let $\cM$ be a rank two affine submanifold of either $\cH(2,1^2)$ or $\cH(1^4)$. Assume that $\cM$ contains a horizontally periodic surface with $4$ horizontal cylinders.
\begin{itemize}
\item[(a)]  If  $\cM \subset \cH(2,1^2)$, then either  $\cM$ contains a horizontally  periodic surface with $5$ horizontal cylinders or $\cM=\tilde{\cQ}(2,1,-1^3)$.

\item[(b)] If $\cH \subset \cH(1^2)$, and assume that $\tilde{\cQ}(2,1,-1^3)$ is the unique rank two affine submanifold in $\cH(2,1^2)$, then either $\cM$ contains a horizontally  periodic surface with at least $5$ horizontal cylinders or $\cM \in \{\dcoverprinc,\prymprinc\}$.
\end{itemize}
\end{proposition}
\begin{proof}
This proposition is the consequence of Corollary~\ref{Case4IVImpDblCov} and the Propositions \ref{Case4IIIProp}, \ref{prop:H211NoCase4II}, \ref{prop:C4II:H1111}, \ref{prop:C4Ia:H211}, \ref{prop:C4Ia:H1111}, \ref{prop:C4Ib:H211}, and \ref{prop:C4Ib:H1111}.
\end{proof}

The proofs of the results mentioned above use several technical lemmas, that are essential but somewhat tedious as the ideas involved already appeared in the previous work~\cite{NguyenWright, AulicinoNguyenWright, AulicinoNguyenGen3TwoZeros}. For this reason, we defer some of their proofs to the appendix in order to keep the focus on the novelties.

\subsection{Case 4.IV)}

In this case the core curves of the cylinders cut the surface into two three-holed spheres, and a two-holed torus.

The following lemma follows from the proof of \cite[Lem. 5.6]{AulicinoZeroExpGen3}.  The two possible conclusions correspond to the possibility that $M$ is $\cM$-cylindrically stable, which was assumed in the proof of \cite[Lem. 5.6]{AulicinoZeroExpGen3}, or to the possibility that $M$ is $\cM$-cylindrically unstable, in which case we can produce more cylinders.

\begin{lemma}
\label{lm:4CylCaseIV}
Let $\cM$ be a rank two affine manifold in genus three.  If $M \in \cM$ is a horizontally periodic translation surface satisfying Case 4.IV), then either $M$ has a free simple cylinder with distinct zeros at each end, or there exists $M' \in \cM$ admitting a cylinder decomposition with at least five cylinders.
\end{lemma}

\begin{corollary}
\label{Case4IVImpDblCov}
Let $\cM$ be a rank two affine manifold in $\cH(2,1^2)$ or $\cH(1^4)$.  If $\cM \subset \cH(1^4)$, then we add the additional assumption that the only rank two affine submanifold of $\cH(2,1^2)$ is $\tilde{\cQ}(2,1,-1^3)$.  If $M \in \mathcal{M}$ is a horizontally periodic translation surface satisfying Case 4.IV), then either $\cM$ contains a horizontally periodic surface in Case 4.I) or 4.II), or there exists $M' \in \cM$ horizontally periodic with at least five cylinders.
\end{corollary}
\begin{proof}
By Lemma \ref{lm:4CylCaseIV}, we only need to consider the case $M$ has a free simple cylinder $C$ with different zeros on its boundary.  Collapsing $C$ results in a translation surface $M'$  which is contained in a rank two affine submanifold $\cM'$ of a lower stratum of genus three. Since there is no rank two affine submanifold in $\cH(3,1)$, we derive that either $\cM' \subset \cH(2,2)$ or $\cM'\subset \cH(2,1^2)$.  From Proposition~\ref{prop:many:cyl:rktwo} we know that either $\cM'$ contains a surface admitting a cylinder decomposition satisfying either Case 4.I), 4.II), or 5.I). We can then conclude by Proposition \ref{CylStDens} and Lemma \ref{kCylsInBdLem}.
\end{proof}

\subsection{Case 4.III)}\label{sec:4cyl:C4III}
In this case,  the core curves of the cylinders cut $M$ into a three-holed sphere and a five-holed sphere, the former contains a simple zero, while the latter contains the other zeros of $M$.
Let us denote by $x_0$ the simple zero contained in the three-holed sphere.  This zero is contained in the boundary of three cylinders, denoted by $C_1,C_2,C_3$. We number them so that $\ell(C_3)=\ell(C_1)+\ell(C_2)$. The remaining cylinder is denoted by $C_4$. Let $c_i$ be a core curve of $C_i$.

 \begin{figure}[htb]
 \centering
\begin{minipage}[t]{0.3\linewidth}
\centering
 \begin{tikzpicture}[scale=0.6]
 \draw (-1,6) -- (0,4) -- (0,2) -- (2,2) -- (2,0) -- (6,0) -- (6,2) -- (4,2) -- (4,6) -- (2,6) -- (2,4) -- (1,6) -- cycle;
 \foreach \x in {(-1,6), (0,2), (1,6),(2,6), (2,2), (2,0), (4,6),(4,2),(4,0), (4.5,2), (5.5,0), (6,2),(6,0)} \filldraw[fill=black] \x circle (3pt);
  \foreach \x in {(0,4),(2,4),(4,4)} \filldraw[fill=white] \x circle (3pt);

  \draw (0,6) node[above] {\tiny $1$} (1,2) node[below] {\tiny $1$} (3,6) node[above] {\tiny $2$} (3,0) node[below] {\tiny $2$} (5.3,2) node[above] {\tiny $3$} (4.7,0) node[below] {\tiny $3$} (4.3,2) node[above] {\tiny $4$} (5.7,0) node[below] {\tiny $4$};
  \draw (0.5,5) node {\tiny $C_1$} (3,5) node {\tiny $C_2$} (2,3) node {\tiny $C_3$} (4,1) node {\tiny $C_4$};
\end{tikzpicture}
\end{minipage}
\caption{Case 4.III) in $\cH(2,1^2)$: cylinder labels, the white vertex is $x_0$}
\label{fig:Case4IIICylLbls}
 \end{figure}
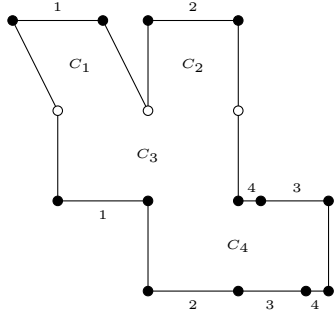
Remark that all of the horizontal saddle connections starting from $x_0$ end at $x_0$. The other horizontal saddle connections form a connected planar graph $\Gr$ with $2$ or $3$ vertices, such that a neighborhood of $\Gr$ is homeomorphic to a five-holed sphere.  We start by

\begin{lemma}
\label{4CylCaseIIIEqCls}
Let $\cM$ be a rank two affine manifold in a stratum in genus three with at least three zeros.  If $\cM$ contains a horizontally periodic surface $M$ with four cylinders satisfying Case 4.III) such that $M$ is $\cM$-cylindrically stable, then the equivalence classes are $\{C_1, C_2, C_3\}$ and $\{C_4\}$.
\end{lemma}
\begin{proof}
Since $\cM$ is of rank two, the horizontal cylinders belong to at least two equivalence classes. Either $C_4$ is free, or it is not.  If $C_4$ is free, then  the relation $c_1 + c_2 = c_3$ implies that either all three cylinders $C_1,C_2,C_3$ are free, or they belong to the same equivalence class. In the former case, all four cylinders are free and we have a contradiction with the rank two assumption.  Hence, if $C_4$ is free, then we are done.

If $C_4$ is not free, then it is $\cM$-parallel to another cylinder say $C_1$. Since there are at least two equivalence classes, no other cylinder can be $\cM$-parallel to $C_4$.  Hence, $C_2$ and $C_3$ must be free.   Let $\xi_i$ denote the vector in $H^1(M,\Sig,\R)$ which is tangent to the path defined by the shearing of $C_i$.
Since $C_2$ and $C_3$ are free, it follows that $\xi_2$ and $\xi_3$ are contained in $T^\R_M\cM$. The condition that $C_1$ and $C_4$ are $\cM$-parallel implies that  $\xi_1+\xi_4 \in T^\R_M\cM$.

Note that one can identify $H^1(M,\Sig,\R)$ and $H^1(M,\R)$ with $H_1(M\smin\Sig,\R)$ and $H_1(M,\R)$ respectively by using Poincar\'e duality (see \cite[$\S$4.1]{MirzakhaniWrightBoundary} for details). Moreover, in this setting, the natural projection $p: H^1(M,\Sig,\R) \ra H^1(M,\R)$ can be identified with the projection $p': H_1(M\smin\Sig, \R) \ra H_1(M,\R)$.  Using this identification, up to a non-zero constant $\xi_i$ is equal to $[c_i] \in H_1(M\smin\Sig, \R)$, and $p(\xi_i)$ is equal to $p'([c_i])=[c_i] \in H_1(M,\R)$ (see~\cite[Rem. 2.5]{WrightCylDef}).
As a consequence,  we see that there exist $a,b \in \R_{>0}$ such that the vectors
$a[c_1]+b[c_4], [c_2], [c_3]$  all belong to $p(T^\R_M\cM)$.  But those vectors span a three dimensional isotropic subspace of $H_1(M,\R)$ which contradicts the fact that $p(T^\R_M\cM)$ is symplectic and the assumption that $\dim p(T^\R_M\cM)=4$. Therefore, $C_4$ must be free and the lemma follows.
\end{proof}

\begin{lemma}\label{lm:4cyl:C4III:C4}
Let $\cM$ be a rank two affine submanifold in a stratum with at least three zeros.  Assume that $M \in \cM$ admits a cylinder decomposition satisfying Case 4.III) in the horizontal direction. Then either $C_4$ is semi-simple or $C_4$ contains a free simple cylinder with two distinct zeros in its boundary.
\end{lemma}
\begin{proof}
 See Appendix~\ref{sec:aux:lms}.
\end{proof}

We can now show

\begin{proposition}
\label{Case4IIIProp}
Let $\cM$ be a rank two affine manifold in a stratum in genus three with at least three zeros.  If $\cM \subset \cH(1^4)$, then we add the assumption that $\tilde{\cQ}(2,1,-1^3)$ is the unique rank two affine submanifold in $\cH(2,1^2)$. If $M \in \cM$ is horizontally periodic with four cylinders and $M$ is $\cM$-cylindrically stable, then $M$ does not  satisfy Case 4.III).
\end{proposition}
\begin{proof}
By Lemma~\ref{lm:4cyl:C4III:C4} we have to consider two cases
\begin{itemize}
 \item $C_4$ contains a free simple cylinder $D$ with two distinct zeros in its boundary. Collapsing $D$, we get a surface $M'$ in a rank two affine manifold $\cM'$ which is contained in either $\cH(3,1)$ or $\cH(2,1,1)$. The former case is ruled out since $\cH(3,1)$ contains no rank two affine submanifolds. For the latter case, by the hypothesis, we must have $\cM' = \tilde{\cQ}(2,1,-1^3)$, hence $M'$ admits a Prym involution $\tau$. This involution must send $x_0$ to another simple zero $x_1$, hence  the  saddle  connections containing $x_0$ are mapped to those that contain $x_1$. Since all the horizontal saddle connections starting from $x_0$ join $x_0$ to itself, the same is true for the saddle connections starting from $x_1$. But by assumption, $x_1$ is contained in the same component as the double zero after the pinching of the core curves of $C_i, \, i=1,\dots,4$, which means that there is a horizontal saddle connection joining $x_1$ and the double zero. Thus we have a contradiction which rules out
this case.

\item $C_4$ is semi-simple. By Lemma~\ref{lm:C4III:C4notS}, we know that $C_4$ is not simple.  Using the fact that each saddle connection in the boundary of $C_4$ must be contained in the boundary of another cylinder, by an angle count, it is not difficult to check that the boundary of $C_4$ contains at least two distinct zeros.  We can twist then collapse $C_4$ such that there is a unique (vertical) saddle connection joining two different zeros that is shrunk to a point. The resulting surface must be contained in a rank two affine submanifold of $\cH(3,1)$ or $\cH(2,1,1)$. The remainder of the proof follows from the same arguments as the previous case.
\end{itemize}
\end{proof}

\subsection{Case 4.II)}\label{sec:4cyl:C4II}
Let $M$ be a horizontally periodic surface in $\cM$ with four cylinders satisfying Case 4.II). We will denote the homologous cylinders by $C_1$ and $C_2$, and the remaining cylinders by $C_3$ and $C_4$.  In particular, $C_1$ and $C_2$ are $\cM$-parallel.

Note that if we cut $M$ along a core curve of $C_1$ and a core curve of $C_2$ and exchange the gluings, we will obtain two genus two translation surfaces containing $C_3$ and $C_4$, respectively. Thus we have

\begin{lemma}\label{lm:C4II:C3C4:likeg2}
Either $C_3$ (resp. $C_4$) is a simple cylinder, or there exist some saddle connections that are contained in both top and bottom of $C_3$ (resp. $C_4$).
\end{lemma}
For $i=3,4$, let $k_i$ be the number of saddle connections that are contained in both top and bottom of $C_i$. Lemma~\ref{lm:C4II:C3C4:likeg2} implies that $k_i=0$ if and only if $C_i$ is a simple cylinder.  We will need the following

\begin{lemma}\label{lm:C4II:preliminary}
Let $M$ be a horizontally periodic translation surface satisfying Case 4.II) in a rank two affine manifold $\cM \subset \cH(2,1^2) \cup \cH(1^4)$.  Assume that $M$ is $\cM$-cylindrically stable. Then $C_3$ and $C_4$ are $\cM$-parallel, and $k_3=k_4$.
\end{lemma}
\begin{proof}
See Appendix~\ref{sec:proof:lm:C4II:preliminary}.
\end{proof}

\begin{proposition}
\label{prop:H211NoCase4II}
Let $M \in \cH(2,1^2)$ be a horizontally periodic translation surface in a rank two affine manifold $\cM$.  If $M$ is $\cM$-cylindrically stable, then $M$ does not satisfy Case 4.II).
\end{proposition}
\begin{proof}
 Set $k=k_3=k_4$. If  $k=0$, then both $C_3$ and $C_4$ are simple.  Twist and perform an extended cylinder collapse (see \cite[Lem. 4.7]{AulicinoNguyenGen3TwoZeros}) to get a new translation surface such that both $C_3$ and $C_4$ contain simple cylinders. Therefore, we can assume that $k>0$, which means that $\ol{C}_3$ and $\ol{C}_4$ contain some simple cylinders.

Without loss of generality, let $C_3$ be the cylinder with the double zero in its boundary, and $C_4$ be the cylinder with two simple zeros in its boundary. Since $C_3$ can be realized as a cylinder in some surface in $\cH(2)$, there is  a unique saddle connection, denoted by $\sig_3$, which is contained in both top and bottom of $C_3$. By Lemma~\ref{lm:C4II:preliminary}, there is also a unique saddle connection $\sig_4$ which is contained in both top and bottom of $C_4$.

Let $D$ be a simple cylinder in $\ol{C}_3$ consisting of closed geodesics crossing  $\sig_3$ once. Let $D'$ be the cylinder in the equivalence class of $D$ which is contained in $\ol{C}_4$. Note that $D'$ is also a simple cylinder (but its core curves may cross $\sig_4$ more than once), and its complement  in $C_4$ is a rectangle that we will denote by  $R_4$.

\begin{figure}[htb]
 \centering
 \begin{tikzpicture}[scale=0.3]
  \fill[blue!30] (8,5) rectangle (12,9);
  \fill[blue!30] (10,0) -- (13,-1) -- (13,2) -- (10,3);
  \fill[green!30] (2,3) rectangle (4,11);
  \draw (0,9) -- (0,5) -- (2,3) -- (4,3) -- (4,0) -- (10,0) -- (13,-1) -- (13,2) -- (10,3) -- (8,5) -- (12,5) -- (12,9) -- (8,9) -- (10,11) -- (2,11) -- cycle;
  \draw (0,9) -- (8,9) (0,5) -- (8,5) (4,3) -- (10,3) (2,11) -- (2,3) (4,11) -- (4,3);
  \foreach \x in {(4,11), (4,3), (4,0), (13,2), (13,-1)} \filldraw[fill=white] \x circle (3pt);
  \foreach \x in {(2,11), (2,3), (10,11), (10,3), (10,0)} \filldraw[fill=black] \x circle (3pt);
  \foreach \x in {(0,9), (0,5), (8,9), (8,5), (12,9), (12,5)} \filldraw[fill=white] \x +(-3pt,-3pt) rectangle +(3pt,3pt);
%     %{\filldraw[fill=white] \x circle (3pt); \draw \x +(-3pt,0) -- + (3pt,0) +(0,3pt) -- +(0,-3pt); }
     \draw (10,7) node {$D$};
     \draw (11.5,1) node {$D'$};
     \draw (3,7) node {$E$};
     \draw (6,10) node {$C_1$};
     \draw (6,7) node {$C_3$};
     \draw (6,4) node {$C_2$};
     \draw (7,1.5) node {$R_4$};
 \end{tikzpicture}
\caption{Case 4.II) for surfaces in $\cH(2,1,1)$: $C_3,C_4$ not simple}
\label{fig:H211:C4II:no:sim}
\end{figure}
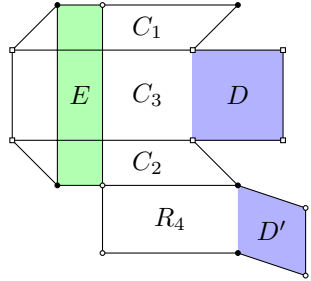

Using the arguments of \cite[Lem. 6.17]{AulicinoNguyenGen3TwoZeros}, we see that $\{C_1,C_2\}$ can be twisted simultaneously so that  there is a vertical cylinder $E$ contained in the union of $C_1,C_2,C_3$ crossing each of those cylinders once (see Figure~\ref{fig:H211:C4II:no:sim}). Since $C_4$ is $\cM$-parallel to $C_3$, it must be crossed by some cylinders in the equivalence class of $E$.
Consider a cylinder $E'$ in the equivalence class of $E$ which crosses $C_4$. Let $\gamma$ be a core curve of $E'$.
Let $n_i, \,  i=1,2,3$, be the number of intersections of $\gamma$ with a core curve of $C_i$, and $n_4$ be its number of intersections with the top side of $R_4$ ($n_4$ is not necessarily the number of intersections of $\gamma$ with a core curve of $C_4$).
Observe that we must have $n_1=\dots=n_4$. Let $h_i$ denote  the height of $C_i, \, i=1,2,3$, and $h_4$  the height of the rectangle $R_4$ (which is also the circumference of $D'$).
Denote by  $\cC$  the equivalence class $\{C_3,C_4\}$. By the Cylinder Proportion Lemma, we must have $ P(E,\cC)=P(E',\cC)$ which implies
$$
\frac{h_3}{h_1+h_2+h_3}=\frac{h_3+h_4}{h_1+h_2+h_3+h_4} \Leftrightarrow \frac{h_1+h_2}{h_3}=\frac{h_1+h_2}{h_3+h_4}.
$$
The last equation holds if and only if $h_4=0$ or $h_1 + h_2 = 0$.  In either case, we have a contradiction which proves the proposition.
\end{proof}

\begin{proposition}
\label{prop:C4II:H1111}
Let $\cM$ be  a rank two affine submanifold of $\cH(1^4)$. If $\cM$ contains a horizontally periodic surface $M$ satisfying Case 4.II) such that $M$ is $\cM$-cylindrically stable, then either $\cM=\dcoverprinc$ or $\cM=\prymprinc$.
\end{proposition}

\begin{proof}
Recall that by Lemma~\ref{lm:C4II:preliminary}, we have $k_3=k_4=k$.  If $C_3$ is a simple cylinder, then $C_4$ is as well.  In this case, by twisting so that neither cylinder contains a vertical saddle connection and performing an extended cylinder collapse as in \cite[Pf. of Lem. 4.7]{AulicinoNguyenGen3TwoZeros}, we get a translation surface satisfying Case 4.II) such that in the new surface each of  $C_3$ and $C_4$   contains at least one cylinder. Therefore we only need to consider the case $k \in \{1,2\}$.

\medskip

\noindent \ul{Case $k=1$.} Let $\sig_3$ (resp. $\sig_4$) be the unique saddle connection contained in both top and bottom of $C_3$ (resp. $C_4$). There is simple cylinder $D$ in $\ol{C}_3$ that contains $\sig_3$. We can assume that $M$ is square-tiled, and $D$ is vertical. Let $D'$ be the cylinder in $\ol{C}_4$ which is $\cM$-parallel to $D$.  Note that $D'$ is also a simple cylinder.

 We claim that $D$ and $D'$ are similar (proportional). If $D$ and $D'$ are not similar, then we can twist them so that one of them contains a horizontal saddle connection but the other does not. As  $D$ and $D'$ are collapsed simultaneously only one saddle connection is contracted to a point. Thus the resulting surface belongs to an rank two affine submanifold in $\cH(2,1^2)$. By construction, this new surface also admits a cylinder decomposition in Case 4.II) in the horizontal direction, but this contradicts Proposition~\ref{prop:H211NoCase4II}.

Since $D$ and $D'$ are proportional, we can collapse them simultaneously so that two saddle connections joining distinct simple zeros are contracted. The resulting surface, denoted by $M'$, belongs to a rank two  affine submanifold $\cM'$ in $\cH(2,2)$. By Proposition~\ref{prop:collapse:similar:cyl}, we have $\dim \cM=\dim \cM'+1$.  Note that the cylinders in $M'$ that correspond to $C_3$ and $C_4$ are simple.  By a slight abuse of notation, we will also denote them by $C_3$ and $C_4$, respectively.

 Since $M'$ admits a cylinder decomposition in Case 4.II), we derive that $M' \in \cH^{\rm odd}(2,2)$ (see~\cite[Sec. 6.3]{AulicinoNguyenGen3TwoZeros}). By the main  result of \cite{AulicinoNguyenGen3TwoZeros}, we know that $\cM' \in \{\dcoverodd, \allowbreak \prym\}$.  In both cases $C_3$ and $C_4$ are exchanged by the Prym involution of $M'$, thus they are isometric. It follows that the circumferences of $D$ and $D'$ are equal.  Since $D$ and $D'$ are similar, they are actually isometric. By Proposition~\ref{DblCovExtSimpCyl}, the Prym involution of $M'$ extends to an involution of $M$, that also exchanges $C_3$ and $C_4$. In particular, we see that $M\in \prymprinc$. Since the same arguments apply to the surfaces in a neighborhood of $M$ in $\cM$, we conclude that $\cM \subset \prymprinc$.

 If $\cM'=\prym$, then $\dim \cM=\dim\prym+1=6$ by Proposition~\ref{prop:collapse:similar:cyl}. Since $\dim \prymprinc =6$, we conclude that $\cM=\prymprinc$.

 If $\cM'=\dcoverodd$, then  $M'$ has a hyperelliptic involution $\iota$ that fixes $C_3$ and $C_4$. It is easy to check that $\iota$ preserves the saddle connection in $C_3$ (resp. in $C_4$) which is the degeneration of $D$ (resp. of $D'$) in $M'$. Thus $\iota$ extends to a hyperelliptic involution on $M$ (see Proposition~\ref{DblCovExtSimpCyl}).  Hence, $M \in \dcoverprinc$ by Lemma~\ref{lm:2inv:dblcover} and $\cM \subset\dcoverprinc$ by Proposition~\ref{prop:collapse:similar:cyl}.  Note that in this case we have $\allowbreak \dim \cM=\dim \dcoverodd +1=5=\dim\dcoverprinc$.  Thus $\cM$ must be the locus $\dcoverprinc$.

\medskip

\noindent \ul{Case $k=2$.}  Consider a  simple cylinder $D \subset \ol{C}_3$ that crosses the core curves of $C_3$ once. Let $\cD$ denote the equivalence class of $D$. Since $C_4$ is $\cM$-parallel to $C_3$, it must be crossed by a cylinder $D' \in \cD$. Since $D$ is disjoint from $C_1$ and $C_2$, so is $D'$, which means that $D'$ is contained in $\ol{C}_4$.

We can assume that $M$ is square-tiled and $D$ and $D'$ are vertical.  Since $C_3$ can be realized as a cylinder in a two-cylinder decomposition of a surface in the stratum $\cH(1,1)$, $\ol{C}_3$ contains at most one vertical cylinder. This implies that $D$ is the unique cylinder in $\cD$ that crosses $C_3$, because any other vertical cylinder that crosses $C_3$ would also cross $C_1$ or $C_2$ while $D$ does not.

We now claim that $D'$ is simple. To see this, we first remark that $C_4$ can be realized as a cylinder in a surface in $\cH(1,1)$. Thus $D'$ can be viewed as a cylinder in a translation surface  of genus two as well. Assume that $D'$ is not simple, then its closure contains a simple cylinder $E'$.
 There must exist a cylinder $E$  which is $\cM$-parallel to $E'$ and crosses $D$ (and hence $C_3$). Since $D$ is simple, $E$ cannot be contained in $D$. Now, since $D$ is the unique cylinder in $\cD$ that crosses $C_3$, we have $P(E,\cD) <1$. But by assumption, we have $P(E',\cD)=1$, therefore we get a contradiction to the Cylinder Proportion Lemma~\ref{CylinderPropProp} which proves the claim.

The remainder of the proof follows the same lines  as the previous case.
\end{proof}

\subsection{Case 4.I)}\label{sec:4cyl:C4I}
In this case the core curves of the cylinders cut the surface into two four-holed spheres.
Denote the horizontal cylinders of $M$ by $C_1,\dots,C_4$. For $i=1,\dots,4$, let $h_i$ and $\ell_i$ denote respectively the height and the circumference of $C_i$, and $\gamma_i$ be a core curve of $C_i$.
By assumption,  the following homological relation holds:
$$
\gamma_1 + \varepsilon_2 \gamma_2 + \varepsilon_3 \gamma_3 + \varepsilon_4 \gamma_4 = 0,
$$
where $\veps_i =\pm 1$. After possibly relabeling the cylinders and multiplying $\veps_i$ by $-1$, there are two distinct equations that are possible:
\begin{itemize}
\item Case 4.I.a) $\gamma_1 - \gamma_2 - \gamma_3 - \gamma_4 = 0,$ or
\item Case 4.I.b) $\gamma_1 + \gamma_2 - \gamma_3 - \gamma_4 = 0.$
\end{itemize}
We will analyze the cylinder diagrams according to the equation they satisfy.

\medskip

Let $\Gr$ be the embedded graph in $M$ whose vertices are the zeros and edges are the horizontal saddle connections. By assumption, $\Gr$ has two connected components denoted by $\Gr_1$ and $\Gr_2$.  Cutting $M$ along $\gamma_1,\dots,\gamma_4$, we obtain two four-holed spheres, which can be considered as regular neighborhoods of $\Gr_1$ and $\Gr_2$. It follows in particular that $\Gr_1$ and $\Gr_2$ are planar graphs. Observe also that any closed curve in $M$ cannot intersect $\gamma_1\cup\dots\cup\gamma_4$ only once. Therefore, none of $C_1,\dots,C_4$ contains a saddle connection in both its top and bottom.

Using the fact that $\Gr_1$ and $\Gr_2$ are planar, one can easily produce the list of admissible configurations for $\Gr_1$ and $\Gr_2$ together with the corresponding homological relation satisfied by $\gamma_1,\dots,\gamma_4$ (see Figure~\ref{fig:RibonGraph:C4I:ab}).
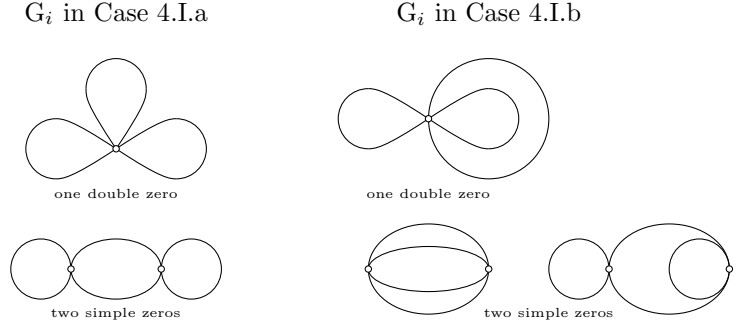
\begin{figure}[htb]
\begin{minipage}[t]{0.45\linewidth}
\centering
\begin{tikzpicture}[scale=0.4]
\draw (0,13.5) node {$\Gr_i$ in Case 4.I.a};

\draw (0,9) .. controls (-1,10.5) and (-1,10.7) .. (-1,11);
 \draw (0,9) .. controls (1,10.5) and (1,10.7) .. (1,11);
 \draw (1,11) arc (0:180:1);

 \draw (0,9) .. controls (-1.5,10) and (-1.8,10) .. (-2,10);
 \draw (0,9) .. controls (-1.5,8) and (-1.8,8) .. (-2,8);
 \draw (-2,10) arc (90:270:1);

 \draw (0,9) .. controls (1.5, 10) and (1.8,10) .. (2,10);
 \draw (0,9) .. controls (1.5,8) and (1.8,8) .. (2,8);
 \draw (2,8) arc (-90:90:1);

 \filldraw[fill=white] (0,9) circle (3pt);

 \draw (0,7.5) node {\tiny one double zero};

 \draw (0,5) ellipse (1.5 and 1);
 \draw (-2.5,5) circle (1);
 \draw (2.5,5) circle (1);
 \filldraw[fill=white] (-1.5,5) circle (3pt);
 \filldraw[fill=white] (1.5,5) circle (3pt);

\draw (0,3.5) node {\tiny two simple zeros};
\end{tikzpicture}
\end{minipage}
\begin{minipage}[t]{0.45\linewidth}
\centering
\begin{tikzpicture}[scale=0.4]
\draw (2,13.5) node {$\Gr_i$ in Case 4.I.b};

\draw (0,10) .. controls (-1.5,11) and (-1.8,11) .. (-2,11);
 \draw (0,10) .. controls (-1.5,9) and (-1.8,9) .. (-2,9);
 \draw (-2,11) arc (90:270:1);

 \draw (0,10) .. controls (1.5, 11) and (1.8,11) .. (2,11);
 \draw (0,10) .. controls (1.5,9) and (1.8,9) .. (2,9);
 \draw (2,9) arc (-90:90:1);

 \draw (2,10) circle (2);

 \filldraw[fill=white] (0,10) circle (3pt);

 \draw (0,7.5) node {\tiny one double zero};

 \draw (0,5) ellipse (2 and 0.75);
 \draw (0,5) ellipse (2 and 1.5);
 \filldraw[fill=white] (2,5) circle (3pt);
 \filldraw[fill=white] (-2,5) circle (3pt);

\draw (5,5) circle (1);
 \draw (8,5) ellipse (2 and 1.5);
 \draw (9,5) circle (1);
 \filldraw[fill=white] (6,5) circle (3pt);
 \filldraw[fill=white] (10,5) circle (3pt);

 \draw (4,3.5) node {\tiny two simple zeros};

\end{tikzpicture}
\end{minipage}
\caption{Configurations of $\Gr_i, \; i=1,2$, in Case 4.I}
\label{fig:RibonGraph:C4I:ab}
\end{figure}

\subsubsection{Case 4.I.a)}

In this case, we can assume without loss of generality that the top of $C_1$ is equal to the union of the bottoms of $C_2,C_3,C_4$.

\begin{lemma}\label{lm:C4Ia:ssim:cyl:no:par:C1}
Assume that $M$ is a horizontally periodic surface in $\cM$ satisfying Case 4.I.a) such that $M$ is $\cM$-cylindrically stable.  Then for $i=2,3,4$, if $C_i$ is semi-simple, then $C_i$ is not $\cM$-parallel to $C_1$.
\end{lemma}
\begin{proof}
Since $\cM$ is of rank two, the horizontal cylinders fall into at least two equivalence classes. Let $\cC$ denote the equivalence class of $C_1$. Observe that $\{C_2, C_3, C_4\}$ cannot be an equivalence class by the homological relation.

By contradiction, assume that $C_2$ is semi-simple and $\cM$-parallel to $C_1$. Since we have at least two equivalence classes of cylinders, neither $C_3$ nor $C_4$ is $\cM$-parallel to $C_1$, which means that $\cC=\{C_1,C_2\}$.
Since $C_2$ is semi-simple, we can assume that the bottom of $C_2$ consists of one saddle connection $\sig$. Let $\sig'$ be a saddle connection in the top of $C_2$. Note that $\sig$ and $\sig'$ are contained in the top and bottom of $C_1$ respectively. We can twist $C_1$ and $C_2$ such that any vertical ray entering $C_1$ through $\sig'$ crosses $\sig$. There exists in this case a transverse cylinder $D_1$, not necessarily  vertical, contained in $\ol{C}_1\cup\ol{C}_2$ whose core curves cross each of $\gamma_1,\gamma_2$ once. Twisting $\{C_1,C_2\}$ again,  we can assume that $D_1$ is vertical. Let $\cD$ denote the equivalence class of $D_1$, and assume that $\cD=\{D_1,\dots,D_s\}$.

We claim that $D_j$ is contained in $\ol{C}_1\cup\ol{C}_2$ for all $j=1,\dots,s$. This is a consequence of the Cylinder Proportion Lemma and the fact that
$ P(D_1,\cC)=1$. It follows that each  $D_j$ crosses $\gamma_1$ and $\gamma_2$ the same number of times $n_j$. Let $h'_j$ be the height of $D_j$, and $\ell_i$ be the circumference of $C_i$.  Applying the Cylinder Proportion Lemma, we have $P(C_1,\cD) = P(C_2,\cD)$, which is equivalent to
$$
\frac{n_1h'_1+\dots+n_sh'_s}{\ell_1}=\frac{n_1h'_1+\dots+n_sh'_s}{\ell_2} \Leftrightarrow \ell_1=\ell_2.
$$
However, this is impossible because $\ell_1=\ell_2+\ell_3+\ell_4$.
\end{proof}

\begin{proposition}\label{prop:C4Ia:H211}
 If $\cM$ is a rank two affine submanifold of $\cH(2,1^2)$, then $\cM$ does not contain an $\cM$-cylindrically stable horizontally periodic surface satisfying Case 4.I.a).
\end{proposition}
\begin{proof}
Assume to the contrary that $M$ is an $\cM$-cylindrically stable horizontally periodic surface in $\cM$ satisfying Case 4.I.a).  In $\cH(2,1^2)$, by inspection of the admissible configurations of the graphs $\Gr_1,\Gr_2$, we see that each of $\{C_2, C_3, C_4\}$ is semi-simple.  Lemma~\ref{lm:C4Ia:ssim:cyl:no:par:C1} establishes the existence of a free semi-simple cylinder in this case.  However, $M$ cannot have a free semi-simple cylinder because it could be twisted to contain a single vertical saddle connection between a double zero and a simple one, and hence could be collapsed to a translation surface in $\cH(3,1)$. But this contradicts the non-existence of a rank two affine manifold in that stratum.
\end{proof}

The following lemma follows from an inspection of the admissible configurations of the graphs $\Gr_1, \Gr_2$.

\begin{lemma}
\label{Case4IaPrinStrCD}
In the principal stratum in genus three, there are exactly two cylinder diagrams satisfying Case 4.I.a). They are depicted in Figure~\ref{Case4IaPrinStrCylDiagsFig}.
\begin{figure}[htb]
\begin{minipage}[t]{0.49\linewidth}
\centering
\begin{tikzpicture}[scale=0.4]
\draw (0,0)--(0,2)--(-2,2)--(-3,4)--(1,4)--(2,2)--(3,4)--(5,4)--(4,2)--(6,2)--(7,4)--(9,4)--(8,2)--(8,0) -- cycle;
\foreach \x in {(0,0),(0,2),(-2,2),(-3,4),(-1,4),(1,4),(2,2),(3,4),(5,4),(4,2),(6,2),(7,4),(9,4),(8,2),(8,0),(6,0),(4,0),(2,0)} \filldraw[fill=white] \x circle (3pt);

\draw (-2,4) node[above] {\tiny $1$} (0,4) node[above] {\tiny $2$} (4,4) node[above] {\tiny $3$} (5,2) node[above] {\tiny $4$} (8,4) node[above] {\tiny $5$} (-1,2) node[below] {\tiny $4$} (1,0) node[below] {\tiny $3$} (3,0) node[below] {\tiny $1$} (5,0) node[below] {\tiny $5$} (7,0) node[below] {\tiny $2$};

\draw (4,-2) node {(A)};

\end{tikzpicture}
\end{minipage}
\begin{minipage}[t]{0.49\linewidth}
\begin{tikzpicture}[scale=0.4]
\draw (0,0)--(0,2)--(-2,2)--(-3,4)--(1,4)--(2,2)--(3,4)--(7,4)--(6,2)--(8,2)--(9,4)--(11,4)--(10,2)--(10,0) -- cycle;
\foreach \x in {(0,0),(0,2),(-2,2),(-3,4),(1,4),(2,2),(3,4),(5,4),(7,4),(6,2),(8,2),(9,4),(11,4),(10,2),(10,0),(6,0),(4,0),(2,0)} \filldraw[fill=white] \x circle (3pt);

\draw (-1,4) node[above] {\tiny $1$} (4,4) node[above] {\tiny $2$} (6,4) node[above] {\tiny $3$} (7,2) node[above] {\tiny $4$} (10,4) node[above] {\tiny $5$} (-1,2) node[below] {\tiny $4$} (1,0) node[below] {\tiny $3$} (3,0) node[below] {\tiny $5$} (5,0) node[below] {\tiny $2$} (8,0) node[below] {\tiny $1$};

\draw (4,-2) node {(B)};
\end{tikzpicture}
\end{minipage}

\caption{The Two Cylinder Diagrams Satisfying Case 4.I.a) in $\cH(1^4)$}
\label{Case4IaPrinStrCylDiagsFig}
\end{figure}
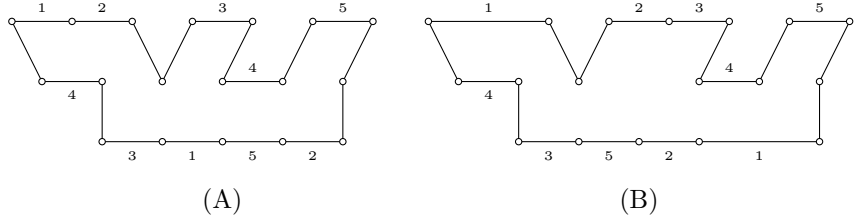
\end{lemma}

\begin{proposition}\label{prop:C4Ia:H1111}
 Let $\cM$ be a rank two affine submanifold of $\cH(1^4)$.  Assume that $\tilde \cQ(2,1,-1^3)$ is the only rank two affine manifold in $\cH(2,1^2)$. If $\cM$ contains an $\cM$-cylindrically stable horizontally periodic surface $M$ satisfying Case 4.I.a), then either  $\cM=\dcoverprinc$ or $\cM=\prymprinc$.
\end{proposition}
\begin{proof}
By Lemma \ref{Case4IaPrinStrCD}, there are two cases to consider.

\medskip
\noindent \ul{Case (A):} Denote the simple cylinders by $C_3$ and $C_4$. By Lemma~\ref{lm:C4Ia:ssim:cyl:no:par:C1}, neither of them is $\cM$-parallel to $C_1$. Therefore, either one of them, say $C_4$ is free, or they are $\cM$-parallel.
 Suppose to a contradiction that $C_4$ is free.  Collapse it so that two zeros in its boundary collide. The resulting surface $M'$ belongs to a rank two affine submanifold $\cM'$ of $\cH(2, 1^2)$. By assumption, $\cM'=\tilde{\cQ}(2,1,-1^3)$. In particular, $M'$ admits an involution $\inv$ with four fixed points whose derivative is $-\id$. Note that $M'$ has three horizontal cylinders. It is easy to see that none of them can be permuted with another one by $\inv$. Thus all three cylinders are invariant by $\inv$, which implies that $\inv$ has at least six fixed points in the interior of the cylinders. This contradiction means that $C_3$ and $C_4$ must be $\cM$-parallel.

 We claim that $C_3$ and $C_4$ are $\cM$-similar. Indeed, if they are not, then twist and collapse them such that only one pair of simple zeros in their boundaries collide. The resulting surface $M'$ belongs to a rank two affine submanifold of $\cH(2,1^2)$. By assumption, $M' \in \tilde \cQ(2,1,-1^3)$, thus $M'$ has an involution $\inv$ with four fixed points. Remark that $M'$ is horizontally periodic with two horizontal cylinders that we keep denoting by $C_1$ and $C_2$. Observe that $\inv$ must fix each of $C_1$ and $C_2$, hence $\inv$ has at least four fixed points in the interiors of $C_1$ and $C_2$. But the double zero of $M'$ must also be a fixed point of $\inv$. Thus, $\inv$ has at least five fixed points, and we have a contradiction which implies that $C_3$ and $C_4$ are $\cM$-similar.

 Twist and collapse $C_3$ and $C_4$ simultaneously such that the pairs of zeros in their boundaries collide, we get a surface $M'$ which is contained in a rank two affine submanifold $\cM'$ of $\cH(2,2)$ (by Proposition~\ref{prop:collapse:similar:cyl}).  For $i=3,4$, let $\sig_i$ denote saddle connection which the degeneration of $C_i$ on $M'$.

 By the results  of \cite{AulicinoNguyenGen3TwoZeros}, $M'$ admits a Prym involution $\inv$ with four fixed points. Since  $M'$ has two horizontal cylinders which cannot be exchanged by an involution, $\inv$ must fix each of these cylinders. Consequently, $\inv$ has four fixed points in the interiors of the cylinders.  It follows that the zeros of $M'$ are exchanged by $\inv$, which means that $M' \in \cH^{\rm odd}(2,2)$, and  hence $\cM'\in \allowbreak \{\dcoverodd,\prym\}$.

 Observe also that $\inv$ must exchange $\sig_3$ and $\sig_4$, otherwise $\tau$ would have more than four fixed points. Thus, $\tau$ extends to a Prym involution on $M$ by Proposition~\ref{DblCovExtSimpCyl}. Therefore, $M\in \prymprinc$. It follows from Proposition~\ref{prop:collapse:similar:cyl} that $\prymprinc$ contains a neighborhood of $M$ in $\cM$, hence $\cM \subseteq \prymprinc$.

 If $\cM'=\prym$, then $\dim\cM \allowbreak = \dim \prym+1 = \dim\prymprinc=6$. Thus, $\cM=\prymprinc$.

 If $\cM'=\dcoverodd$, then $M'$ has a hyperelliptic involution $\iota$. One can check that $\iota$ fixes each of $\sig_3$ and $\sig_4$, thus extends to a hyperelliptic involution of $M$. Therefore $M \in \cH(1^4)\cap \cP\cap\cL=\dcoverprinc$. Since in this case $\dim\cM=\dim\dcoverprinc=5$, we must have $\cM=\dcoverprinc$.

 \medskip

 \noindent \ul{Case (B):} Let $C_4$ be the unique simple cylinder. Since in this case all of the cylinders $C_2,C_3,C_4$ are semi-simple, none of them is $\cM$-parallel to $C_1$ by Lemma~\ref{lm:C4Ia:ssim:cyl:no:par:C1}. Since they cannot belong to the same equivalence class either, at least one of them is free.

 If $C_2$ or $C_3$ is free, then collapse it to obtain a surface $M'$ in $\cH(2,1^2)$. By \cite[Prop. 2.16]{AulicinoNguyenGen3TwoZeros}, $M'$ belongs to a rank two affine submanifold $\cM'$ of $\cH(2,1^2)$. By assumption, $\cM=\tilde{\cQ}(2,1,-1^3)$, which means that $M'$ admits a Prym involution $\inv$ with four fixed points. But such an involution must fix all three cylinders, which means that $\inv$ has at least six fixed points and we get a contradiction.

 It remains to consider the case $C_4$ is free. Collapsing it, we obtain a surface $M' \in \tilde{\cQ}(2,1,-1^3)$. Note that in this case the Prym involution $\inv$ of $M'$ fixes $C_1$, and permutes $C_2$ and $C_3$. In particular, $\inv$ leaves invariant the saddle connection which is the degeneration of $C_4$. By Proposition~\ref{DblCovExtSimpCyl}, $\inv$ extends to an involution of $M$ with four fixed points. Thus we have $\cM \subset \prymprinc$. Since we have $ \dim \cM=\dim \tilde{\cQ}(2,1,-1^3)+1=6=\dim\prymprinc$, it follows $\cM=\prymprinc$.
 The proof of the proposition is now complete.
\end{proof}

\subsubsection{Case 4.I.b)}
Recall that in this case we number the horizontal cylinders such that
\begin{equation}\label{eq:C4Ib:hom:rel}
\ell_1+\ell_2=\ell_3+\ell_4.
\end{equation}

We first observe
\begin{lemma}
\label{lm:Case4IbEqClasses}
Let $\cM$ be a rank two affine manifold in genus three and $M \in \cM$ an $\cM$-cylindrically stable horizontally periodic translation surface satisfying Case 4.I.b).  Then up to a renumbering of the cylinders respecting \eqref{eq:C4Ib:hom:rel} one of the following occurs:
\begin{itemize}
\item The equivalence classes are  $\{C_1, C_2\}$, $\{C_3\}$, $\{C_4\}$,
\item The equivalence classes are $\{C_1, C_3\}$, $\{C_2\}$, $\{C_4\}$,
\item The equivalence classes are $\{C_1, C_3\}$, $\{C_2, C_4\}$ and $\ell_1 = \ell_3$ and $\ell_2 = \ell_4$.
\end{itemize}
\end{lemma}
\begin{proof}
We first notice that the four cylinders cannot all be free since this would contradict the rank two hypothesis.
By the homological relation, there cannot be three cylinders in the same equivalence class because it would imply that all of the cylinders are $\cM$-parallel.  Similarly, if $C_1$ and $C_2$ are $\cM$-parallel, then the homological relation implies that each of $C_3$ and $C_4$ is free.

Finally, assume that the equivalence classes are $\{C_1, C_3\}$ and $\{C_2, C_4\}$.  Then there exist non-zero real numbers $\mu$ and $\lambda$ such that $\gamma_1 = \mu\gamma_3$ and $\gamma_2 = \lambda \gamma_4$.  Combining this with the homological relation yields
$$\mu\gamma_3 + \lambda\gamma_4 = \gamma_3 + \gamma_4.$$
This implies that unless $\mu = \lambda = 1$, there is only one equivalence class of cylinders, which would contradict $\cM$-cylindrical stability.  Furthermore, the relation $\mu = \lambda = 1$ implies that there are two pairs of cylinders with equal circumferences.
\end{proof}

The following lemma improves  Lemma~\ref{ANLem33Gen}.  Despite its rather technical statement, it will be useful for us in the sequel.

\begin{lemma}
 \label{lm:nonfree:sim:cyl}
 Let $\cM$ be a rank two affine manifold in genus three in a stratum with $k \geq 3$ zeros. Assume that every rank two affine manifold in genus three with at most $k-1$ zeros admits an involution with four fixed points whose derivative is $-\id$.  If $\cM$ contains a horizontally periodic surface such that one of the horizontal cylinders is simple and not free, then $\cM$ contains an $\cM$-cylindrically stable horizontally periodic surface $M$ satisfying one of the following:
 \begin{itemize}
  \item[(i)] There are three horizontal cylinders, two of which are simple and $\cM$-parallel to each other, and the cylinder decomposition satisfies Case 3.I),

  \item[(ii)] There are at least four horizontal cylinders three of which are $\cM$-parallel to one another,

  \item[(iii)] There are at least four horizontal cylinders, one of which is simple and not free.
 \end{itemize}

\end{lemma}
\begin{proof}
 Let $M$ be a horizontally periodic surface in $\cM$ with a non-free simple cylinder $C_1$. By \cite[Lem. 2.14]{AulicinoNguyenGen3TwoZeros}, we can suppose that $M$ is a square-tiled surface and $\cM$-cylindrically stable. Since  $M$ has at least two equivalence classes of horizontal cylinders, and $C_1$ is not free, we draw that $M$ has at least three horizontal cylinders.  If $M$ has four or more horizontal cylinders then we get the last  assertion.
 Assume from now on that $M$ contains exactly three horizontal cylinders.

 We first remark that the cylinder decomposition of $M$  does not satisfy Case 3.II) since in this case all three cylinders are free. It does not satisfy Case 3.III) either by
 \cite[Lem. 4.6]{AulicinoNguyenGen3TwoZeros}. Thus we have a cylinder decomposition in Case 3.I).

 Let $C_2$ be $\cM$-parallel to $C_1$. If $C_2$ is also simple,  by \cite[Lem. 2.11]{AulicinoNguyenGen3TwoZeros} and  \cite[Lem. 2.15]{AulicinoNguyenGen3TwoZeros}, we know that  $C_1$ and $C_2$ are $\cM$-parallel and the remaining cylinder is free. Therefore, we get the first assertion.

 Assume that $C_2$ is not simple. Let $C_3$ denote the remaining horizontal cylinder. Following the arguments in the proof of \cite[Prop. 5.6]{AulicinoNguyenGen3TwoZeros} we get two possibilities:
 \begin{itemize}
  \item[$\bullet$] If $C_1$ is only adjacent to $C_3$, then the conclusion is that we get an equivalence class $\cD$, with at least three vertical cylinders which do not fill $M$. Thus we have the second assertion.

  \item[$\bullet$] If $C_1$ is adjacent to both $C_2$ and $C_3$, then we have a contradiction.
 \end{itemize}
 The proof of the lemma is then complete.
\end{proof}

We also need the following

\begin{lemma}\label{lm:C4IV:nonfree:sim:cyl}
Let $M$ be an $\cM$-cylindrically stable horizontally periodic surface in $\cM$. If one of the horizontal cylinders of $M$ is simple and not free, then the cylinder decomposition of $M$ does not belong to Case 4.IV).
\end{lemma}
\begin{proof}
Assume that the cylinder decomposition of $M$ satisfies Case 4.IV). We label the horizontal cylinders by $C_1,\dots,C_4$, and let $\gamma_i$ be a (geodesic) core curve of $C_i$. Recall that in this case the family $\{\gamma_1,\dots,\gamma_4\}$ cuts $M$ into two three-holed spheres and a two-holed torus. We choose the numbering such that $\gamma_3\cup \gamma_4$ is the boundary of the two-holed torus. Observe that the following homological relations hold
$$
\gamma_1+\gamma_2=\gamma_3=\gamma_4.
$$

By cutting $M$ along $\gamma_3$ and $\gamma_4$, then exchanging the gluings, we get two translation surfaces of genus two, both of which are horizontally periodic. Observe that one of the two surfaces has a single horizontal cylinder, which is formed by one half of $C_3$ and one half of $C_4$. This observation allows us to conclude that neither  $C_3$ nor $C_4$ is simple.

By assumption,  either $C_1$ or $C_2$ is simple and not free. But from the homological relation, it can be easily seen that in either case, all four cylinders belong to the same equivalence class, which contradicts the $\cM$-cylindrical stability of $M$.
\end{proof}

\subsubsection{Case 4.I.b): The Stratum $\cH(2,1,1)$}

The following lemma follows from an inspection of admissible configurations of $\Gr_1$ and $\Gr_2$.
\begin{lemma}
\label{Case4IbH211CD}
There are exactly three cylinder diagrams up to symmetry satisfying Case 4.I.b) in $\cH(2,1^2)$ and they are depicted in Figure~\ref{Case4IbH211CylDiagsFig}.

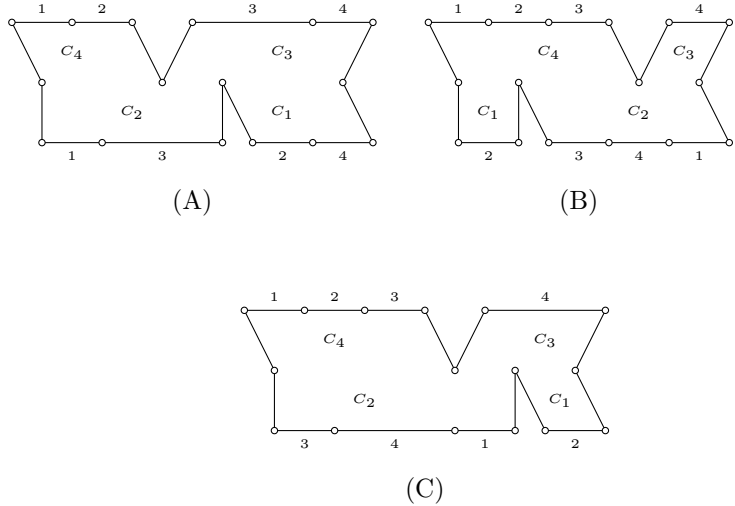
\begin{figure}[htb]
\begin{minipage}[t]{0.49\linewidth}
\centering
\begin{tikzpicture}[scale=0.4]
\draw (0,0)--(0,2)--(-1,4)--(3,4)--(4,2)--(5,4)--(11,4)--(10,2)--(11,0)--(7,0)--(6,2)--(6,0) -- cycle;
\foreach \x in {(0,0),(0,2),(-1,4),(1,4),(3,4),(4,2),(5,4),(9,4),(11,4),(10,2),(11,0),(9,0),(7,0),(6,2),(6,0),(2,0)} \filldraw[fill=white] \x circle (3pt);

\draw (0,4) node[above] {\tiny $1$} (2,4) node[above] {\tiny $2$} (7,4) node[above] {\tiny $3$} (10,4) node[above] {\tiny $4$} (1,0) node[below] {\tiny $1$} (4,0) node[below] {\tiny $3$} (8,0) node[below] {\tiny $2$} (10,0) node[below] {\tiny $4$};
\draw (1,3) node {\tiny $C_4$} (8,3) node {\tiny $C_3$} (3,1) node {\tiny $C_2$} (8,1) node {\tiny $C_1$};
\draw (5,-2) node {(A)};

\end{tikzpicture}
\end{minipage}
\begin{minipage}[t]{0.49\linewidth}
\begin{tikzpicture}[scale=0.4]
\draw (0,0)--(0,2)--(-1,4)--(5,4)--(6,2)--(7,4)--(9,4)--(8,2)--(9,0)--(3,0)--(2,2)--(2,0) -- cycle;
\foreach \x in {(0,0),(0,2),(-1,4),(1,4),(3,4),(5,4),(6,2),(7,4),(9,4),(8,2),(9,0),(7,0),(5,0),(3,0),(2,2),(2,0)} \filldraw[fill=white] \x circle (3pt);

\draw (0,4) node[above] {\tiny $1$} (2,4) node[above] {\tiny $2$} (4,4) node[above] {\tiny $3$} (8,4) node[above] {\tiny $4$} (1,0) node[below] {\tiny $2$} (4,0) node[below] {\tiny $3$} (6,0) node[below] {\tiny $4$} (8,0) node[below] {\tiny $1$};
\draw (3,3) node {\tiny $C_4$} (7.5,3) node {\tiny $C_3$} (1,1) node {\tiny $C_1$} (6,1) node {\tiny $C_2$};

\draw (4,-2) node {(B)};
\end{tikzpicture}
\end{minipage}

\vspace{0.75cm}

\begin{minipage}[t]{\linewidth}
\centering
\begin{tikzpicture}[scale=0.40]
\draw (0,0)--(0,2)--(-1,4)--(5,4)--(6,2)--(7,4)--(11,4)--(10,2)--(11,0)--(9,0)--(8,2)--(8,0) -- cycle;
\foreach \x in {(0,0),(0,2),(-1,4),(1,4),(3,4),(6,2),(5,4),(7,4),(11,4),(10,2),(11,0),(9,0),(8,2),(8,0),(6,0),(2,0)} \filldraw[fill=white] \x circle (3pt);

\draw (0,4) node[above] {\tiny $1$} (2,4) node[above] {\tiny $2$} (4,4) node[above] {\tiny $3$} (9,4) node[above] {\tiny $4$} (1,0) node[below] {\tiny $3$} (4,0) node[below] {\tiny $4$} (7,0) node[below] {\tiny $1$} (10,0) node[below] {\tiny $2$};
\draw (2,3) node {\tiny $C_4$} (9,3) node {\tiny $C_3$} (3,1) node {\tiny $C_2$} (9.5,1) node {\tiny $C_1$};
\draw (5,-2) node {(C)};
\end{tikzpicture}
\end{minipage}
\caption{The three cylinder diagrams satisfying Case 4.I.b) in $\cH(2,1,1)$}
\label{Case4IbH211CylDiagsFig}
\end{figure}
\end{lemma}

We will show
\begin{proposition}\label{prop:C4Ib:H211}
Let $\cM$ be a rank two affine submanifold of $\cH(2,1^2)$. If $\cM$ contains a horizontally periodic surface $M$ satisfying Case 4.I.b), then either $\cM$ contains a horizontally periodic surface with five cylinders, or $\cM=\tilde{\cQ}(2,1,-1^3)$.
\end{proposition}

We first prove

\begin{lemma}
\label{Case4IbH211FreeSSCyl}
If $\cM$ is a rank two affine manifold in $\cH(2,1^2)$ and $M \in \cM$ is horizontally periodic satisfying Case 4.I.b), then $M$ does not have a free semi-simple cylinder.
\end{lemma}

\begin{proof}
By contradiction, if $M$ has a free semi-simple cylinder $C$, then one boundary of $C$ contains a double zero and the other must contain one or more simple zeros.  Twist $C$ so that it admits a vertical saddle connection, which by necessity connects a double zero to a simple zero.  Collapsing $C$ results in a translation surface in a rank two affine manifold in $\cH(3,1)$.  Since no such affine manifold exists by \cite{AulicinoNguyenGen3TwoZeros}, we achieved the desired contradiction.
\end{proof}

\begin{lemma}
\label{Case4IbH211EqClsLem1}
Let $\cM$ be a rank two affine manifold in $\cH(2,1^2)$. If $M \in \cM$ is an $\cM$-cylindrically stable horizontally periodic satisfying Case 4.I.b), then neither  $\{C_1, C_2\}$ nor $\{C_3, C_4\}$ can be equivalence classes.
\end{lemma}
\begin{proof}
By contradiction, if either of them is an equivalence class, then Lemma \ref{Case4IbH211FreeSSCyl} implies that the other one must be an equivalence class because each of $\{C_1, C_2\}$ and $\{C_3, C_4\}$ contains a semi-simple cylinder.  But this contradicts Lemma \ref{lm:Case4IbEqClasses}.
\end{proof}

\begin{lemma}\label{lm:C4Ib:H211:C}
 Let $\cM$ be a rank two affine submanifold of $\cH(2,1^2)$. Then $\cM$ does not contain an $\cM$-cylindrically stable horizontally periodic surface $M$ satisfying Case 4.I.b) with cylinder diagram (C).
\end{lemma}
\begin{proof}
Assume to a contradiction that $\cM$ contains an $\cM$-cylindrically stable horizontally periodic surface satisfying Case 4.I.b) with cylinder diagram (C).  We claim that one of $C_1,C_3,C_4$ is free. Assume that none of them is free, by Lemma \ref{Case4IbH211EqClsLem1}, $C_1$ must be $\cM$-parallel to either $C_3$ or $C_4$ and either $\ell_1=\ell_3$ or $\ell_1=\ell_4$ by Lemma \ref{lm:Case4IbEqClasses}. But clearly, in this case we always have $\ell_1 < \ell_3$ and $\ell_1 < \ell_4$. If one of $C_1,C_3,C_4$ is free, then we get a contradiction to Lemma \ref{Case4IbH211FreeSSCyl}, and  the lemma follows.
\end{proof}

\begin{lemma}\label{lm:C4Ib:H211:B}
 Let $\cM$ be a rank two affine submanifold of $\cH(2,1^2)$. If $\cM$ contains an $\cM$-cylindrically stable horizontally periodic surface $M$ satisfying Case 4.I.b) with cylinder diagram (B), then $\cM=\tilde{\cQ}(2,1,-1^3)$.
\end{lemma}
\begin{proof}
 We number the cylinders so that $C_1$ and $C_3$ are the simple ones. By Lemma \ref{Case4IbH211FreeSSCyl}, neither $C_1$ nor $C_3$ are free. From Lemma \ref{lm:Case4IbEqClasses}, either $C_1$ is $\cM$-parallel to $C_3$ and $\ell_1=\ell_3$, or $C_1$ is $\cM$-parallel to $C_4$ and $\ell_1=\ell_4$. Since the latter cannot happen because $\ell_1 < \ell_4$, we conclude that the equivalence classes of horizontal cylinders are $\{C_1,C_3\}$ and $\{C_2,C_4\}$, and $\ell_1=\ell_3$ and $\ell_2=\ell_4$.

 We next claim that $C_1$ and $C_3$ are similar. If they are not, then after twisting, we can assume that there is a vertical saddle connection in $C_1$, but $C_3$ contains no vertical saddle connections. Collapsing simultaneously $C_1$ and $C_3$ yields a surface in $\cH(3,1)$. Since there are no rank two affine submanifolds in $\cH(3,1)$, we get a contradiction.

 The previous claim implies that $C_3=\lambda C_1$, where $\lambda > 0$. Since $\ell_1=\ell_3$, it follows that $\lambda=1$.  Hence, $C_1$ and $C_3$ are isometric. Collapse $C_1$ and $C_3$ simultaneously to get a surface $M' \in \cH(4)$. By Proposition~\ref{prop:collapse:similar:cyl}, $M'$ is contained in a rank two affine submanifold $\cM'$ of $\cH(4)$ such that $\dim \cM \allowbreak = \dim \cM'+1$.

 By the result of \cite{AulicinoNguyenWright}, $\cM'=\tilde{\cQ}(3,-1^3)$. Hence, $M'$ admits a Prym involution $\inv$. It is easy to check that this involution permutes the two horizontal cylinders of $M'$, and exchanges the saddle connections which are the degenerations of $C_1$ and $C_3$.  By Proposition \ref{DblCovExtSimpCyl}, $\inv$ gives rise to an involution of $M$ with four fixed points, which implies that $M \in \tilde{\cQ}(2,1,-1^3)$. Remark that this also holds for all of the surfaces in $\cM$ close to $M$, therefore $\cM\subset \tilde{\cQ}(2,1,-1^3)$. Finally, from the dimension count
 $$
 \dim \cM=\dim \tilde{\cQ}(3,-1^3)+1=5=\dim \tilde{\cQ}(2,1,-1^3),
 $$
 we conclude that $\cM=\tilde{\cQ}(2,1,-1^3)$.
\end{proof}

\begin{lemma}\label{lm:C4Ib:H211:A}
 Let $\cM$ be a rank two affine submanifold of $\cH(2,1^2)$. Assume that $\cM$ contains a horizontally periodic surface $M$ satisfying Case 4.I.b) with cylinder diagram (A), then either $\cM=\tilde{\cQ}(2,1,-1^3)$ or $\cM$ contains a horizontally periodic surface with five cylinders.
\end{lemma}
\begin{proof}
If $M$ is not $\cM$-cylindrically stable, then we conclude that there is a horizontally periodic surface in $\cM$ with at least five cylinders.  Otherwise, assume $M$ is $\cM$-cylindrically stable.
% Let us number the horizontal cylinders so that $C_1$ and $C_4$ are the semi-simple ones.
By Lemma~\ref{Case4IbH211FreeSSCyl}, neither $C_1$ nor $C_4$ is free. From Lemma~\ref{lm:Case4IbEqClasses}, we must have two equivalence classes $\{C_1,C_4\}$ and $\{C_2,C_3\}$ such that $\ell_1=\ell_4$ and $\ell_2=\ell_3$.
 Observe that there is a saddle connection in the bottom of $C_1$ and the top of $C_3$, and there is another saddle connection in the top of $C_1$ and the bottom of $C_3$. Since $C_1$ and $C_3$ are not $\cM$-parallel, after some twisting, we can assume that there is a vertical simple cylinder $D$ contained in $\ol{C}_1\cup \ol{C}_3$, which crosses each of $\gamma_1$ and $\gamma_3$ once.
 Note that $D$ must be $\cM$-parallel to another vertical cylinder crossing $C_2$ and $C_4$.

Applying Lemma~\ref{lm:nonfree:sim:cyl}, we derive that $\cM$ contains a horizontally periodic surface $M_1$ that is $\cM$-cylindrically stable, and one of the following occurs
\begin{itemize}

  \item[(i)] The cylinder decomposition of $M_1$ in the horizontal direction satisfies Case 3.I), and two of the cylinders are simple. In this case, we use Proposition~\ref{prop:C3I:2sim:cyl:classify} to conclude that $\cM=\tilde{\cQ}(2,1,-1^3)$.

  \item[(ii)] $M_1$ has at least four horizontal cylinders and three of which are $\cM$-parallel. Assume that $M_1$ has exactly four horizontal cylinders. If the cylinder decomposition satisfies Case 4.I) or Case 4.IV), then we only have one equivalence class, which contradicts the $\cM$-cylindrically stable hypothesis.  Case 4.II) is ruled out by Proposition~\ref{prop:H211NoCase4II}, and Case 4.III) is also ruled out by Proposition~\ref{Case4IIIProp}. Thus in this case $M_1$ has at least five horizontal cylinders.

  \item[(iii)] $M_1$ has at least four horizontal cylinders, one of which is simple and not free. We only need to consider the case where $M_1$ has exactly four horizontal cylinders. Again Case 4.II) and Case 4.III) are ruled out by Proposition~\ref{prop:H211NoCase4II} and Proposition~\ref{Case4IIIProp}.  Case 4.IV)  is ruled out by Lemma~\ref{lm:C4IV:nonfree:sim:cyl}.
      In Case 4.I.a), we conclude by Proposition~\ref{prop:C4Ia:H211}.  Finally, in Case 4.I.b), since one of the cylinders is simple, we must have Diagram (B) or (C).  Thus by Lemma~\ref{lm:C4Ib:H211:B} or Lemma~\ref{lm:C4Ib:H211:C}, we can conclude that $\cM=\tilde{\cQ}(2,1,-1^3)$.
 \end{itemize}
\end{proof}

\subsubsection*{Proof of Proposition~\ref{prop:C4Ib:H211}}
\begin{proof}
Proposition~\ref{prop:C4Ib:H211} is a direct consequence of Lemmas~\ref{Case4IbH211CD}, \ref{lm:C4Ib:H211:C}, \ref{lm:C4Ib:H211:B}, and \ref{lm:C4Ib:H211:A}.
\end{proof}

\subsubsection{Case 4.I.b): The Principal Stratum $\cH(1,1,1,1)$}
The following lemma is obtained from a careful inspection of admissible configurations for the graphs $\Gr_1$ and $\Gr_2$.
\begin{lemma}\label{lm:C4Ib:H1111:CylDiags}
There are four diagrams for cylinder decompositions in Case 4.I.b) in the stratum $\cH(1^4)$. They are shown in Figure~\ref{fig:Case4IbH1111CylDiagsFig}.
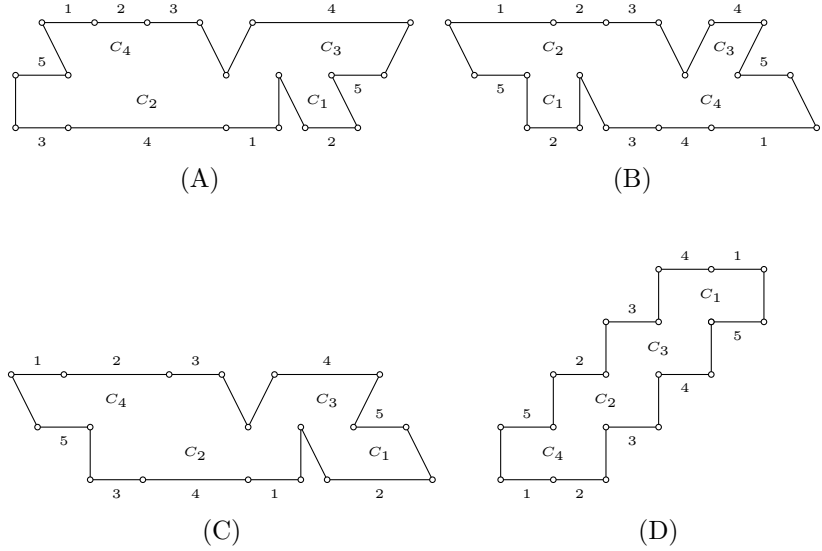
\begin{figure}[htb]
\begin{minipage}[t]{0.45\linewidth}
\centering
\begin{tikzpicture}[scale=0.35]
\draw (-2,0)--(-2,2)--(0,2)--(-1,4)--(5,4)--(6,2)--(7,4)--(13,4)--(12,2)--(10,2)--(11,0)--(9,0)--(8,2)--(8,0) -- cycle;
\foreach \x in {(-2,0),(-2,2),(0,0),(0,2),(-1,4),(1,4),(3,4),(6,2),(5,4),(7,4),(13,4),(12,2),(10,2),(11,0),(9,0),(8,2),(8,0),(6,0)} \filldraw[fill=white] \x circle (3pt);
\draw (0,4) node[above] {\tiny $1$} (2,4) node[above] {\tiny $2$} (4,4) node[above] {\tiny $3$} (10,4) node[above] {\tiny $4$} (-1,0) node[below] {\tiny $3$} (3,0) node[below] {\tiny $4$} (7,0) node[below] {\tiny $1$} (10,0) node[below] {\tiny $2$} (-1,2) node[above] {\tiny $5$} (11,2) node[below] {\tiny $5$};
\draw (2,3) node {\tiny $C_4$} (10,3) node {\tiny $C_3$} (3,1) node {\tiny $C_2$} (9.5,1) node {\tiny $C_1$};

\draw (5,-2) node {(A)};
\end{tikzpicture}
\end{minipage}
\begin{minipage}[t]{0.45\linewidth}
\centering
\begin{tikzpicture}[scale=0.35]
\draw (0,0)--(0,2)--(-2,2)--(-3,4)--(5,4)--(6,2)--(7,4)--(9,4)--(8,2)--(10,2)--(11,0)--(3,0)--(2,2)--(2,0) -- cycle;
\foreach \x in {(0,0),(0,2),(-2,2),(-3,4),(1,4),(3,4),(5,4),(6,2),(7,4),(9,4),(8,2),(10,2),(11,0),(7,0),(5,0),(3,0),(2,2),(2,0)} \filldraw[fill=white] \x circle (3pt);
\draw (-1,4) node[above] {\tiny $1$} (2,4) node[above] {\tiny $2$} (4,4) node[above] {\tiny $3$} (8,4) node[above] {\tiny $4$} (1,0) node[below] {\tiny $2$} (4,0) node[below] {\tiny $3$} (6,0) node[below] {\tiny $4$} (9,0) node[below] {\tiny $1$} (9,2) node[above] {\tiny $5$} (-1,2) node[below] {\tiny $5$};
\draw (1,3) node {\tiny $C_2$} (7.5,3) node {\tiny $C_3$} (1,1) node {\tiny $C_1$} (7,1) node {\tiny $C_4$};
\draw (4,-2) node {(B)};
\end{tikzpicture}
\end{minipage}

\vspace{0.5cm}

\begin{minipage}[t]{0.45\linewidth}
\centering
\begin{tikzpicture}[scale=0.35]
\draw (0,0)--(0,2)--(-2,2)--(-3,4)--(5,4)--(6,2)--(7,4)--(11,4)--(10,2)--(12,2)--(13,0)--(9,0)--(8,2)--(8,0) -- cycle;
\foreach \x in {(0,0),(0,2),(-2,2),(-3,4),(-1,4),(3,4),(6,2),(5,4),(7,4),(11,4),(10,2),(12,2),(13,0),(9,0),(8,2),(8,0),(6,0),(2,0)} \filldraw[fill=white] \x circle (3pt);
\draw (-2,4) node[above] {\tiny $1$} (1,4) node[above] {\tiny $2$} (4,4) node[above] {\tiny $3$} (9,4) node[above] {\tiny $4$} (1,0) node[below] {\tiny $3$} (4,0) node[below] {\tiny $4$} (7,0) node[below] {\tiny $1$} (11,0) node[below] {\tiny $2$} (11,2) node[above] {\tiny $5$} (-1,2) node[below] {\tiny $5$};
\draw (1,3) node {\tiny $C_4$} (9,3) node {\tiny $C_3$} (4,1) node {\tiny $C_2$} (11,1) node {\tiny $C_1$};

\draw (5,-2) node {(C)};
\end{tikzpicture}
\end{minipage}
\begin{minipage}[t]{0.45\linewidth}
\centering
\begin{tikzpicture}[scale=0.35]
\draw (0,0)--(0,2)--(2,2)--(2,4)--(4,4)--(4,6)--(6,6)--(6,8)--(10,8)--(10,6)--(8,6)--(8,4)--(6,4)--(6,2)--(4,2)--(4,0)--cycle;
\foreach \x in {(0,0),(0,2),(2,2),(2,4),(4,4),(4,6),(6,6),(6,8),(10,8),(10,6),(8,6),(8,8),(8,6),(8,4),(6,4),(6,2),(4,2),(4,0),(2,0)} \filldraw[fill=white] \x circle (3pt);
\draw(2,1) node {\tiny $C_4$};
\draw(4,3) node {\tiny $C_2$};
\draw(6,5) node {\tiny $C_3$};
\draw(8,7) node {\tiny $C_1$};

\draw(1,2) node[above] {\tiny 5};
\draw(3,4) node[above] {\tiny 2};
\draw(5,6) node[above] {\tiny 3};
\draw(7,8) node[above] {\tiny 4};
\draw(9,8) node[above] {\tiny 1};
\draw(1,0) node[below] {\tiny 1};
\draw(3,0) node[below] {\tiny 2};
\draw(5,2) node[below] {\tiny 3};
\draw(7,4) node[below] {\tiny 4};
\draw(9,6) node[below] {\tiny 5};

\draw (6,-2) node {(D)};
\end{tikzpicture}
\end{minipage}

\caption{Cylinder Diagrams Satisfying Case 4.I.b) in $\cH(1,1,1,1)$}
\label{fig:Case4IbH1111CylDiagsFig}
\end{figure}
\end{lemma}

\begin{lemma}\label{lm:C4Ib:H1111:A}
 Assume that $\cM$ contains an $\cM$-cylindrically stable horizontally periodic surface $M$ satisfying Case 4.I.b) with cylinder diagram (A).  Then $\cM=\prymprinc$.
\end{lemma}
\begin{proof}
 We number the cylinders so that $C_1$ is the simple one, and $C_3$ and $C_4$ are the semi-simple ones. Recall that the relation \eqref{eq:C4Ib:hom:rel} always holds. We claim that neither $C_3$ nor $C_4$ is free. Suppose that $C_3$ is free. Since it is semi-simple, we can collapse it to get a surface $M'\in \cH(2,1^2)$. By assumption, $M'$ must belong to $\tilde{\cQ}(2,1,-1^3)$, hence it admits a Prym involution $\inv$ with four fixed points. Observe that $M'$ has three horizontal cylinders, and none of them can be permuted with another one by $\inv$. Thus $\inv$ must fix all three cylinders, hence it must have at least six fixed points and we get a contradiction. The same arguments apply if $C_4$ is free.

 Using Lemma~\ref{lm:Case4IbEqClasses}, we derive that $C_1$ must be free and $\{C_3,C_4\}$ is an equivalence class. Collapsing $C_1$ so that the two zeros in its boundary collide, we obtain a surface $M' \in \tilde{\cQ}(2,1,-1^3)$. In particular, $M'$ admits a Prym involution $\inv$.
 Since $\inv$ has four fixed points, it must fix $C_2$ and exchange $C_3$ and $C_4$. In particular, it fixes the saddle connection which is the degeneration of $C_1$. Therefore, $\inv$ extends to a Prym involution  of $M$ that fixes $C_1$ by Proposition~\ref{DblCovExtSimpCyl}. It follows that $M \in \prymprinc$, and $\cM\subset \prymprinc$ by Proposition~\ref{prop:collapse:free:sim:cyl}. Since we have $ \dim \cM = \dim \tilde{\cQ}(2,1,-1^3)+1=6=\dim\prymprinc$, it follows that $\cM=\prymprinc$.
\end{proof}

\begin{lemma}\label{lm:C4Ib:H1111:B}
 Assume that $\cM$ contains an $\cM$-cylindrically stable horizontally periodic surface $M$ satisfying Case 4.I.b) with cylinder diagram (B).  Then $\cM\in\{\dcoverprinc,\prymprinc\}$.
\end{lemma}
\begin{proof}
 We number the cylinders so that $C_1$ and $C_3$ are the simple ones, and $C_1$ is adjacent to $C_2$. Assume that $C_1$ is free. We can collapse it to get a surface $M'\in \tilde{\cQ}(2,1,-1^3)$ with three horizontal cylinders.  Remark that $C_1$ degenerates to a saddle connection contained in both top and bottom of $C_2$.

 Since $C_3$ is the unique horizontal simple cylinder in $M'$, it must be fixed by $\inv$. Recall that $\inv$ has four fixed points, hence it must exchange $C_2$ and $C_4$. But this is impossible since there are no saddle connections that are contained in both top and bottom of $C_4$. The same arguments apply for the case $C_3$ is free. Thus we can conclude that neither $C_1$ nor $C_3$ is free.

 By Lemma~\ref{lm:Case4IbEqClasses}, we derive that $C_1$ and $C_3$ are $\cM$-parallel. Let $\cC$ denote the equivalence class $\{C_1,C_3\}$.

 We now claim that $C_2$ and $C_4$ are not free. Assume that $C_2$ is free which means that $C_4$ is also free. Observe that we can twist $\cC$ and $C_2$ such that there is a vertical cylinder $D$  contained in $\ol{C_1}\cup\ol{C}_2$.
 Since any other vertical cylinder crossing $C_2$ must cross $C_4$, we derive that $D$ is free. But this contradicts the Cylinder Proportion Lemma, since we have $ P(C_1,\{D\})=1$ but
  $ P(C_3,\{D\})=0$. Therefore, we can conclude that $C_2$ and $C_4$ are $\cM$-parallel. Using again Lemma~\ref{lm:Case4IbEqClasses}, we draw that $\ell_1=\ell_3$ and $\ell_2=\ell_4$.

  We next claim that $C_1$ and $C_3$ are similar. If they are not, then we can twist them so that $C_1$ contains a vertical saddle connection, but $C_3$ does not. Collapsing simultaneously $C_1$ and $C_3$ we get a surface $M' \in \tilde{\cQ}(2,1 -1^3)$ with two horizontal cylinders. By counting the number of saddle connections on the borders of these two cylinders, we see that they cannot be exchanged by the Prym involution $\inv'$ of $M'$. Thus, they are both fixed by $\inv'$, which implies that $\inv'$ has four regular fixed points in $M'$. Since the double zero of $M'$ must be a fixed point of $\inv'$, we derive that $\inv'$ has at least $5$ fixed points which is a contradiction.

  Since $C_1$ and $C_2$ are similar and $\ell_1=\ell_3$, we conclude that $C_1$ and $C_3$ are isometric. Collapsing $C_1$ and $C_3$ simultaneously yields a surface $M' \in \cH(2,2)$, which is contained in a rank two affine submanifold $\cM'$. Again, let $\inv'$ be the Prym involution of $M'$, and   let $\sig_1$ and $\sig_3$ be respectively the saddle connections which are the degenerations of $C_1$ and $C_3$ in $M'$. Note that $\sig_1$ (resp. $\sig_3$) is contained in both top and bottom of $C_2$ (resp. $C_4$).

  We claim that $\inv'$ exchanges $C_2$ and $C_4$. If $\inv'$ fixes $C_2$, then it also fixes $C_4$, therefore it has four regular fixed points in $M'$. Moreover, since $\sig_1$ is the unique saddle connection contained in both top and bottom of $C_2$, it must be invariant by $\inv'$. But $\sig_1$ connects a zero of $M'$ to itself, therefore $\inv'$ fixes a zero of $M'$. This contradicts the condition that $\inv'$ has exactly four fixed points.

  Since $\inv'$ exchanges  $C_2$ and $C_4$, it must exchange $\sig_1$ and $\sig_3$ and permute the zeros of $M'$. We derive in particular that $M'\in \cH^{\rm odd}(2,2)$. Thus $\cM'=\dcoverodd$ or $\cM'=\prym$. It follows from Proposition~\ref{DblCovExtSimpCyl} that $\inv'$ extends to a Prym involution of $M$. Thus $M \in \prymprinc$ and $\cM \subseteq \prymprinc$.  If $\cM'=\dcoverodd$, then $M'$ also admits a hyperelliptic involution, which also extends to $M$. Hence in this case, we have $\cM \subseteq \dcoverprinc$.

  By Proposition~\ref{prop:collapse:similar:cyl}, we know that $\allowbreak \dim \cM  = \dim \cM'+1$. Using this dimension relation, we conclude that if $\cM'=\dcoverodd$, then $\cM=\dcoverprinc$, and if $\cM'=\prym$, then $\cM=\prymprinc$.
\end{proof}

\begin{lemma}\label{lm:C4Ib:H1111:C:D}
 Assume that $\cM$ contains a horizontally periodic surface $M$ satisfying Case 4.I.b) with cylinder diagram (C) or (D). Then either $\cM \in \{\dcoverprinc, \prymprinc\}$, or $\cM$ contains a horizontally periodic surface with at least five horizontal cylinders.
\end{lemma}
\begin{proof}
It suffices to assume that $M$ is $\cM$-cylindrically stable, otherwise, we conclude that $\cM$ contains a horizontally periodic surface with at least five cylinders. Using Lemma~\ref{lm:Case4IbEqClasses}, one can check that there always exists a pair of cylinders $C_i$ and $C_j$ which are not $\cM$-parallel such that
 \begin{itemize}
  \item[$\bullet$] There is a saddle connection $\sig$ in the bottom of $C_i$ and in the top of $C_j$,

  \item[$\bullet$] There is a saddle connection $\sig'$ in the top of $C_i$ and in the bottom of $C_j$,

  \item[$\bullet$] $C_i$ is $\cM$-parallel to another cylinder.
 \end{itemize}
Since $C_i$ and $C_j$ are not $\cM$-parallel, we can twist them so that there is a vertical simple cylinder $D$ contained in $\ol{C}_i\cup \ol{C}_j$ which crosses only $\sig$ and $\sig'$. Since $C_i$ is $\cM$-parallel to another cylinder, $D$ is not free. Applying Lemma~\ref{lm:nonfree:sim:cyl}, we get a horizontally periodic $\cM$-cylindrically stable surface $M_1\in \cM$, and one of the following occurs:
\begin{itemize}
 \item[(i)] There are three horizontal cylinders, two of which are simple, and the cylinder diagram satisfies Case 3.I). In this case we conclude by Proposition~\ref{prop:C3I:2sim:cyl:classify}.

 \item[(ii)] There are at least four cylinders, and one of the equivalence classes consists of at least three cylinders. If $M_1$ has five horizontal cylinders or more, we are done. Assume that $M_1$ has exactly $4$ cylinders.  By the homological relations and $\cM$-cylindrical stability, Case 4.IV) and Case 4.I) are ruled out. If the cylinder decomposition satisfies Case 4.II) or Case 4.III), then we conclude by Proposition~\ref{prop:C4II:H1111} or Proposition~\ref{Case4IIIProp}, respectively.

 \item[(iii)] There are at least four horizontal cylinders, one of which is simple and not free. Obviously, we only need to consider the case $M_1$ has exactly $4$ cylinders. Case 4.IV) is then ruled out by Lemma~\ref{lm:C4IV:nonfree:sim:cyl}. In Case 4.II) and 4.III), we conclude by  Proposition~\ref{prop:C4II:H1111} and Proposition~\ref{Case4IIIProp}, respectively. In Case 4.I.a), we conclude by Proposition~\ref{prop:C4Ia:H1111}. Finally, in Case 4.I.b), since there exists a simple cylinder, we must have diagrams (A) or (B), and we can use Lemma~\ref{lm:C4Ib:H1111:A} or Lemma~\ref{lm:C4Ib:H1111:B} to conclude.
\end{itemize}
\end{proof}

As  a direct consequence of Lemmas~\ref{lm:C4Ib:H1111:CylDiags}, \ref{lm:C4Ib:H1111:A}, \ref{lm:C4Ib:H1111:B}, and \ref{lm:C4Ib:H1111:C:D}, we get

\begin{proposition}\label{prop:C4Ib:H1111}
 Suppose that $\tilde{\cQ}(2,1,-1^3)$ is the unique rank two affine submanifold in $\cH(2,1^2)$.
 Let $\cM$ be a rank two affine submanifold of $\cH(1^4)$. Assume that $\cM$ contains a horizontally periodic surface $M$ satisfying Case 4.I.b).
 Then either $\cM$ contains a horizontally periodic surface with at least five cylinders, or $\cM \in \{\dcoverprinc, \prymprinc\}$.
\end{proposition}

%\subsubsection*{Proof of Proposition~\ref{prop:C4Ib:H1111}}

\section{Five Cylinders}\label{sec:5cyl}

\begin{lemma}
\label{5CylDeg}
If a horizontally periodic genus three translation surface $M$ decomposes into exactly five cylinders, then pinching the core curves of those cylinders degenerates the surface to one of two possible surfaces:
\begin{itemize}
\item 5.I) 	Three spheres where two spheres have a pair of simple poles between them and the third sphere has two pairs of simple poles joined to each of the other two spheres.
\item 5.II) Three spheres where two spheres have three simple poles and the third sphere carries a pair of simple poles.
\end{itemize}
\end{lemma}

\begin{proof}
Let $X'$ denote the degenerate Riemann surface.  We use the classical terminology \emph{part} to mean a connected component of a degenerate Riemann surface from which the nodes have been removed.  Observe that a degenerate Riemann surface with $p$ parts imposes $p-1$ homological relations on the core curves of parallel cylinders.  In particular, there are no homological relations among the core curves of parallel cylinders on a degenerate Riemann surface with one part.  Thus, if $X'$ has one part and $M$ consisted of five cylinders, $M$ would have to have genus at least five.  Likewise, if $X'$ has two parts and $M$ has five cylinders, then $M$ would have to have genus at least four.

In genus three the degenerate surface $X'$ can never have more than four parts, which is given by the general upper bound $2(g-1)$.

If $X'$ has four parts and at least one part has positive genus, then the original surface would have genus at least four.  This can be seen by replacing the part with positive genus with a sphere with a corresponding number of poles and observing that it arises from a surface with at least six cylinders and such a configuration can never occur in genus three.  Thus, if $X'$ has four parts, then all four parts have genus zero.  However, every sphere must carry a meromorphic differential with at least three simple poles, and this would require at least six cylinders.  Hence, $X'$ has exactly three parts.

Finally, if $X'$ has three parts, we claim that no part has positive genus because again each part of genus $g'$ can be replaced by a sphere with $g'$ pairs of simple poles.  Since each pair of poles corresponds to a pinched cylinder and every six cylinder surface in genus three degenerates to a punctured Riemann surface consisting of exactly four spheres, all three parts of $X'$ must be spheres.  Recalling that every sphere must carry a differential with at least three simple poles, we leave the reader to deduce that there are exactly two possibilities.
\end{proof}

As usual, we denote by $C_1,\dots,C_5$ the horizontal cylinders of $M$, and for $i=1,\dots,5$,  $\gamma_i$ is a core curve of $C_i$. We choose the orientation of $\gamma_i$ to be from the left to the right.

\subsection{Case 5.I)}

In Case 5.I) there is a unique cylinder between the spheres with three simple poles.  Throughout this subsection we call that cylinder $C_1$. We  choose a numbering of the cylinders such that the following homological relations hold
\begin{equation}\label{eq:C5I:homo:rel}
\gamma_1=\eps_2\gamma_2+\eps_3\gamma_3=\eps_4\gamma_4+\eps_5\gamma_5
\end{equation}
where $\eps_i \in \{\pm 1\}$.

Let us denote by $x_1$ and $x_2$ the two simple zeros in the spheres with three simple poles. If $M\in \cH(2,1^2)$, we denote by $x_0$ the double zero, and if $M\in \cH(1^4)$ we denote the two simple zeros on the sphere with four simple poles by $x'_0$ and $x''_0$. For $i=1,2$, we denote by $\Gr_i$ the graph which is the union of horizontal saddle connections containing $x_i$. We denote by $\Gr_0$ the graph consisting of horizontal saddle connections in the sphere with four simple poles. Note that by assumption,  the graphs $\Gr_i, \, i=0,1,2$, are {\em  planar}. The admissible configurations of $\Gr_i$ are shown in Figure~\ref{fig:C5I:sc:graph}.

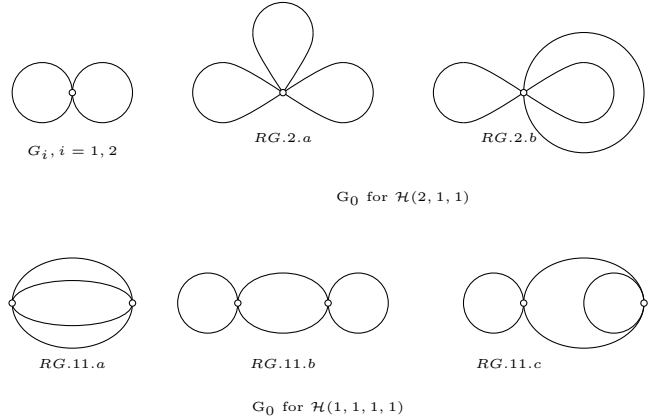
\begin{figure}[htb]
\centering
\begin{tikzpicture}[scale=0.4]
 \draw (0,12) arc (0:360:1);
 \draw (0,12) arc (180:540:1);
 \filldraw[fill=white] (0,12) circle (3pt);
 \draw (0,10) node {\tiny $G_i,i=1,2$};

 \draw (7,12) .. controls (6,13.5) and (6,13.7) .. (6,14);
 \draw (7,12) .. controls (8,13.5) and (8,13.7) .. (8,14);
 \draw (8,14) arc (0:180:1);

 \draw (7,12) .. controls (5.5,13) and (5.2,13) .. (5,13);
 \draw (7,12) .. controls (5.5,11) and (5.2,11) .. (5,11);
 \draw (5,13) arc (90:270:1);

 \draw (7,12) .. controls (8.5, 13) and (8.8,13) .. (9,13);
 \draw (7,12) .. controls (8.5,11) and (8.8,11) .. (9,11);
 \draw (9,11) arc (-90:90:1);

 \filldraw[fill=white] (7,12) circle (3pt);

 \draw (15,12) .. controls (13.5,13) and (13.2,13) .. (13,13);
 \draw (15,12) .. controls (13.5,11) and (13.2,11) .. (13,11);
 \draw (13,13) arc (90:270:1);

 \draw (15,12) .. controls (16.5, 13) and (16.8,13) .. (17,13);
 \draw (15,12) .. controls (16.5,11) and (16.8,11) .. (17,11);
 \draw (17,11) arc (-90:90:1);

 \draw (17,12) circle (2);

 \filldraw[fill=white] (15,12) circle (3pt);

 \draw (7,10.5) node {\tiny $RG.2.a$};
 \draw (14.5,10.5) node {\tiny $RG.2.b$};
 \draw (11,8.5) node {\tiny $\Gr_0$ for $\cH(2,1,1)$};

 \draw (0,5) ellipse (2 and 0.75);
 \draw (0,5) ellipse (2 and 1.5);
 \filldraw[fill=white] (2,5) circle (3pt);
 \filldraw[fill=white] (-2,5) circle (3pt);

 \draw (7,5) ellipse (1.5 and 1);
 \draw (4.5,5) circle (1);
 \draw (9.5,5) circle (1);
 \filldraw[fill=white] (5.5,5) circle (3pt);
 \filldraw[fill=white] (8.5,5) circle (3pt);

 \draw (14,5) circle (1);
 \draw (17,5) ellipse (2 and 1.5);
 \draw (18,5) circle (1);
 \filldraw[fill=white] (15,5) circle (3pt);
 \filldraw[fill=white] (19,5) circle (3pt);

 \draw (0,3) node {\tiny $RG.11.a$};
 \draw (7,3) node {\tiny $RG.11.b$};
 \draw (14.5,3) node {\tiny $RG.11.c$};
 \draw (8.5,1.5) node {\tiny $\Gr_0$ for $\cH(1,1,1,1)$};
\end{tikzpicture}
\caption{Case 5.I): admissible configurations of $\Gr_i, \, i=0,1,2$}
 \label{fig:C5I:sc:graph}
\end{figure}
Recall that in the literature the union $\Gr:=\sqcup_{i=0}^2\Gr_i$ is called the {\em separatrix diagram} of $M$, and in particular has a {\em ribbon structure}
(see \cite[Sec. 4]{KontsevichZorichConnComps}).
 Let $U_i$ be a regular neighborhood of $\Gr_i$ in the plane. We fix the orientation of every edge of $\Gr_i$ to be from left to right. Each  component of $\partial U_i$ is a core curve of a horizontal cylinder which is freely homotopic to a union of edges of $\Gr_i$.

A component of $\partial U_i$ is said to be {\em simple} if it is (freely) homotopic to a single edge of $\Gr_i$, which must be loop.  By definition, the cylinders that contain a simple component of $\partial U_i$ are semi-simple.

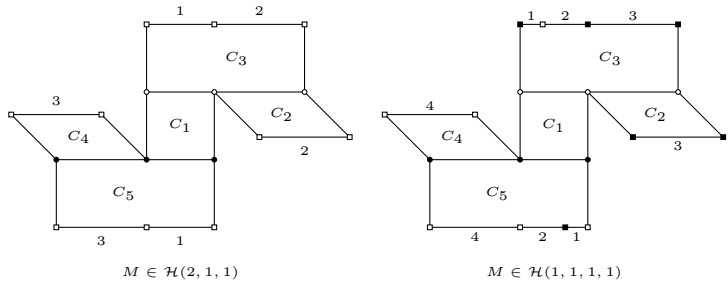
\begin{figure}[htb]
\begin{minipage}[t]{0.4\linewidth}
\centering
\begin{tikzpicture}[scale=0.3]
  %\fill[blue!30] (6,0) rectangle (8,2);
  \draw (-2,5) -- (0,3) -- (0,0) -- (7,0) -- (7,6) -- (9,4) -- (13,4) -- (11,6) -- (11,9) -- (4,9) -- (4,3) -- (2,5) -- cycle;
  \draw (0,3) --(7,3) (4,6) --(11,6);
   \foreach \x in {(4,6), (7,6), (11,6)} \filldraw[fill=white] \x circle (3pt);
   \foreach \x in {(0,3),(4,3),(7,3)} \filldraw[fill=black] \x circle (3pt);
   \foreach \x in {(-2,5),(0,0),(2,5),(4,9), (4,0), (7,9), (7,0), (9,4), (11,9), (13,4)} \filldraw[fill=white] \x +(-3pt,-3pt) rectangle +(3pt,3pt);
   %{\filldraw[fill=white] \x circle (3pt); \draw \x +(-3pt,0) -- + (3pt,0) +(0,3pt) -- +(0,-3pt); }
    \draw (5.5,4.5) node {\tiny $C_1$};
    \draw (10,5) node {\tiny $C_2$};
    \draw (8,7.5) node {\tiny $C_3$};
    \draw (1,4) node {\tiny $C_4$};
    \draw (3,1.5) node {\tiny $C_5$};
    \draw (5.5,9) node[above]  {\tiny $1$} (5.5,0) node[below] {\tiny $1$};
    \draw (9,9) node[above] {\tiny $2$} (11,4) node[below] {\tiny $2$};
    \draw (0,5) node[above] {\tiny $3$}  (2,0) node[below] {\tiny $3$};
    \draw (5.5,-2) node {\tiny $M \in \cH(2,1,1)$};
    \end{tikzpicture}
\end{minipage}
\begin{minipage}[t]{0.4\linewidth}
\centering
\begin{tikzpicture}[scale=0.3]
\draw (-2,5) -- (0,3) -- (0,0) -- (7,0) -- (7,6) -- (9,4) -- (13,4) -- (11,6) -- (11,9) -- (4,9) -- (4,3) -- (2,5) -- cycle;
  \draw (0,3) --(7,3) (4,6) --(11,6);
   \foreach \x in {(4,6), (7,6), (11,6)} \filldraw[fill=white] \x circle (3pt);
   \foreach \x in {(0,3),(4,3),(7,3)} \filldraw[fill=black] \x circle (3pt);
   \foreach \x in {(-2,5),(0,0),(2,5), (4,0), (5,9), (7,0)} \filldraw[fill=white] \x +(-3pt,-3pt) rectangle +(3pt,3pt);
   \foreach \x in {(4,9), (6,0), (7,9), (9,4), (11,9), (13,4)} \filldraw[fill=black] \x +(-3pt,-3pt) rectangle +(3pt,3pt);
   %{\filldraw[fill=white] \x circle (3pt); \draw \x +(-3pt,0) -- + (3pt,0) +(0,3pt) -- +(0,-3pt); }
    \draw (5.5,4.5) node {\tiny $C_1$};
    \draw (10,5) node {\tiny $C_2$};
    \draw (8,7.5) node {\tiny $C_3$};
    \draw (1,4) node {\tiny $C_4$};
    \draw (3,1.5) node {\tiny $C_5$};
    \draw (4.5,9.4) node  {\tiny $1$} (6.5,-0.4) node {\tiny $1$};
    \draw (6,9.4) node  {\tiny $2$} (5,-0.4) node {\tiny $2$};
    \draw (9,9.4) node {\tiny $3$} (11,3.6) node {\tiny $3$};
    \draw (0,5.4) node {\tiny $4$}  (2,-0.4) node {\tiny $4$};
    \draw (5.5,-2) node {\tiny $M \in \cH(1,1,1,1)$};
\end{tikzpicture}
\end{minipage}

 \caption{Some cylinder diagrams in Case 5.I)}
 \label{fig:cyl:diag:C5I}
\end{figure}

\subsubsection{The Stratum $\cH(2,1,1)$}

No horizontally periodic translation surface in $\cH(2,1^2)$ can have more than five cylinders, so throughout this subsection, $M$ is always $\cM$-cylindrically stable.  We will prove the following proposition.
\begin{proposition}
\label{Case5IH211}
If $\cM \subset \cH(2,1^2)$ is a rank two affine manifold and $M \in \cM$ admits a cylinder decomposition  satisfying Case 5.I), then $\cM=\tilde \cQ(2,1,-1^3)$.
\end{proposition}

\begin{lemma}
\label{H211Case5IC1FreeLem}
Let $M$ be a translation surface satisfying Case 5.I) in a rank two affine manifold $\cM \subset \cH(2,1^2)$.  Then $C_1$ is free.
\end{lemma}
\begin{proof}
By contradiction, assume that $C_1$ is not free.  Let $\cC$ be the equivalence class of $C_1$. From the relation \eqref{eq:C5I:homo:rel}, we derive that if $C_2 \in \cC$, then $C_3 \in \cC$ and vice versa. The same is true for the pair $\{C_4,C_5\}$. By assumption, the cylinders must split into two or three equivalence classes. Thus, without loss of generality we can assume that $\cC=\{C_1,C_2,C_3\}$ and $C_4$ and $C_5$ are free.

We now claim that at least one of $C_4$ or $C_5$ is semi-simple. To see this, we observe that the boundaries of $C_4$ and $C_5$ contain the same simple zero. We can  assume that this simple zero is $x_1$. Since the graph $\Gr_1$ is planar, we see that among three cylinders $\{C_1,C_4,C_5\}$, there are two that are semi-simple. Thus at least one of $C_4$ and $C_5$ is semi-simple.

Collapsing the free semi-simple cylinder yields a translation surface  which is contained in a rank two affine submanifold $\cM'$ of $\cH(3,1)$ by \cite[Prop. 2.16]{AulicinoNguyenGen3TwoZeros}. But such a submanifold does not exist by \cite{AulicinoNguyenGen3TwoZeros}.  Hence, we get a contradiction.
\end{proof}

\begin{lemma}\label{lm:H211:C5I:eq:cl:cyl}
 Let $M$ be a translation surface in $\cM$ which admits a cylinder decomposition in the horizontal direction satisfying Case 5.I). Then up to a renumbering of the cylinders, the equivalence classes are $\{C_1\}, \{C_2,C_4\}, \{C_3,C_5\}$.
\end{lemma}
\begin{proof}
 By Lemma~\ref{H211Case5IC1FreeLem}, we know that one of the equivalence classes is $\{C_1\}$. If $C_2$ and $C_3$ are $\cM$-parallel, then their equivalence class would contain $C_1$, and we have a contradiction. The argument of Lemma~\ref{H211Case5IC1FreeLem} actually shows that $C_2$ and $C_3$ cannot both be free. Let us assume that $C_2$ is free and $C_3$ is $\cM$-parallel to $C_5$. Since $C_4$ cannot be $\cM$-parallel to $C_5$ (otherwise it would be $\cM$-parallel to $C_1$), $C_4$ must be free. But in this case, we would have three free cylinders $C_1,C_2,C_4$ whose core curves span a Lagrangian subspace of $H_1(M,\bR)$, which contradicts the hypothesis that $\cM$ is of rank two. Thus the only possibility remaining is that $C_2$ is $\cM$-parallel to $C_4$, and $C_3$ is $\cM$-parallel to $C_5$.
\end{proof}

\begin{lemma}\label{lm:H211:C5I:2:sim:cyl}
 Following the convention of Lemma \ref{lm:H211:C5I:eq:cl:cyl}, one of the equivalence classes $\{C_2,C_4\}$ and $\{C_3,C_5\}$ consists of two simple cylinders.
\end{lemma}
\begin{proof}
We first consider the case $C_1$ is simple. Up to a renumbering of the cylinders, we have
$$
\gamma_1=\gamma_2-\gamma_3=\pm(\gamma_4-\gamma_5) \Rightarrow \gamma_2-\gamma_3 +\eps(\gamma_4-\gamma_5)=0 \in H_1(M,\Z),
$$
where $\eps \in \{\pm 1\}$. By Lemma~\ref{lm:H211:C5I:eq:cl:cyl}, there exist constants $\lambda, \mu \in \R_{>0}$ such that $\gamma_4=\lambda \gamma_2$ and $\gamma_5=\mu\gamma_3$ as elements of $(T_M\cM)^*$. It follows
$$
(1+\eps\lambda)\gamma_2-(1+\eps\mu)\gamma_3=0 \in (T_M\cM)^*.
$$
If one of $\{1+\eps\lambda, 1+\eps\mu\}$ does not vanish, then $C_2$ and $C_3$ are $\cM$-parallel, and we have a contradiction. Thus we must have $\eps=-1$ and $\lambda=\mu=1$. It follows that
\begin{equation}
\label{eq:C5I:H211:hom:rel:1}
|\gamma_2|=|\gamma_4|, |\gamma_3|=|\gamma_5|, \hbox{ and } \gamma_1=\gamma_2-\gamma_3=\gamma_4-\gamma_5 \in H_1(M,\Z).
\end{equation}

Note that the relation \eqref{eq:C5I:H211:hom:rel:1} implies
\begin{equation}
\label{eq:C5I:H211:hom:rel:a}
\gamma_2-\gamma_3-\gamma_4+\gamma_5=0\in H_1(M,\Z)
\end{equation}
and it follows  that $\gamma_2-\gamma_3-\gamma_4+\gamma_5$ is homologous to $\partial U_0$. Therefore the configuration of $\Gr_0$ is given by $RG.2.b$ (see Figure~\ref{fig:C5I:sc:graph}).

We now claim that $C_3$ is a simple cylinder. Without loss of generality, we can assume that the top of  $C_3$ is contained in $\Gr_0$, while its bottom is contained in $\Gr_1$. From \eqref{eq:C5I:H211:hom:rel:1}, we draw that $|\gamma_3| < |\gamma_2|$ and the bottom of $C_3$ contains a single saddle connection. If $C_3$ is not simple then its top must contain exactly two saddle connections since it is homotopic to a component of $\partial U_0$. Note that the relation \eqref{eq:C5I:H211:hom:rel:a}  implies that the top of $C_4$ is also contained in $\Gr_0$.  Since a saddle connection cannot be contained in the top  of two cylinders, it follows that the top of $C_4$ consists of a single saddle connection. But this saddle connection is contained in the bottom of $C_2$ or $C_5$. Thus we must have either $|\gamma_4| < |\gamma_2|$, or $|\gamma_4| < |\gamma_5|$. In either case we have a  contradiction to \eqref{eq:C5I:H211:hom:rel:1}. Therefore, the top of $C_3$ must contain a single saddle connection, which means
that $C_3$ is simple.

By similar arguments,  $C_5$ is also simple, and the lemma is proved for this case.

\medskip

Let us consider the case $C_1$ is semi-simple (but not simple). In this case, we have
$$
\gamma_1=\gamma_2-\gamma_3=\gamma_4+\gamma_5 \Rightarrow \gamma_2 -\gamma_3-\gamma_4-\gamma_5=0 \in H_1(M,\Z).
$$
It follows that the configuration of $\Gr_0$ is given by $RG.2.a$ (see Figure~\ref{fig:C5I:sc:graph}). Let $\lambda,\mu$ be the constants above, we have
$$
(1-\lambda)\gamma_2=(1+\mu)\gamma_3 \in (T_M\cM)^*.
$$
Thus $C_2$ and $C_3$ are $\cM$-parallel, which contradicts Lemma~\ref{lm:H211:C5I:eq:cl:cyl}. Therefore, this case does not occur.

\medskip

Finally, consider the case $C_1$ is not semi-simple. In this case,  the relation \eqref{eq:C5I:homo:rel} gives
\begin{equation}
\label{eq:C5I:H211:hom:rel:b}
 \gamma_1=\gamma_2+\gamma_3=\gamma_4+\gamma_5, \hbox{ and } \gamma_2+\gamma_3-\gamma_4-\gamma_5=0 \in H_1(M,\Z).
\end{equation}
Hence the configuration of $\Gr_0$ is given by $RG.2.b$. In particular $\partial U_0$ has two simple components. Since these two simple components are paired with some simple components of $\partial U_1\sqcup \partial U_2$, the corresponding cylinders are simple. We will show that they must be $\cM$-parallel.  Let $\lambda,\mu$ be the constants above. We have
$$
(1-\lambda)\gamma_2+(1-\mu)\gamma_3=0 \in (T_M\cM)^*
$$
If one of $\{1-\lambda,1-\mu\}$ does not vanish, then $C_2$ and $C_3$ are parallel which contradicts Lemma~\ref{lm:H211:C5I:eq:cl:cyl}. Thus we must have $|\gamma_2|=|\gamma_4|$, and $|\gamma_3|=|\gamma_5|$.
Without loss of generality we can assume that  the top of $C_2$ is  contained in $\Gr_0$.   The relation \eqref{eq:C5I:H211:hom:rel:b} implies  that $\Gr_0$ contains the top of $C_3$ and the bottoms of $C_4$ and $C_5$.
Let $\sig_0$, $\sig_1$, and $\sig_2$ be the saddle connections in $\Gr_0$.   We are done unless, without loss of generality, the top of $C_2$ is $\sig_0 \cup \sig_1$ and the bottom of $C_5$ is $\sig_1 \cup \sig_2$. It follows that the top of $C_3$ is $\sig_2$ and the bottom of $C_4$ is $\sig_0$. However, the relations above imply $|\sig_0| + |\sig_1| = |\sig_0|$ and $|\sig_2| = |\sig_1| + |\sig_2|$, which is impossible.
\end{proof}

\bigskip

\begin{proof}[Proof of Proposition~\ref{Case5IH211}]
 By Lemma~\ref{lm:H211:C5I:2:sim:cyl}, we can assume that $C_2$ and $C_4$ are two $\cM$-parallel simple cylinders.  We claim that they are similar. If they are not, twist them so that there is a vertical saddle connection in $C_2$, but there are no vertical saddle connections in $C_4$. Collapsing the equivalence class $\{C_2,C_4\}$ yields a surface $M'$ which is contained in a rank two submanifold of $\cH(3,1)$ by \cite[Prop. 2.16]{AulicinoNguyenGen3TwoZeros}. But there are no affine submanifolds of rank two in $\cH(3,1)$, and we have a contradiction.

 Since the pairs of zeros in the boundaries of $C_2$ and $C_4$ are not the same, collapsing them simultaneously so that all of the zeros collide yields a surface $M' \in \cH(4)$. By Proposition~\ref{prop:collapse:similar:cyl}, $M'$ is contained in a rank two affine submanifold $\cM' \subset \cH(4)$ such that $\allowbreak \dim \cM'=\dim \cM-1$. By the results of \cite{NguyenWright, AulicinoNguyenWright}, we have $\cM'=\prymmin$. Hence, $M'$ admits a Prym involution $\tau'$.

 Note that $C_2$ and $C_4$ degenerate to two horizontal saddle connections $\sig'_2$ and $\sig'_4$ in $M'$. We claim that $\sig'_2$ and $\sig'_4$ are permuted by $\tau'$. If they are not, then there is a surface $M'_1 \in \prymmin$ close to $M'$ in which they are {\bf not} parallel. But from Proposition~\ref{prop:collapse:similar:cyl}, these  saddle connections are the degenerations of two parallel cylinders in a surface $M_1 \in \cM$ close to $M$, hence must be parallel. Thus we get a contradiction.

 Since $\sig'_2$ and $\sig'_4$ are permuted by $\tau'$, they must have the same length, which implies that $C_2$ and $C_4$ are isometric.  By Proposition~\ref{DblCovExtSimpCyl}, $\tau'$ extends to an involution with four fixed points on $M$. The same holds for any surface in a neighborhood of $M \in \cM$. It follows that $\cM \subseteq \cP\cap\cH(2,1,1)=\tilde{\cQ}(2,1,-1^3)$. Finally, since
 $$
 \dim\cM=\dim\prymmin+1=5=\dim\tilde{\cQ}(2,1,-1^3),
 $$
 we can conclude that $\cM=\tilde{\cQ}(2,1,-1^3)$.
\end{proof}

\subsubsection{The Principal Stratum $\cH(1,1,1,1)$}

The key to this section is studying the cylinder $C_1$.  The main result of this section is

\begin{proposition}
\label{Case5IH1111}
Let $\cM \subset \cH(1^4)$ be a rank two affine manifold.  Assume that $\tilde \cQ(2,1,-1^3)$ is the only rank two affine manifold in $\cH(2,1^2)$. If  $\cM$ contains a horizontally periodic surface $M$ satisfying Case 5.I), then either there exists $M' \in \cM$ horizontally periodic with six cylinders or $\cM =\prymprinc$.
\end{proposition}

We first prove the following lemmas

\begin{lemma}\label{lm:H1111:C5I:C1:free}
 Following the notation and assumption of Proposition~\ref{Case5IH1111}, either $\cM$ contains a horizontally periodic surface with six cylinders or the cylinder $C_1$ is free.
\end{lemma}
\begin{proof}
If $M$ is not $\cM$-cylindrically stable, then we can get a horizontally periodic surface with more cylinders, so in this case we are done.  Assume that $M$ is $\cM$-cylindrically stable.  If $C_1$ is not free, then by the same arguments as Lemma~\ref{H211Case5IC1FreeLem}, we see that there are two free cylinders among $\{C_2,\dots,C_5\}$, and that one of them is semi-simple. Collapsing the free semi-simple cylinder, yields a surface $M'\in \cH(2,1^2)$  which is contained in a rank two affine submanifold $\cM'$. By assumption, $\cM'=\prymthreezero$, hence $M'$ admits a Prym involution which fixes the double zero and permutes the simple ones.  But in this case, one of the simple zeros is joined to the double zero by a horizontal saddle connection whereas the other one is not. Therefore we get a contradiction which proves the lemma.
 \end{proof}

\begin{lemma}\label{lm:H1111:C5I:eq:cl:cyl}
Either the equivalence classes of horizontal cylinders in $M$ are $\{C_1\}, \{C_2,C_3\}, \{C_4,C_5\}$ or $\cM$ contains a horizontally periodic surface with six cylinders.
\end{lemma}
\begin{proof}
 We only need to consider the case when $M$ is $\cM$-cylindrically stable, which implies that the cylinders of $M$ fall into at least two equivalence classes. By Lemma \ref{lm:H1111:C5I:C1:free}, one of the equivalence classes is $\{C_1\}$.  It follows that $C_2$ and $C_3$ are not $\cM$-parallel. If both $C_2$ and $C_3$ are free, then we can conclude by the arguments of Lemma~\ref{lm:H1111:C5I:C1:free}. Consider the case where $C_2$ is free but $C_3$ is not. We can assume that $C_3$ is $\cM$-parallel to $C_5$. It follows that $C_4$ is free. But  the core curves of $C_1,C_2,C_4$ span a Lagrangian subspace of dimension $3$ in $H_1(M,\Z)$, which contradicts the hypothesis that $\cM$ is of rank two. We can then conclude that $C_2$ is $\cM$-parallel to $C_4$, and $C_3$ is $\cM$-parallel to $C_5$ up to a renumbering of the cylinders.
\end{proof}

\begin{lemma}\label{lm:H1111:C5I:C1:no:ssim}
 Assume that $M$ is $\cM$-cylindrically stable.  Then $C_1$ is not strictly semi-simple.
\end{lemma}
\begin{proof}
If $C_1$ is strictly semi-simple, then up to a renumbering of the cylinders, we have
$$
\gamma_1=\gamma_2-\gamma_3=\gamma_4+\gamma_5 \Rightarrow \gamma_2 -\gamma_3-\gamma_4-\gamma_5=0 \in H_1(M,\Z).
$$
From Lemma~\ref{lm:H1111:C5I:eq:cl:cyl}, there exist constants $\lambda, \mu \in \R_{>0}$, such that
$$
\gamma_4=\lambda\gamma_2, \quad \gamma_5=\mu\gamma_3 \in (T_M\cM)^*.
$$
Consequently
$$
(1-\lambda)\gamma_2=(1+\mu)\gamma_3 \in (T_M\cM)^*.
$$
Since $1+\mu>0$, this means that  $C_2$ and $C_3$ are $\cM$-parallel, which contradicts Lemma~\ref{lm:H211:C5I:eq:cl:cyl}. Therefore, $C_1$ cannot be semi-simple.
\end{proof}

\begin{lemma}
\label{H1t4Case5IC1NotSS}
There are two cylinder diagrams in which $C_1$ is not semi-simple which are shown in Figure \ref{fig:Case5IC1NotSS}.
\end{lemma}
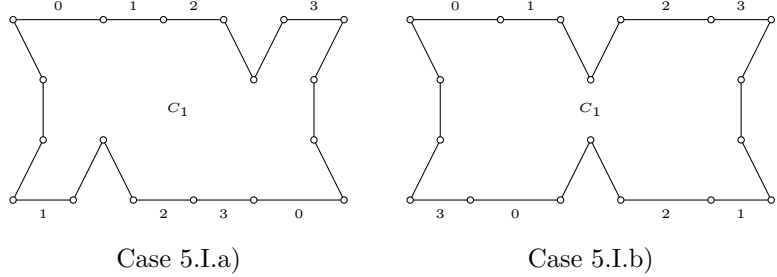
\begin{figure}[htb]
\centering
\begin{minipage}[t]{0.48\linewidth}
\centering
\begin{tikzpicture}[scale=0.4]
\draw[thin] (-1,0) -- (0,2) -- (0,4) -- (-1,6) -- (6,6) -- (7,4) -- (8,6) -- (10,6) -- (9,4) -- (9,2) -- (10,0) -- (3,0) -- (2,2) -- (1,0) -- cycle;
\foreach \x in {(-1,0), (0,2), (0,4), (-1,6),(2,6),(4,6), (6,6), (7,4), (8,6), (10,6), (9,4), (9,2), (10,0), (7,0), (5,0), (3,0), (2,2), (1,0)} \filldraw[fill=white] \x circle (3pt);

\draw (0.5,6) node[above] {\tiny $0$} (3,6) node[above] {\tiny $1$} (5,6) node[above] {\tiny $2$} (9,6) node[above] {\tiny $3$} (0,0) node[below] {\tiny $1$} (4,0) node[below] {\tiny $2$} (6,0) node[below] {\tiny $3$} (8.5,0) node[below] {\tiny $0$};

\draw (4.5,3) node {\tiny $C_1$};

\draw (4.5,-2) node {Case 5.I.a)};
\end{tikzpicture}
\end{minipage}
\begin{minipage}[t]{0.48\linewidth}
\begin{tikzpicture}[scale=0.4]
\draw[thin] (-1,0) -- (0,2) -- (0,4) -- (-1,6) -- (4,6) -- (5,4) -- (6,6) -- (11,6) -- (10,4) -- (10,2) -- (11,0) -- (6,0) -- (5,2) -- (4,0) -- cycle;
\foreach \x in {(-1,0), (0,2), (0,4), (-1,6),(2,6),(4,6), (6,6), (9,6), (11,6), (10,4), (10,2), (5,4),(11,0), (9,0), (6,0), (4,0), (5,2), (1,0)} \filldraw[fill=white] \x circle (3pt);

\draw (0.5,6) node[above] {\tiny $0$} (3,6) node[above] {\tiny $1$} (7.5,6) node[above] {\tiny $2$} (10,6) node[above] {\tiny $3$} (0,0) node[below] {\tiny $3$} (2.5,0) node[below] {\tiny $0$} (7.5,0) node[below] {\tiny $2$} (10,0) node[below] {\tiny $1$};

\draw (5,3) node {\tiny $C_1$};

\draw (5,-2) node {Case 5.I.b)};
\end{tikzpicture}
\end{minipage}
\caption{The two cylinder diagrams in $\cH(1^4)$ satisfying Case 5.I) where $C_1$ is not semi-simple}
\label{fig:Case5IC1NotSS}
\end{figure}

\begin{proof}
If $C_1$ is not semi-simple, then each of $C_2, \ldots, C_5$ are semi-simple because the identifications between each of the cylinders at $C_1$ is completely determined.  The remaining identifications can be deduced from Lemma \ref{lm:C4Ib:H1111:CylDiags}.
\end{proof}

\subsubsection*{Proof of Proposition~\ref{Case5IH1111}}
\begin{proof}
It suffices to assume that $M$ is $\cM$-cylindrically stable, otherwise we are done.  By Lemma~\ref{lm:H1111:C5I:C1:free}, we know that $C_1$ is free, and from Lemma~\ref{lm:H1111:C5I:C1:no:ssim}, we only need to consider two cases:

\medskip
\noindent \ul{\em $C_1$ is simple.}  Collapsing $C_1$ so that the two zeros in its boundary collide  yields a surface $M' \in \cH(2,1^2)$ which is contained in a rank two affine submanifold $\cM'$ by \cite[Prop. 2.16]{AulicinoNguyenGen3TwoZeros}. By assumption, $\cM'=\prymthreezero$, thus $M'$ admits a Prym involution $\tau'$.

 Let $x'_0$ be the double zero of $M'$.  Observe that all the horizontal saddle connections starting from $x'_0$ end at $x'_0$. Let us denote those saddle connections by $\sig_0,\sig_1,\sig_2$, where $\sig_0$ is the degeneration of $C_1$.  Since $C_1$ is not $\cM$-parallel to any other cylinder, we can choose $M$ such  that $|\sig_0|\neq |\sig_1|$ and $|\sig_0|\neq |\sig_2|$.   Since $\tau'$ fixes $x'_0$, it induces a permutation of $\{\sig_0,\sig_1,\sig_2\}$.
 We claim that $\sig_0$ is invariant by $\tau'$, since otherwise we have either $|\sig_0|=|\sig_1|$ or $|\sig_0|=|\sig_2|$ contradicting our assumption. We can now use Proposition \ref{DblCovExtSimpCyl} to conclude that  $\allowbreak M \in \cP\cap\cH(1^4)=\tilde \cQ(2^2,-1^4)$, and hence $\cM\subseteq \prymprinc$.  Since we have
 $$
 \dim \cM=\dim \prymthreezero +1=\dim \prymprinc=6,
 $$
it follows that $\cM=\prymprinc$.

\medskip

\noindent \ul{\em $C_1$ is not semi-simple.} There are two cylinder diagrams to consider by Lemma \ref{H1t4Case5IC1NotSS}.  In this case, each saddle connection in the boundary of $C_1$ is one component of the boundary of a cylinder in the family  $\{C_2,\dots,C_5\}$. Let us denote those saddle connections by $\sig_2,\dots,\sig_5$ such that $\sig_i$ is one boundary component of $C_i$.
Since $C_2$ is not $\cM$-parallel to $C_5$ by Lemma~\ref{lm:H1111:C5I:eq:cl:cyl},  $M$ can be chosen such that $|\sig_2|\neq|\sig_5|$.  It follows that $C_1$ can be twisted such that it contains only one vertical saddle connection joining two distinct zeros in its  boundary.  Collapsing $C_1$ yields a surface $M' \in \cH(2,1^2)$ which is contained in a rank two affine submanifold by \cite[Prop. 2.16]{AulicinoNguyenGen3TwoZeros}.
By assumption, this submanifold must be $\prymthreezero$. Thus $M'$ admits a Prym involution.  By inspecting the cylinder diagram of $M'$, we see that
this Prym involution  extends to a Prym involution of $M$ that fixes $C_1$. In particular, $\allowbreak M \in \cP\cap\cH(1^4)=\prymprinc$.

We now claim that $\cM\subseteq \prymprinc$. To see this choose a small positive real number $\eps$ such that, for any  $v \in T^\R_M\cM \subset H^1(M,\Sigma,\R)$ such that $||v|| < \eps$, we have $M_v:=M+v \in \cM$ and the condition $|\sig_2|\neq |\sig_5|$ still holds. Here, we identify $M$ with an element of $H^1(X,\Sigma,\R+\imath\R)$. Remark that since $v$ is purely real, all the horizontal saddle connections of $M$ remain horizontal in $M_v$, which means that $M_v$ is also horizontally periodic with the same cylinder diagram as $M$. By the same argument as above, we have $M_v \in \prymprinc$. Therefore, we have
$$
T^\R_M\cM \subseteq T^\R_M\prymprinc,
$$
which implies that $\cM \subseteq \prymprinc$.  In particular, we have $\allowbreak \dim \cM \leq \dim \prymprinc=6$. We now notice that by \cite[Prop. 2.16]{AulicinoNguyenGen3TwoZeros}, we must have
$$
\dim \cM > \dim \prymthreezero =5.
$$
Thus, it follows $\dim \cM=\dim \prymprinc=6$ and $\cM=\prymprinc$. The proof of Proposition~\ref{Case5IH1111} is now complete.
\end{proof}

\subsection{Case 5.II)}\label{sec:C5II}

We label the cylinders so that $C_1$ and $C_2$ are the homologous cylinders, and $C_5$ is the unique cylinder which degenerates to a pair of simple poles in the same component of the limit surface (which is obtained as one pinches all of the horizontal cylinders).

Cutting $M$ along the core curves of $C_1$ and $C_2$, then permuting the gluings, we get two translation surfaces of genus two. Let us denote by $M^1$ the surface that contains $C_3$ and $C_4 $, and by $M^2$ the surface that contains $C_5$. In particular, $C_3,C_4,C_5$ can be realized as cylinders in a translation surface of genus two.  Therefore, the cylinder diagrams in this case can be constructed by considering the unique $3$-cylinder diagram in $\cH(1,1)$, and the $2$-cylinder diagrams in $\cH(2)$ and $\cH(1,1)$.

Let  $\gamma_i$ denote the core curve of $C_i$ oriented from left to right. We have either $\gamma_1 + \gamma_3 = \gamma_4$, or $\gamma_3 + \gamma_4 = \gamma_1$.  Our goal in this section is to show

\begin{proposition}\label{prop:C5II}
 Let $\cM$ be rank two affine submanifold of rank two in $\cH(2,1^2)\cup\cH(1^4)$. Assume that $\cM$ contains a horizontally periodic surface $M$ satisfying Case 5.II), then
 \begin{itemize}
  \item[(i)] If $\cM \subset \cH(2,1^2)$, then $\cM=\prymthreezero$.

  \item[(ii)] If $\cM\subset \cH(1^4)$ and $\prymthreezero$ is the unique rank two submanifold in $\cH(2,1^2)$, then either $\cM$ contains a horizontally periodic surface with six cylinders, or $\cM=\prymprinc$.
 \end{itemize}
\end{proposition}

Let us start by proving some conditions that the cylinders in $M$ must satisfy.

\begin{lemma}\label{lm:C5II:H211:C5:free}
 Let $M\in \cM$ be a horizontally periodic surface satisfying Case 5.II). If $M$ is $\cM$-cylindrically stable, then $C_5$ is free.
\end{lemma}
\begin{proof}
For $i=1,\dots,5$, let us denote by $\xi_i$ the vector in $\allowbreak H^1(M,\Sig,\R)\simeq T^\R_M\cM$ tangent to the path defined by the twisting of $C_i$. By Poincar\'e duality, up to a non-zero constant, $\xi_i$ can be identified with $\gamma_i$, where $\gamma_i$ is an element of  $H_1(M\setminus\Sig,\R)$ (see \cite[Sect. 4.1]{MirzakhaniWrightBoundary}). Since the projection $p: H^1(M,\Sig,\R) \rightarrow H^1(M,\R)$ is dual to the map $p': H_1(M\setminus\Sigma,\R) \rightarrow H_1(M,\R)$, we see that $p(\xi_i)$ is dual to $\lambda_i\gamma_i \in  H_1(M,\R)$, with $\lambda_i \in \R\setminus\{0\}$.

The assumption of Case 5.II) means that we have the following homological relations
$$
\gamma_1=\gamma_2=\gamma_3+\gamma_4,
$$
and $\{\gamma_3,\gamma_4,\gamma_5\}$ span a Lagrangian in $H_1(M,\R)$.

By assumption we have at least two equivalence classes of horizontal cylinders. Assume that $C_5$ is not free. If $C_5$ is $\cM$-parallel to $\{C_1,C_2\}$, then $C_3$ and $C_4$ must be free because otherwise we only have one equivalence class. It follows that $T^\R_M\cM$ contains the vectors $\{\xi_1+\xi_2+\xi_5, \xi_3,\xi_4\}$. Thus $p(T^\R_M\cM)$ contains the duals of $\{\lambda_1\gamma_1+\lambda_2\gamma_2+\lambda_5\gamma_5, \gamma_3,\gamma_4\}$. But as this family spans a Lagrangian (of dimension $3$) in $H_1(M,\R) \simeq H^1(M,\R)$, we get a contradiction to the hypothesis that $\cM$ is of rank two. Thus this case cannot occur.

Assume now that $C_5$ is $\cM$-parallel to either $C_3$ or $C_4$.  By the homological relations, it follows that one of $C_3$ and $C_4$ is free, and $\{C_1,C_2\}$ is an equivalence class. We can suppose that  $\{C_4,C_5\}$ is an equivalence class and $C_3$ is free. By the same argument as above, $p(T^\R_M\cM)$ contains the duals of the vectors $\gamma_1,\gamma_3, \lambda_4\gamma_4+\lambda_5\gamma_5 \in H_1(M,\R)$.
Since $\{\gamma_1,\gamma_3,\gamma_5\}$ spans a Lagrangian in $H_1(M,\R)$, and $\gamma_4=\gamma_1-\gamma_3$, we see that the family $\{\gamma_1,\gamma_3, \lambda_4\gamma_4+\lambda_5\gamma_5\}$ also spans a Lagrangian in $H_1(M,\R)$. Hence we also have contradiction in this case, which shows that $C_5$ must be free.
\end{proof}

\begin{lemma}\label{lm:C5II:H211:eq:cl:cyl}
The equivalence classes of cylinders on $M$ are $\{C_1,C_2,C_3,C_4\}$ and $\{C_5\}$.
\end{lemma}
\begin{proof}
Lemma~\ref{lm:C5II:H211:C5:free} proves that one of the equivalence classes is $\{C_5\}$. It remains to show that $C_3$ and $C_4$ are $\cM$-parallel to  $C_1$ and $C_2$. If one of them is $\cM$-parallel to $\{C_1,C_2\}$, then so is the other. Thus we only need to rule out the case when they are both free. But this follows already from the arguments of Lemma~\ref{lm:C5II:H211:C5:free}.
\end{proof}

\begin{lemma}\label{lm:C5II:C3C4:sim}
The cylinders $C_3$ and $C_4$ are simple.
\end{lemma}
\begin{proof}
Recall that $C_3$ and $C_4$ are two cylinders in a 3-cylinder decomposition of a surface in $\cH(1,1)$. Thus at least one of them, say $C_3$, must be simple. If $C_4$ is not simple, then the boundary of $C_3$ is contained in the boundary of $C_4$. But since $C_3$ and $C_4$ are $\cM$-parallel, this contradicts \cite[Lem. 2.11]{AulicinoNguyenGen3TwoZeros}. Therefore, both $C_3,C_4$ must be simple.
\end{proof}

There are two admissible cylinder diagrams for Case 5.II) in $\cH(2,1^2)$, which are shown in Figure~\ref{fig:C5II:H211:cyl:diagram}.

\begin{figure}[htb]
\begin{minipage}[t]{0.4\linewidth}
\centering
\begin{tikzpicture}[scale=0.3]
  %\fill[blue!30] (6,0) rectangle (8,2);
  \draw (0,10) -- (0,8)  -- (5,8) -- (5,5) -- (0,5) -- (0,0) -- (4,0) -- (4,3) -- (8,3) -- (8,13) -- (4,13) -- (4,10) -- cycle;
  \draw (4,10) -- (8,10) (5,8) -- (8,8) (5,5) -- (8,5) (0,3) -- (4,3) ;
   \foreach \x in {(0,3), (4,13), (4,3), (8,13), (8,3)} \filldraw[fill=white] \x circle (3pt);
   \foreach \x in {(0,10), (0,0), (4,10), (4,0), (8,10)} \filldraw[fill=black] \x circle (3pt);
   \foreach \x in {(0,8), (0,5), (5,8), (5,5), (8,8), (8,5)} \filldraw[fill=white] \x +(-3pt,-3pt) rectangle +(3pt,3pt);
   %{\filldraw[fill=white] \x circle (3pt); \draw \x +(-3pt,0) -- + (3pt,0) +(0,3pt) -- +(0,-3pt); }
    \draw (3,9) node {\tiny $C_1$};
    \draw (3,4) node {\tiny $C_2$};
    \draw (6,11.5) node {\tiny $C_3$};
    \draw (2,1.5) node {\tiny $C_4$};
    \draw (6.5,6.5) node {\tiny $C_5$};
     \draw (2,5.5) node  {\tiny $1$} (2,7.5) node {\tiny $1$};
     \draw (2,10.5) node {\tiny $2$} (2,-0.5) node {\tiny $2$};
     \draw (6,13.5) node {\tiny $3$} (6,2.5) node {\tiny $3$};
     \draw (4,-2) node {\tiny $C_5$ is simple};
    \end{tikzpicture}
\end{minipage}
\begin{minipage}[t]{0.4\linewidth}
\centering
\begin{tikzpicture}[scale=0.3]
  %\fill[blue!30] (6,0) rectangle (8,2);
  \draw (0,10) -- (0,0) -- (3,0) -- (3,3) -- (6,3) -- (6,5) -- (10,5) -- (10,8) -- (6,8) -- (6,13) -- (3,13) -- (3,10)  -- cycle;
   \draw (3,10) -- (6,10) (0,8) -- (6,8) (0,5) -- (6,5) (0,3) -- (3,3) ;
   \foreach \x in {(0,3), (3,13), (3,3), (6,13), (6,3)} \filldraw[fill=white] \x circle (3pt);
   \foreach \x in {(0,10), (0,0), (3,10), (3,0), (6,10)} \filldraw[fill=black] \x circle (3pt);
    \foreach \x in {(0,8), (0,5), (6,8), (6,5), (10,8), (10,5)} \filldraw[fill=white] \x +(-3pt,-3pt) rectangle +(3pt,3pt);
%    %{\filldraw[fill=white] \x circle (3pt); \draw \x +(-3pt,0) -- + (3pt,0) +(0,3pt) -- +(0,-3pt); }
     \draw (3,9) node {\tiny $C_1$};
     \draw (3,4) node {\tiny $C_2$};
     \draw (4.5,11.5) node {\tiny $C_3$};
     \draw (1.5,1.5) node {\tiny $C_4$};
     \draw (3,6.5) node {\tiny $C_5$};
      \draw (8,8.5) node  {\tiny $1$} (8,4.5) node {\tiny $1$};
      \draw (1.5,10.5) node {\tiny $2$} (1.5,-0.5) node {\tiny $2$};
      \draw (4.5,13.5) node {\tiny $3$} (4.5,2.5) node {\tiny $3$};
      \draw (4,-2) node {\tiny $C_5$ is not simple};
    \end{tikzpicture}
\end{minipage}

\caption{Case 5.II) in $\cH(2,1,1)$: Admissible cylinder diagrams}
 \label{fig:C5II:H211:cyl:diagram}
\end{figure}
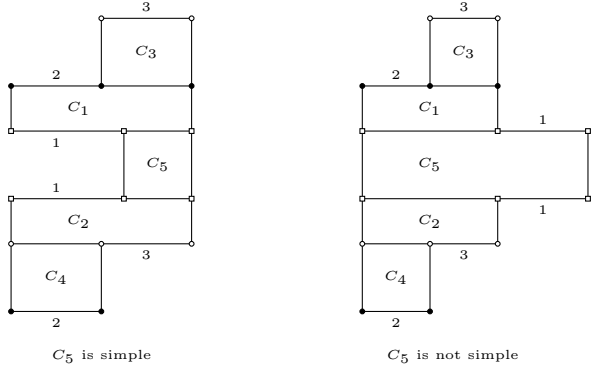

\subsubsection{Proof of Proposition~\ref{prop:C5II}: Case $\cM \subset \cH(2,1^2)$}
\begin{proof}
Note that in this case $M^1\in \cH(1,1)$ and $M^2\in \cH(2)$. In particular,  either $C_5$ is  a simple cylinder, or  there exists a saddle connection  which is contained in both its top and bottom.
Let $x_0$ denote the double zero of $M$, and $x_1,x_2$ denote the simple ones.  For $i\in \{1,\dots,5\}$, let $h_i$ and $\ell_i$  denote respectively the height, and the  circumference of $C_i$. Without loss of generality, let $h_1 \leq h_2$.

\medskip

\noindent \ul{\em Case $C_5$ is simple.} We will perform an extended cylinder deformation to show that in fact, $h_1 = h_2$ and that $C_1$ and $C_2$ must simultaneously admit vertical saddle connections.  Then we will collapse to $\cH(4)$ to conclude.

Let $\sig_5$ and $\sig'_5$ denote respectively  the top and bottom borders of $C_5$. We can assume that $\sig_5$ is contained in the bottom of $C_1$. Remark that the top of $C_1$ contains a unique zero of $M$, which can be supposed to be $x_1$.  Twist the cylinders in the equivalence class $\{C_1,\dots,C_4\}$ such that none of the descending vertical rays from the copies of $x_1$ in the top of $C_1$ hits a copy of $x_0$ in the bottom of $C_1$ before exiting $C_1$, and one of those rays intersects the interior of $\sig_5$.
We can then twist $C_5$ such that this ray hits $x_0$ after crossing $C_5$. We then have a vertical saddle connection $\del$ from $x_0$ to $x_1$ crossing $C_5$ and $C_1$. Note that we have $|\del|=h_5+h_1$.

Consider now the deformations of $M$ by stretching $C_5$. These deformations define a path in $\cM$ whose tangent vector $\xi \in H^1(M,\Sigma;\imath\R)$ satisfies $ \xi(c)= \imath \langle \gamma_5,c\rangle$ for any $c\in H_1(M,\Sigma;\Z)$, where $\langle , \rangle$ is the intersection form (see \cite[Lem. 2.4]{WrightCylDef}).

The path in $\cM$ corresponding to this family of deformations is $M+I\xi$, where $I$ is an interval of $\R$, and $M$ is identified with an element of $H^1(M,\Sigma;\R+\imath\R)$. Recall that $\cM$ is locally identified with an open subset of a linear subspace $V$ of  $H^1(M,\Sigma;\R+\imath\R)$. Since $M\in V$ and $\xi \in V$, we have $M+t\xi \in V$. Hence as long as $M_t:=M+t\xi$ corresponds to a surface in $\cH(2,1,1)$, this surface must belong to $\cM$.

Observe now that when $t=-h_5$, the cylinder $C_5$ degenerates to the union of two horizontal saddle connections. Consider now $M_t$ for $t \in (-(h_1+h_5),-h_5)$. We first observe that for those values of $t$,  $C_3$ and $C_4$ are not affected by the deformations, $M_t$ remains horizontally periodic and always satisfies Case 5.II). The cylinders $C_1$ and $C_2$ give rise to two homologous cylinders on $M_t$, which will be denoted by $C'_1$ and $C'_2$ respectively, and there is an additional horizontal cylinder that we denote by $C'_5$ (see Figure~\ref{fig:C5II:H211:C5:sim:deform}). The bottom of $C'_1$ (resp. the top of $C'_2$) consists of a single saddle connection, and the new $C'_5$ is not simple. The heights of $C'_1,C'_2$, and $C'_5$ are given by  $h_1+h_5+t, h_2+h_5+t$,  and $h_5-t$ respectively.  Such $M_t$ are called  {\em extended cylinder deformations} of $M$ and are described in \cite[Sec. 4.2]{AulicinoNguyenGen3TwoZeros}.

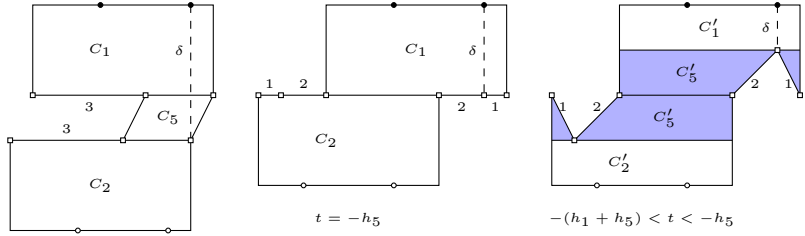
\begin{figure}[htb]
\centering
\begin{tikzpicture}[scale=0.3]
   \draw (0,4) -- (0,0) -- (8,0) -- (8,4) -- (9,6) --(9,10) -- (1,10) -- (1,6) -- (6,6) -- (5,4) -- cycle;
   \draw (6,6) -- (9,6) (5,4) -- (8,4);
   \draw[dashed] (8,10) -- (8,4);
   \foreach \x in {(3,0), (7,0)} \filldraw[fill=white] \x circle (3pt);
   \foreach \x in {(4,10), (8,10)} \filldraw[fill=black] \x circle (3pt);
    \foreach \x in {(0,4), (1,6), (5,4), (6,6), (8,4), (9,6)} \filldraw[fill=white] \x +(-3pt,-3pt) rectangle +(3pt,3pt);
      \draw (4,8) node {\tiny $C_1$};
      \draw (4,2) node {\tiny $C_2$};
      \draw (7,5) node {\tiny $C_5$};
       \draw (3.5,5.5) node  {\tiny $3$} (2.5,4.5) node {\tiny $3$};
       \draw (7.5,8) node {\tiny $\del$};

  \draw (11,6) -- (11,2) -- (19,2) -- (19,6) -- (22,6) -- (22,10) -- (14,10) -- (14,6) -- cycle;
   \draw (14,6) -- (19,6);
   \draw[dashed] (21,10) -- (21,6);
   \foreach \x in {(13,2), (17,2)} \filldraw[fill=white] \x circle (3pt);
    \foreach \x in {(17,10), (21,10)} \filldraw[fill=black] \x circle (3pt);
     \foreach \x in {11,12,14, 19, 21,22} \filldraw[fill=white] (\x,6) +(-3pt,-3pt) rectangle +(3pt,3pt);
      \draw (18,8) node {\tiny $C_1$};
       \draw (14,4) node {\tiny $C_2$};
       \draw (20.5,8) node {\tiny $\del$};
       \draw (11.5,6.5) node  {\tiny $1$} (21.5,5.5) node {\tiny $1$};
       \draw (13,6.5) node  {\tiny $2$} (20,5.5) node {\tiny $2$};
       \draw (15,0.5) node {\tiny $t=-h_5$};

\fill[blue!30] (27,8) -- (34,8) -- (32,6) -- (32,4) -- (25,4) -- (27,6) --cycle;
\fill[blue!30] (24,6) -- (24,4) -- (25,4) --cycle;
\fill[blue!30] (34,8) -- (35,8) -- (35,6) -- cycle;

\draw (24,6) -- (24, 2) -- (32,2) -- (32,6) -- (34,8) -- (35,6) -- (35,10) -- (27,10) -- (27,6) -- (25,4)  -- cycle;
   \draw (27,8) -- (35,8) (27,6) -- (32,6) (24,4) -- (32,4);
   \draw[dashed] (34,10) -- (34,8);
    \foreach \x in {(26,2), (30,2)} \filldraw[fill=white] \x circle (3pt);
    \foreach \x in {(30,10), (34,10)} \filldraw[fill=black] \x circle (3pt);
      \foreach \x in {(24,6), (25,4), (27,6), (32,6), (34,8), (35,6)} \filldraw[fill=white] \x +(-3pt,-3pt) rectangle +(3pt,3pt);
       \draw (31,9) node {\tiny $C'_1$};
        \draw (27,3) node {\tiny $C'_2$};
       \draw (29,5) node {\tiny $C'_5$} (30,7) node {\tiny $C'_5$};
       \draw  (33.5,9) node {\tiny $\del$};

      \draw (24.5,5.5) node  {\tiny $1$} (34.5,6.5) node {\tiny $1$};
      \draw (26,5.5) node {\tiny $2$} (33,6.5) node {\tiny $2$};
      \draw (28,0.5) node {\tiny $-(h_1+h_5) < t < -h_5$};
\end{tikzpicture}

 \caption{Extended Cylinder Deformation for Case 5.II) in $\cH(2,1,1)$, $C_5$ is simple}
 \label{fig:C5II:H211:C5:sim:deform}
\end{figure}

By construction the saddle connection $\del$ remains vertical in $M_t$, and its length is given by $h_1+h_5+t$. As $t$ tends to $-(h_1+h_5)$,  $\del$ shrinks to a point, which means that $x_0$ and $x_1$ collide. If $h_2>h_1$ then no other collision of zeros occurs, and the resulting surface, denoted by $M'$, belongs to $\cH(3,1)$ (which can easily be checked by hand). One can now use the arguments of \cite[Prop. 2.16]{AulicinoNguyenGen3TwoZeros} to conclude that $M'$ must be contained in a rank two affine submanifold of $\cH(3,1)$. But since such a submanifold does not exist by the results of \cite{AulicinoNguyenGen3TwoZeros}, we get a contradiction.

Assume from now on that $h_1=h_2$. Consider again the limit surface $M'$ as $t$ tends to $-(h_1+h_5)$.  One can easily check that  $M'$ is a translation surface of genus three. If there is no vertical saddle connection in $C'_2$, then $M' \in \cH(3,1)$ and we get again a contradiction. If $C'_2$ contains a vertical saddle connection, then this one is unique, and in the limit the three zeros of $M$ collide, and the resulting surface  $M'$ belongs to $\cH(4)$.

By \cite[Cor. 1.2]{MirzakhaniWrightBoundary}, $M'$ is contained in an affine manifold $\cM'$ of rank at most two. Observe that since the cylinder $C'_5$ on $M_t$ is free, it must also be free on $M'$ with respect to $\cM'$. Thus  $\cM'$ must be an affine manifold of rank  two.    By the results of \cite{AulicinoNguyenWright, NguyenWright}, $M'$ must belong to the Prym locus $\prymmin$. In particular, $M'$ admits a Prym involution. Observe that this involution exchanges $C_3$ and $C_4$, which means that $C_3$ and $C_4$ are isometric. It is now easy to check that the Prym involution of $M'$ extends to a Prym involution on $M_t$, for any $t\in (-(h_1+h_5),-h_5)$. Therefore we have $M_t\in \prymthreezero$.

Choose $t \in (-(h_1+h_5),-h_5)$ and consider the surface $M_t+v$, where $v \in T_{M_t}^\R\cM$ is a vector in a neighborhood of $0$. By the same arguments as above, we see that $M_t+v \in \prymthreezero$ (remark that if $M_t \in \prymthreezero$, then twisting simultaneously the equivalence class of $C'_1$ also gives a surface in $\prymthreezero$). It follows that $T^\R_{M_t}\cM \subseteq  T^\R_{M_t}\prymthreezero$, hence  $\cM \subseteq \prymthreezero$. Since we also have
$$
\dim \cM > \dim \prymmin \Rightarrow \dim\cM \geq \dim \prymmin+1=\dim \prymthreezero,
$$
we conclude that $\cM=\prymthreezero$.

\begin{figure}[htb]
\centering
\begin{tikzpicture}[scale=0.3]
\fill[blue!30] (4,4) -- (6,4) -- (9,6) -- (7,6);
   \draw (-4,4) --(-4,0) -- (4,0) -- (4,4) -- (6,4) -- (9,6) -- (7,6) -- (7,10) -- (-1,10) -- (-1,6)  -- cycle;
    \draw (-1,6) -- (7,6) -- (4,4) -- (-4,4);
    \draw[dashed] (6,10) -- (6,4);
    \foreach \x in {(-2,0), (2,0)} \filldraw[fill=white] \x circle (3pt);
    \foreach \x in {(2,10), (6,10)} \filldraw[fill=black] \x circle (3pt);
     \foreach \x in {(-4,4), (-1,6), (4,4), (7,6), (6,4), (9,6)} \filldraw[fill=white] \x +(-3pt,-3pt) rectangle +(3pt,3pt);
      \draw (2,8) node {\tiny $C_1$};
       \draw (-1,2) node {\tiny $C_2$};
       \draw (1,5) node {\tiny $C_5$};
       \draw (6.5,5) node {\tiny $D$};
        \draw (5,3.7) node  {\tiny $\sig$} (8,6.3) node {\tiny $\sig$};
        \draw (5.5,8) node {\tiny $\del$};

%***************************************************************************************************************************

 \draw (11,6) -- (11,2) -- (19,2) -- (19,6) -- (22,6) -- (22,10) -- (14,10) -- (14,6) -- cycle;
   \draw (14,6) -- (19,6);
   \draw[dashed] (21,10) -- (21,6);
   \foreach \x in {(13,2), (17,2)} \filldraw[fill=white] \x circle (3pt);
    \foreach \x in {(17,10), (21,10)} \filldraw[fill=black] \x circle (3pt);
     \foreach \x in {11,12,14, 19, 21,22} \filldraw[fill=white] (\x,6) +(-3pt,-3pt) rectangle +(3pt,3pt);
      \draw (18,8) node {\tiny $C_1$};
       \draw (14,4) node {\tiny $C_2$};
       \draw (20.5,8) node {\tiny $\del$};
       \draw (11.5,6.5) node  {\tiny $\eta$} (21.5,5.5) node {\tiny $\eta$};
       \draw (13,6.5) node  {\tiny $\sig$} (20,5.5) node {\tiny $\sig$};
       \draw (15,0.5) node {\tiny $t=-h_5$};

\fill[blue!30] (27,7) -- (32,7) -- (32,5) -- (27,5) -- cycle;
\fill[blue!30] (24,7) -- (24,5) -- (25,5) --cycle;
\fill[blue!30] (34,7) -- (35,7) -- (35,5) -- cycle;

\draw (24,7) -- (24, 2) -- (32,2) -- (32,7) -- (34,7) -- (35,5) -- (35,10) -- (27,10) -- (27,5) -- (25,5)  -- cycle;
 \draw (27,7) -- (32,7) (27,5) -- (32,5) (24,5) -- (25,5) (34,7) -- (35,7);
 \draw[dashed] (34,10) -- (34,7);
 \foreach \x in {(26,2), (30,2)} \filldraw[fill=white] \x circle (3pt);
  \foreach \x in {(30,10), (34,10)} \filldraw[fill=black] \x circle (3pt);
      \foreach \x in {(24,7), (25,5), (27,5), (32,7), (34,7), (35,5)} \filldraw[fill=white] \x +(-3pt,-3pt) rectangle +(3pt,3pt);
       \draw (31,9) node {\tiny $C'_1$};
        \draw (28,3) node {\tiny $C'_2$};
       \draw (29.5,6) node {\tiny $C'_5$};
       \draw  (33.5,9) node {\tiny $\del$};

       \draw (25,6) node  {\tiny $\eta$} (34.2,5.7) node {\tiny $\eta$};
      \draw (26,5.3) node {\tiny $\sig$} (33,6.5) node {\tiny $\sig$};
      \draw (28,0.5) node {\tiny $-(h_1+h_5) < t < -h_5$};
\end{tikzpicture}
\caption{Extended Cylinder Deformations for Case 5.II) in $\cH(2,1,1)$, $C_5$ is not simple}
\label{fig:C5II:H211:C5:no:sim:deform}
\end{figure}
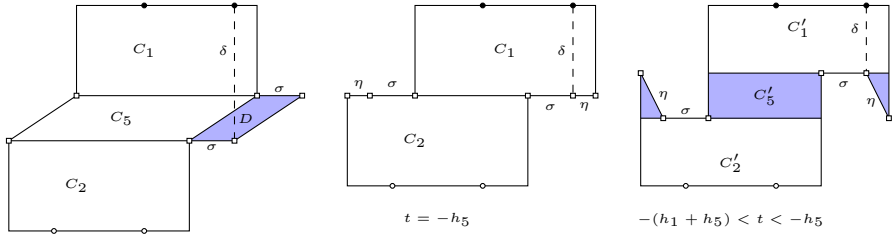

\medskip

\noindent \ul{\em Case $C_5$ is not simple.} In this case $C_1$ and $C_2$ are semi-simple, and there is a saddle connection $\sig$ which is contained in both the top and bottom of $C_5$. Let $D$ be a simple cylinder in the closure of $C_5$ that contains $\sig$.  Since $C_5$ is free, so is $D$. By stretching $D$, we can assume that the length of $\sig$ is smaller than the length of any other horizontal saddle connection. Twist $C_1$ such that all of the descending vertical rays from the singularities in its top (there are two such rays) do not hit the singularity in its bottom.  Twist $C_5$ so that one of those  rays hits $x_0$ after crossing $C_5$, but the other one does not  (see Figure~\ref{fig:C5II:H211:C5:no:sim:deform}).

Consider the extended cylinder deformations along the vector corresponding to the stretching of $C_5$. It is straightforward to verify that the same arguments as the previous case allow us to conclude that $\cM=\prymthreezero$.
\end{proof}

\subsubsection{Proof of Proposition~\ref{prop:C5II}: Case $\cM \subset \cH(1,1,1,1)$}
\begin{proof}\hfill\\
\noindent \ul{\em Case $C_5$ is simple.} By Lemma~\ref{lm:s:cyl:dist:sim:zeros}, the zeros of $M$ in the top and bottom of $C_5$ are distinct. Since $C_5$ is free by Lemma \ref{lm:C5II:H211:C5:free}, collapse it so that these two zeros collide to yield a surface $M' \in \cH(2,1,1)$. By Proposition~\ref{prop:collapse:free:sim:cyl}, $M'$ is contained in a rank two affine submanifold $\cM'$  of $\cH(2,1,1)$ such that $\allowbreak \dim \cM=\dim \cM'+1$.
By assumption, we have $\cM'=\prymthreezero$, and in particular  $M'$ admits a Prym involution $\tau'$.
Since the degeneration of $C_5$ on $M'$ is a saddle connection which is fixed by  $\tau'$, by Proposition~\ref{DblCovExtSimpCyl}, this Prym involution extends to a Prym involution of $M$.
Thus we have $M \in \prymprinc$. Since the same holds for any surface in $\cM$ close enough to $M$, we conclude that $\cM \subseteq \prymprinc$. Finally, by a dimension count
$$
\dim \cM=\dim \prymthreezero+1=\dim \prymprinc,
$$
we conclude that $\cM=\prymprinc$.

\medskip

\noindent \ul{\em Case $C_5$ is not simple.} In this case, the closure of $C_5$ contains a simple cylinder whose core curve crosses $C_5$ once.
Let $D$ be such a simple cylinder. By twisting $C_5$, we can assume that $D$ is vertical.
It is not difficult to check that $D$ is free because any other vertical cylinder $D'$ which crosses $C_5$ must cross either $C_1$ or $C_2$, which would contradict Lemma \ref{lm:C5II:H211:C5:free}. The remainder of the proof follows from the same lines as the previous case.
\end{proof}

%**************************************************************************************
%**************************************************************************************
%**************************************************************************************

\subsection{Proof of Theorem \ref{thm:rk2:H211:H1111}: Part I}\label{sec:prf:MainThm:H211}
\begin{proof}
Let $\cM$ be a rank two affine manifold in $\cH(2,1,1)$.  By Proposition~\ref{Min4CylProp} Part (1), there exists a horizontally periodic surface with at least four cylinders.
By Proposition~\ref{prop:4cyl}(a) we reduce to the case of horizontally periodic surfaces with five cylinders. Let $M \in \cM$ be a horizontally periodic surface with five cylinders.
By Lemma~\ref{5CylDeg}, $M$  satisfies  Case 5.I) or Case 5.II). In either case, Proposition~\ref{Case5IH211} or Proposition~\ref{prop:C5II} allows us to conclude that $\cM=\prymthreezero$.
The first part of Theorem~\ref{thm:rk2:H211:H1111} is then proved.
\end{proof}

\section{Six Cylinders}\label{sec:6cyl}

This section obviously only concerns the principal stratum in genus three. We remark that there is no longer a need to assume that $\tilde \cQ(2,1,-1^3)$ is the only rank two affine manifold in $\cH(2,1,1)$ because this fact has been established in Section~\ref{sec:prf:MainThm:H211}.
Furthermore, since no horizontally periodic surface in genus three can have more than six cylinders, $M$ is $\cM$-cylindrically stable throughout this section.

\begin{proposition}
\label{prop:6CylDiags}
There are four $6$-cylinder diagrams in genus three, they are shown in Figure~\ref{6CylDiagsFig}.
\end{proposition}

\begin{proof}
See Appendix~\ref{sec:6cyl:diag:proof}.
\end{proof}

\begin{figure}[htb]

\begin{minipage}[t]{0.49\linewidth}
\centering
\begin{tikzpicture}[scale=0.4]
\draw (4,0)--(4,2)--(2,2)--(2,4)--(0,4)--(0,8)--(-1,10)--(1,10)--(2,8)--(3,10)--(5,10)--(4,8)--(4,6)--(6,6)--(6,0) -- cycle;
\foreach \x in {(4,0),(4,2),(2,2),(2,4),(0,4),(0,6),(0,8),(-1,10),(1,10),(2,8),(3,10),(5,10),(4,8),(4,6),(6,6),(6,4),(6,2),(6,0)} \filldraw[fill=white] \x circle (3pt);

\draw (0,10) node[above] {\tiny $1$} (4,10) node[above] {\tiny $2$} (5,6) node[above] {\tiny $3$} (1,4) node[below] {\tiny $1$} (3,2) node[below] {\tiny $2$} (5,0) node[below] {\tiny $3$};

\draw (5,1) node {\tiny $C_6$};
\draw (4,3) node {\tiny $C_5$};
\draw (3,5) node {\tiny $C_4$};
\draw (2,7) node {\tiny $C_3$};
\draw (0.5,9) node {\tiny $C_1$};
\draw (3.5,9) node {\tiny $C_2$};

\draw (2.5,-2) node {Case 6.a)};

\end{tikzpicture}
\end{minipage}
\begin{minipage}[t]{0.49\linewidth}
\begin{tikzpicture}[scale=0.4]
\draw (2,0)--(2,4)--(0,4)--(0,8)--(2,8)--(2,10)--(4,10)--(4,6)--(5,8)--(7,8)--(6,6)--(6,2)--(4,2)--(4,0) -- cycle;
\foreach \x in {(2,0),(2,2),(2,4),(0,4),(0,6),(0,8),(2,8),(2,10),(4,10),(4,8),(4,6),(5,8),(7,8),(6,6),(6,4),(6,2),(4,2),(4,0)} \filldraw[fill=white] \x circle (3pt);

\draw (1,8) node[above] {\tiny $1$} (3,10) node[above] {\tiny $2$} (6,8) node[above] {\tiny $3$} (1,4) node[below] {\tiny $3$} (5,2) node[below] {\tiny $2$} (3,0) node[below] {\tiny $1$};

\draw (3,1) node {\tiny $C_6$};
\draw (4,3) node {\tiny $C_5$};
\draw (3,5) node {\tiny $C_4$};
\draw (2,7) node {\tiny $C_3$};
\draw (5.5,7) node {\tiny $C_2$};
\draw (3,9) node {\tiny $C_1$};

\draw (3.5,-2) node {Case 6.b)};
\end{tikzpicture}
\end{minipage}

\vspace{0.5cm}

\begin{minipage}[t]{0.49\linewidth}
\centering
\begin{tikzpicture}[scale=0.4]
\draw (0,0)--(0,2)--(-1,4)--(1,4)--(2,2)--(3,4)--(3,6)--(2,8)--(4,8)--(5,6)--(6,8)--(8,8)--(7,6)--(7,4)--(5,4)--(4,2)--(4,0)-- cycle;
\foreach \x in {(0,0),(0,2),(-1,4),(1,4),(2,2),(3,4),(3,6),(2,8),(4,8),(5,6),(6,8),(8,8),(7,6),(7,4),(5,4),(4,2),(4,0),(2,0)} \filldraw[fill=white] \x circle (3pt);

\draw (0,4) node[above] {\tiny $1$} (3,8) node[above] {\tiny $2$} (7,8) node[above] {\tiny $3$} (1,0) node[below] {\tiny $1$} (3,0) node[below] {\tiny $2$} (6,4) node[below] {\tiny $3$};

\draw (2,1) node {\tiny $C_6$};
\draw (0.5,3) node {\tiny $C_4$};
\draw (3.5,3) node {\tiny $C_5$};
\draw (5,5) node {\tiny $C_3$};
\draw (3.5,7) node {\tiny $C_1$};
\draw (6.5,7) node {\tiny $C_2$};

\draw (3.5,-2) node {Case 6.c)};

\end{tikzpicture}
\end{minipage}
\begin{minipage}[t]{0.49\linewidth}
\begin{tikzpicture}[scale=0.4]
\draw (0,0)--(0,2)--(-1,4)--(1,4)--(2,2)--(3,4)--(3,6)--(2,8)--(4,8)--(5,6)--(6,8)--(8,8)--(7,6)--(7,4)--(5,4)--(4,2)--(4,0)-- cycle;
\foreach \x in {(0,0),(0,2),(-1,4),(1,4),(2,2),(3,4),(3,6),(2,8),(4,8),(5,6),(6,8),(8,8),(7,6),(7,4),(5,4),(4,2),(4,0),(2,0)} \filldraw[fill=white] \x circle (3pt);

\draw (0,4) node[above] {\tiny $1$} (3,8) node[above] {\tiny $2$} (7,8) node[above] {\tiny $3$} (1,0) node[below] {\tiny $2$} (3,0) node[below] {\tiny $3$} (6,4) node[below] {\tiny $1$};

\draw (2,1) node {\tiny $C_6$};
\draw (0.5,3) node {\tiny $C_4$};
\draw (3.5,3) node {\tiny $C_5$};
\draw (5,5) node {\tiny $C_3$};
\draw (3.5,7) node {\tiny $C_1$};
\draw (6.5,7) node {\tiny $C_2$};

\draw (3.5,-2) node {Case 6.d)};
\end{tikzpicture}
\end{minipage}

\caption{The Four $6$-Cylinder Diagrams}
\label{6CylDiagsFig}
\end{figure}
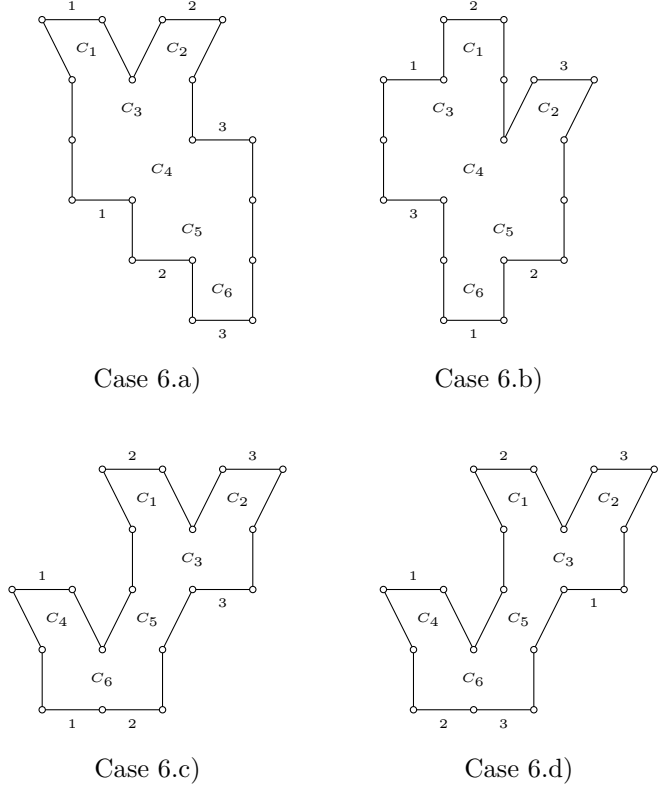

Throughout this section $\gamma_i$ denotes the homology class of the core curve of the cylinder $C_i$ oriented from left to right, for $i=1,\dots,6$.

\begin{lemma}[Case 6.a]
\label{lm:Case6a}
Let $\cM$ be a rank two affine manifold in $\cH(1^4)$.  If $\cM$ contains a horizontally periodic surface $M$ satisfying Case 6.a) (see Figure~\ref{6CylDiagsFig}), then $\cM = \prymprinc$.
\end{lemma}

\begin{proof}
Observe that the following  homological equations hold:
\begin{eqnarray*}
\gamma_1 + \gamma_2 & = &  \gamma_3\\
\gamma_2 + \gamma_6 & = & \gamma_5 \\
\gamma_1 + \gamma_5 & = & \gamma_4\\
\gamma_3+ \gamma_6  & = &  \gamma_4.
\end{eqnarray*}
If $C_1$  is free, then we can collapse it to get a surface $M' \in \cH(2,1,1)$ which must be contained in $\prymthreezero$. In particular, $M'$ has a Prym involution. Observe that this involution must exchange $C_2$ and $C_6$ and fix $C_4$. Since $C_6$ is adjacent to $C_4$ while $C_2$ is not, such an involution cannot exist, and we have a contradiction. The same argument also yields a contradiction if $C_6$ is free.

If $C_1$ is $\cM$-parallel to $C_2$ or $C_3$, then $\{C_1, C_2, C_3\}$ is an equivalence class. From the homological relations and fact that $M$ is $\cM$-cylindrically stable, all of the other cylinders are free.  But the possibility that $C_6$ is free has been excluded by the argument above. If $C_1$ is $\cM$-parallel to $C_4$ or $C_5$, then $\{C_1,C_4,C_5\}$ is an equivalence class and all of the remaining cylinders are free. Thus we also get a contradiction.

Finally, consider the case where $C_1$ is $\cM$-parallel to $C_6$.   We can twist $\{C_1,C_6\}$ so that there is a vertical saddle connection in $C_1$, then collapse this equivalence class simultaneously. If there is no vertical saddle connection in $C_6$, the collapsing yields a surface $M' \in \cH(2,1,1)$ which must be contained in $\prymthreezero$. But it is easy to see that $M'$ cannot admit a Prym involution since the cylinders corresponding to $C_3$ and $C_5$ on $M'$ are strictly semi-simple cylinders that contain different numbers of saddle connections in their boundaries. We thus get a contradiction, which means that $C_1$ contains a vertical saddle connection if and only if $C_6$ does, which means that $C_1$ and $C_6$ are similar.

By Proposition~\ref{prop:collapse:similar:cyl}, collapsing $\{C_1,C_6\}$ simultaneously yields  a surface  which is contained in a rank two affine submanifold $\cM'$ in $\cH(2,2)$ such that $\allowbreak \dim \cM=\dim \cM'+1$. We now remark that the cylinder diagram of $M'$ satisfies Case 4.I.OB) (see \cite[Sec. 6.4]{AulicinoNguyenGen3TwoZeros}), thus $M'\in \cH^{\rm odd}(2,2)$ and $\cM' = \prym$ by \cite[Lem. 6.16]{AulicinoNguyenGen3TwoZeros}.
Since $C_1$ and $C_6$ degenerate to two saddle connections that are exchanged by the Prym involution of $M'$, it follows from Proposition~\ref{DblCovExtSimpCyl} that $M$ also admits a Prym involution and so does any surface in $\cM$ close enough to $M$. We thus have $\cM \subseteq \prymprinc$. From the dimension count
$$
\dim \cM=\dim \prym +1=\dim\prymprinc=6,
$$
we conclude that $\cM=\prymprinc$.
\end{proof}

\begin{lemma}[Case 6.b]
\label{lm:Case6b}
Let $\cM$ be a rank two affine manifold in $\cH(1^4)$.  If $\cM$ contains a horizontally periodic surface $M$ satisfying Case 6.b) in Figure~\ref{6CylDiagsFig}, then $\cM = \tilde \cQ(2^2,-1^4)$.
\end{lemma}

\begin{proof}
Observe that the homological equations hold:
\begin{eqnarray*}
\gamma_1 + \gamma_6 & = & \gamma_3 = \gamma_5,\\
\gamma_2 + \gamma_3 & = & \gamma_4.
\end{eqnarray*}
If $C_2$ is free, we can collapse it and conclude by Theorem~\ref{thm:rk2:H211:H1111} Part I, and Proposition~\ref{DblCovExtSimpCyl}.

If $C_2$ is $\cM$-parallel to either $C_3$, $C_4$, or $C_5$, $\{C_2,C_3,C_4,C_5\}$ is an equivalence class, and  $C_1,C_6$ are free.  Hence, $C_1$ can be collapsed to a saddle connection $\sigma$ on a translation surface $M'$ in $\tilde \cQ(2,1,-1^3) \subset \cH(2,1,1)$ satisfying Case 5.II).  The Prym involution $\tau'$ on $M'$ necessarily fixes $C_2$ and $C_4$, which implies that $\tau'$ has at least $4$ regular fixed points. But the double zero of $M'$ must also be a fixed points. Therefore $\tau'$ has at least $5$ fixed points and we have a contradiction.

If $C_2$ is $\cM$-parallel to $C_1$, then the homological relations imply that $\{C_1, C_2\}$ is an equivalence class of cylinders because otherwise all cylinders would be in the same equivalence class.  For the same reason $\{C_3, C_5\}$ is an equivalence class, and $C_4$ and $C_6$ are each free.  Collapsing $C_6$ results in a surface $M'$ satisfying Case 5.II) in $\tilde \cQ(2,1,-1^3)$.  As above, we also get a contradiction.  Finally, if $C_2$ is $\cM$-parallel to $C_6$, then the same argument holds with $C_1$ playing the role of $C_6$ in the preceding argument.  Thus the lemma follows.
\end{proof}

\begin{lemma}[Case 6.c]
\label{lm:Case6c}
Let $\cM$ be a rank two affine manifold in $\cH(1^4)$.  If $\cM$ contains a horizontally periodic surface $M$ satisfying Case 6.c) in Figure~\ref{6CylDiagsFig}, then $\cM \in \{\dcoverprinc, \tilde \cQ(2^2,-1^4)\}$.
\end{lemma}

\begin{proof}
Observe that the homological equations hold:
\begin{eqnarray*}
\gamma_1 + \gamma_2 & = & \gamma_3,\\
\gamma_4 + \gamma_5 & = & \gamma_6,\\
\gamma_1 + \gamma_2 & = & \gamma_5 + \gamma_2 \Leftrightarrow \gamma_1 = \gamma_5.
\end{eqnarray*}
We claim that $C_2$ is not free.  If $C_2$ is free, then collapse it to obtain a surface $M'$ satisfying Case 5.II) in $\prymthreezero \subset \cH(2,1,1)$.  Observe that the Prym involution $\tau'$ of $M'$ must fix $C_3$ and a simple cylinder. Thus $\tau'$ has at least $4$ regular fixed points. Since the double zero of $M'$ must also be a fixed point of $\tau'$, we get a contradiction.

If $C_2$ were $\cM$-parallel to $C_1$, $C_3$, or $C_5$, then the homological equations would imply that $C_4$ is free, and we could collapse it to get the same contradiction as above.

If $C_2$ were $\cM$-parallel to $C_6$, then $\{C_2, C_6\}$ would be an equivalence class of cylinders as would $\{C_1, C_5\}$ by the homological equations.  However, if $C_3$ were $\cM$-parallel to $C_4$, then all cylinders would be in the same equivalence class because
$$\gamma_2 +\gamma_6= \gamma_3 + \gamma_4.$$
Hence, $C_4$ is free and we achieve the same contradiction as above.

Finally, consider the case where $C_2$ is $\cM$-parallel to $C_4$.  Then $\{C_2, C_4\}$ is an equivalence class. We can twist and collapse this equivalence class so that the two zeros in the boundary of $C_2$ collide. If the two zeros in the boundary of $C_4$ do not collide, we obtain a surface in $\prymthreezero \subset \cH(2,1,1)$. But it is easy to check that this surface does not admit a Prym involution and we get a contradiction. Therefore $C_2$ contains a vertical saddle connection if and only if $C_4$ does, which means that $C_2$ and $C_4$ are similar.
Collapsing $C_2$ and $C_4$ simultaneously yields a surface $M'$ in $\cH(2,2)$ with a cylinder diagram satisfying Case 4.II.OB). It follows in particular that $M' \in \cH^{\rm odd}(2,2)$ (cf. \cite[Sec. 6.4]{AulicinoNguyenGen3TwoZeros}).

By Proposition~\ref{prop:collapse:similar:cyl}, $M'$ is contained in a rank two affine submanifold $\cM' \subset \cH^{\rm odd}(2,2)$ such that $\dim \cM =\dim \cM'+1$. By the main results of \cite{AulicinoNguyenGen3TwoZeros}, we have $\cM' \in \{\prym, \dcoverodd\}$. In both cases, $M'$ admits a Prym involution $\tau'$. Observe that the degenerations of $C_2$ and $C_4$ on $M'$ are two saddle connections that are exchanged by $\tau'$. Thus by Proposition~\ref{DblCovExtSimpCyl}, $\tau'$ extends to a Prym involution of $M$, and the same is true for any surface in $\cM$ close enough to $M$. Thus we have $\cM \subseteq \prymprinc$.

If $\cM'=\prym$, then by the dimension count, we have
$$
\dim \cM =\dim \prym+1=\dim\prymprinc=6,
$$
which implies that $\cM=\prymprinc$.

If $\cM'=\dcoverodd$, then $M'$ also admits a hyperelliptic involution $\iota'$. We now observe that the saddle connections which are the degenerations of $C_2$ and $C_4$ are both invariant by $\iota'$. Again, by Proposition~\ref{DblCovExtSimpCyl} we see that $\iota'$ extends to a hyperelliptic involution of $M$. Thus $\allowbreak M \in \cH(1^4)\cap\cP\cap\cL=
\dcoverprinc$. Since the same is true for any surface in $\cM$ close enough to $M$, we have $\cM  \subseteq \dcoverprinc$. Finally, since we have
$$
\dim \cM=\dim\dcoverodd+1=\dim\dcoverprinc=5,
$$
$\cM$ must be $\dcoverprinc$. The proof of the lemma is now complete.
\end{proof}

\begin{lemma}[Case 6.d]
\label{lm:Case6d}
Let $\cM$ be a rank two affine submanifold in $\cH(1^4)$. If $\cM$ contains  a horizontally periodic surface $M$ satisfying Case 6.d) in Figure~\ref{6CylDiagsFig}, then $\cM \in \{\dcoverprinc, \tilde \cQ(2^2,-1^4)\}$.
\end{lemma}

\begin{proof}
Observe that the homological equations hold:
\begin{eqnarray*}
\gamma_1 + \gamma_2 & = &  \gamma_3 = \gamma_6,\\
\gamma_4 + \gamma_5 & = & \gamma_3.
\end{eqnarray*}
If $C_1$ is free, then it can be collapsed and we conclude by Theorem~\ref{thm:rk2:H211:H1111}, Part I, and Proposition~\ref{DblCovExtSimpCyl} that $\cM=\prymprinc$.

By contradiction, if $C_1$ is $\cM$-parallel to $C_2$, $C_3$, or $C_6$, then the homological equations imply that $C_4$ is free (otherwise, all six cylinders would lie in the same equivalence class). Collapsing $C_4$ to get a surface in $\cH(2,1,1)$ allows us to conclude that $\cM=\prymprinc$ by the same argument above.

If $C_1$ is $\cM$-parallel to $C_4$, then the homological equations imply that $\{C_1, C_4\}$ is an equivalence class.  By twisting and collapsing $C_1$ and $C_4$, we reach a surface $M'$ in a lower stratum: $\cH(2,1,1)$ or $\cH(2,2)$.  If $M' \in \cH(2,1,1)$ then from  Theorem~\ref{thm:rk2:H211:H1111}, Part I, we have $M' \in \prymthreezero$. Observe that if $C_1$ and $C_4$ are not similar, then one of $C_1$ and $C_4$ degenerates to a single saddle connection, while the other one degenerates to the union of two saddle connections. Hence, we can suppose that in $M'$, the top of $C_3$ contains two saddle connections, and the the top of $C_6$ contains three saddle connections.

If $M'$ admits an involution whose derivative is $-\id$, then this involution must exchange $C_3$ and $C_6$, and fix $C_2$ and $C_5$. However, such an involution has at least $5$ fixed points ($4$ regular ones in the interiors of $C_2$ and $C_5$, and the double zero of $M'$). Thus it cannot be a Prym involution. This contradiction shows that $C_1$ contains a vertical saddle connection if and only if $C_4$ does, which implies that $C_1$ and $C_4$ are similar. The remainder of the proof then follows from the same lines as  Lemma~\ref{lm:Case6c}.
\end{proof}

\subsection{Proof of Theorem~\ref{thm:rk2:H211:H1111}: Part II}\label{sec:prf:MainThm:II}
We now have all the necessary materials to complete the proof of Theorem~\ref{thm:rk2:H211:H1111}.

\begin{proof}
From Theorem~\ref{thm:rk2:H211:H1111}: Part I, we know that $\tilde \cQ(2,1,-1^3)$ is the only rank two affine manifold contained in $\cH(2,1,1)$. By Proposition~\ref{H11CoverConnProp}, $\dcoverprinc$ is connected, and $\prymprinc$ is connected by the results of \cite{LanneauComponents}.

Let $\cM$ be a rank two affine manifold in $\cH(1^4)$.  By Proposition \ref{Min4CylProp} Part (2), there exists a horizontally periodic surface $M \in \cM$ with at least four cylinders. By Proposition~\ref{prop:4cyl}, we can reduce the case of surface with at least five cylinders.

Assume now that $M$ has five horizontal cylinders. Then   $M$ must satisfy Case 5.I) or Case 5.II) by Lemma~\ref{5CylDeg}. In both cases, either $\cM$ contains a horizontally periodic surface with six cylinders, or  $\cM = \tilde \cQ(2^2,-1^4)$ by Propositions~\ref{Case5IH1111} and \ref{prop:C5II}: Part 2.

Finally, consider the case where $M$ has six horizontal cylinders. Note that in this case the hypothesis that $M$ is $\cM$-cylindrically stable is automatically satisfied.  The cylinder diagram of $M$ must satisfy one of four cylinder diagrams by Proposition~\ref{prop:6CylDiags}, and we conclude by Lemmas~\ref{lm:Case6a}, \ref{lm:Case6b}, \ref{lm:Case6c}, and \ref{lm:Case6d}.  Having addressed all possible cases, the proof of the theorem is complete.
\end{proof}

\appendix

\section{Proof of Lemma~\ref{lm:4cyl:C4III:C4}}\label{sec:aux:lms}
We first need the following two lemmas. In what follows we will use the same notation and conventions as in Section~\ref{sec:4cyl:C4III}.

\begin{lemma}\label{lm:C4III:C4notS}
If $M$ is $\cM$-cylindrically stable and $M$ satisfies Case 4.III), then there are at least two saddle connections in the top of $C_4$.
\end{lemma}
\begin{proof}
Suppose that the top $C_4$ contains only one saddle connection denoted by $\sig_0$.  Let $x_1$ denote the unique zero contained in the top of $C_4$.

We first claim that if $C_4$ is simple, then the zero in the bottom of $C_4$ is not $x_1$. If $M \in \cM(1^4)$, then this follows from Lemma~\ref{lm:s:cyl:dist:sim:zeros}. Assume that  $M \in \cH(2,1^2)$, then $x_1$ must be the double zero. Let $\sig_1$ be the unique saddle connection in the bottom of $C_4$ and assume that $\sig_1$ also joins $x_1$ to itself. Note that $\sig_1$ must be contained in the top of $C_1$ or $C_2$. Without loss of generality, let $\sig_1$ be contained in the top of $C_2$. Clearly, the top $C_2$ must contain other saddle connections.

If the top of $C_2$ contains exactly two saddle connections, then we have another horizontal saddle connection $\sig_2$ joining $x_1$ to itself.  Since we have found three horizontal saddle connections joining $x_1$ to itself, there is no saddle connection from $x_1$ to the remaining zero of $M$, which contradicts the condition that the graph $\Gr$ is connected. Thus the top of $C_2$ contains at least three saddle connections. Since the total number of horizontal saddle connections is $7$, we derive that the top of $C_1$ contains only one saddle connection, which must be contained in the bottom of $C_3$. But this contradicts \cite[Lem. 2.11]{AulicinoNguyenGen3TwoZeros}, thus we can conclude that the zero in the bottom of $C_4$ is not $x_1$.

If the bottom of $C_4$ contains more than one saddle connection, by similar arguments, one can easily show that it must contain a zero $x_2$ different from $x_1$.

Now, since $C_4$ is free, we can collapse it so that $x_1$ and $x_2$ collide, the resulting surface $M'$ is  contained in some rank two affine  submanifold $\cM'$ of a stratum with 2 or 3 zeros in genus three. Note that $x_0$ remains in $M'$, hence $M'$ has at least a simple zero. Since there are no rank two affine submanifolds in $\cH(3,1)$, we only have to consider the case $M' \in \cH(2,1^2)$ which means that $M\in \cH(1^4)$, and the collision of $x_1$ and $x_2$ gives rise to the double zero $x$ of $M'$.

Let $x'_0$ be the other simple zero of $M'$. By assumption, $\cM=\tilde{\cQ}(2,1,-1^3)$, thus $M$ admits a Prym involution $\inv$. Note that $\inv$ must fix $x$ and exchange $x_0$ and $x'_0$. By the hypothesis, there are no horizontal saddle connections joining $x_0$ to $x$, but there are some saddle connections (in the boundary of $C_4$) that connect $x$ to $x'_0$. Therefore, we have a contradiction and the lemma follows.
\end{proof}

\begin{lemma}\label{lm:sc:topC1C2:botC3}
With the same assumption as Lemma~\ref{lm:C4III:C4notS}, there is a saddle connection contained in the top of $C_1$ and the bottom of $C_3$ if and only if there is a saddle connection contained in the top of $C_2$ and the bottom of $C_3$.
\end{lemma}
\begin{proof}
 Assume that there is a saddle connection $\sig$ in the top of $C_1$ and the bottom of $C_3$, then one can twist $C_3$ (and simultaneously $C_1$ and $C_2$) such that $C_3$ is represented by a rectangle in the plane, and $\sig$ is the first saddle connection from the left in its bottom. Note that $\sig$ also occurs in the top of $C_1$. It is not difficult to see that there always exists a simple closed geodesic crossing only $C_1$ and $C_3$ that intersects $\sig$ at one point. Let $D$ denote the cylinder associated to this geodesic. Since $C_2$ is $\cM$-parallel to $C_1$, there must exist another cylinder $D'$ which is $\cM$-parallel to $D$. Since $D$ is contained in the closure of $C_1\cup C_3$,  $D'$ can only cross $C_1,C_2,C_3$. In particular, as a core curve of $D'$ exits $C_2$ through the top, it must enter $C_3$, which implies that there is a saddle connection in the top of $C_2$ and the bottom of $C_3$.

 Since the arguments are completely symmetric, conversely, if there exists a saddle connection in the top of $C_2$ and the bottom of $C_3$, then there must exist a saddle connection in the top of $C_1$ and the bottom of $C_3$.
\end{proof}

\bigskip

\begin{proof}[Proof of Lemma~\ref{lm:4cyl:C4III:C4}]

Assume  that there exists a surface $M \in \cM$ horizontally periodic satisfying Case 4.III) such that $M$ is $\cM$-cylindrically stable.  Let $k$ be the total number of horizontal saddle connections of $M$. Note that if $M\in \cH(2,1^2)$, then $k=7$, and if $M\in \cH(1^4)$, then $k=8$.
Let $k_i$ be the number of saddle connections contained in the top of $C_i$. By assumption, we have $k_3=2$, and by Lemma~\ref{lm:C4III:C4notS}, we have $k_4 \geq 2$, therefore $2\leq k_1+k_2\leq 4$.

Recall that we need to show that either the closure of $C_4$ contains a free simple cylinder with two distinct zeros in its boundary, or $C_4$ is a semi-simple cylinder.

\medskip

\noindent \ul{$\bullet$ Case $k_1 + k_2 = 2$.} Note that we must have $k_1=k_2=1$. If the saddle connection in the top of $C_1$ is contained in the bottom of $C_3$, then so is the saddle connection in the top of $C_2$ by Lemma \ref{lm:C4III:C4notS}.
But this would imply that the bottom of $C_3$ only contains those two saddle connections (by comparing the lengths of the two boundary components of $C_3$), which is impossible since we have four cylinders. By Lemma~\ref{lm:sc:topC1C2:botC3}, we deduce that the tops of both $C_1$ and $C_2$ are contained in the bottom of $C_4$.
It follows that the bottom of $C_3$ is contained in the top of $C_4$.

Assume that  $M\in \cH(2,1^2)$ then $k_4 = 3$. If the bottom of $C_3$ contains three saddle connections, then it equals the top of $C_4$, which means that $C_3$ and $C_4$ are homologous, but this is excluded by the hypothesis of Case 4.III).   By inspection, we also see that the bottom of $C_3$ cannot contain exactly two saddle connections.
 Therefore, we are left with the case where the bottom of $C_3$ contains only one saddle connection. It follows that there are two saddle connections which are contained in both the top and bottom of $C_4$.  Since $M \in \cH(2,1^2)$, there must exist a saddle connection in the top of $C_4$ connecting the double zero to a simple one.
 Let $D$ be a simple cylinder in $\ol{C}_4$ consisting of simple closed geodesics crossing this saddle connection (see Figure~\ref{fig:4III:C1C2:simple} left). It is not difficult to see that $D$ is free and the lemma is proved for this case.

  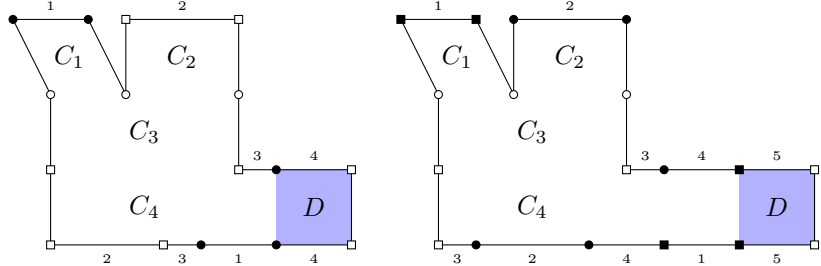
\begin{figure}[htb]
  \begin{minipage}[t]{0.45\linewidth}
  \centering
  \begin{tikzpicture}[scale=0.5]
  \fill[blue!30] (6,0) rectangle (8,2);
  \draw (0,0) -- ( 8,0) -- (8,2) -- (5,2) -- (5,6) -- (2,6) -- (2,4) -- (1,6)  -- (-1,6) -- (0,4) -- cycle;
    \foreach \x in {(0,4), (2,4), (5,4)} \filldraw[fill=white] \x circle (3pt);
    \foreach \x in {(-1,6), (1,6), (4,0),(6,0),(6,2)} \filldraw[fill=black] \x circle (3pt);
    \foreach \x in {(0,2),(0,0), (2,6),(3,0),(5,2), (5,6),(8,2),(8,0)} \filldraw[fill=white] \x +(-3pt,-3pt) rectangle +(3pt,3pt); %{\filldraw[fill=white] \x circle (3pt); \draw \x +(-3pt,0) -- + (3pt,0) +(0,3pt) -- +(0,-3pt); }
    \draw (7,1) node {$D$};
    \draw (0,6) node[above] {\tiny $1$} (5,0) node[below] {\tiny $1$} (3.5,6) node[above] {\tiny $2$} (1.5,0) node[below] {\tiny $2$} (3.5,0) node[below] {\tiny $3$} (5.5,2) node[above] {\tiny $3$} (7,0) node[below] {\tiny $4$} (7,2) node[above] {\tiny $4$};
    \draw (0.5,5) node {$C_1$} (3.5,5) node {$C_2$} (2.5,3) node {$C_3$} (2.5,1) node {$C_4$};
     \end{tikzpicture}
  \end{minipage}
 \begin{minipage}[t]{0.45\linewidth}
  \centering
  \begin{tikzpicture}[scale=0.5]
  \fill[blue!30] (8,0) rectangle (10,2);
  \draw (0,0) -- (10,0) -- (10,2) -- (5,2) -- (5,6) -- (2,6) -- (2,4) -- (1,6) -- (-1,6) -- (0,4) -- cycle;
  \foreach \x in {(0,4), (2,4), (5,4)} \filldraw[fill=white] \x circle (3pt);
  \foreach \x in {(2,6), (5,6), (1,0), (4,0),(6,2)} \filldraw[fill=black] \x circle (3pt);
  \foreach  \x in {(0,2),(0,0),(5,2),(10,2),(10,0)} \filldraw[fill=white] \x +(-3pt,-3pt) rectangle +(3pt,3pt);
  \foreach  \x in {(-1,6), (1,6), (6,0),(8,0),(8,2)} \filldraw[fill=black] \x +(-3pt,-3pt) rectangle +(3pt,3pt);
  \draw (0.5,5) node {$C_1$} (3.5,5) node {$C_2$} (2.5,3) node {$C_3$} (2.5,1) node {$C_4$} (9,1)  node {$D$};
   \draw (0,6) node[above] {\tiny $1$} (7,0) node[below] {\tiny $1$} (3.5,6) node[above] {\tiny $2$} (2.5,0) node[below] {\tiny $2$} (0.5,0) node[below] {\tiny $3$} (5.5,2) node[above] {\tiny $3$} (5,0) node[below] {\tiny $4$} (7,2) node[above] {\tiny $4$} (9,0) node[below] {\tiny $5$} (9,2) node[above] {\tiny $5$};
  \end{tikzpicture}
  \end{minipage}
  \caption{Case 4.III): $C_1$ and $C_2$ are simple in $\cH(2,1^2)$ (left) and $\cH(1^4)$ (right)}
  \label{fig:4III:C1C2:simple}
  \end{figure}

Let us now consider the case $k_4=4$, which means that $M\in \cH(1^4)$. Again, by a careful inspection, one  can show that the bottom of $C_3$ contains only one saddle connection.  The unique cylinder diagram corresponding to this case is shown in Figure~\ref{fig:4III:C1C2:simple} right. We can also easily show that there is a free simple cylinder $D$ contained in $\ol{C}_4$, whose boundary contains two distinct zeros.

\medskip

\noindent \ul{$\bullet$ Case $k_1+k_2=3$}. Up to a renumbering, we can assume that $k_1=1$ and $k_2=2$. We first notice that the top of $C_1$ cannot be contained in the bottom of $C_3$ otherwise we have a contradiction to \cite[Lem. 2.11]{AulicinoNguyenGen3TwoZeros}. Thus the top of $C_1$ must be contained in the bottom of $C_4$. It follows from Lemma~\ref{lm:sc:topC1C2:botC3} that  both saddle connections in the top of $C_2$ are contained in the bottom of $C_4$. Consequently, the bottom of $C_3$ must be contained in the top of $C_4$.

If $k_4=2$, that is $M\in \cH(2,1^2)$, then the bottom of $C_3$ contains a single saddle connection. The unique corresponding cylinder diagram corresponding to this is shown in Figure~\ref{fig:4III:C1sim:C2semsim} (left). Observe that there is a free simple cylinder $D$ contained in $\ol{C}_4$. Remark that the boundary of $D$ contains only the double zero of $M$.
Twisting $C_4$ and simultaneously $\{C_1,C_2,C_3\}$ and using the fact that square-tiled surfaces are dense in $\cM$, we can assume that $M$ is square-tiled and there is a vertical cylinder $D_1$ crossing only $C_1,C_3,C_4$. Note that $D_1$ fills $C_1$ and is disjoint from $D$ and $C_2$. There must exist vertical cylinders $D_2,\dots,D_s$,  $\cM$-parallel to $D_1$ that fill  $C_2$. But it is easy to see that one of the cylinders in the family $\{D_2,\dots,D_s\}$ must intersect $D$. Since $D_1$ does not intersect $D$, this is a contradiction which means that this case cannot occur.

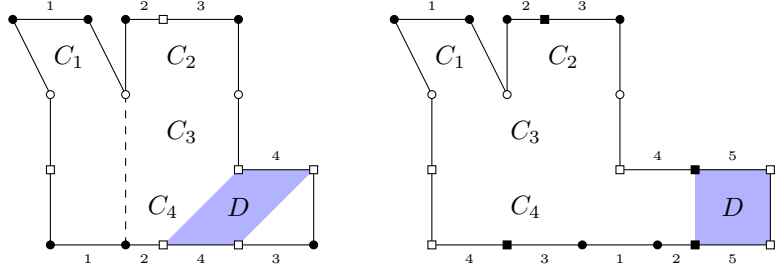
\begin{figure}[htb]
 \begin{minipage}[t]{0.45\linewidth}
\centering
  \begin{tikzpicture}[scale=0.5]
   \fill[blue!30] (3,0) -- (5,0) --( 7,2) -- (5,2) -- cycle;
  \draw (0,0) -- (7,0) -- (7,2) -- (5,2) -- (5,6) -- (2,6) -- (2,4) -- (1,6) -- (-1,6) -- (0,4) -- cycle;
  \draw[thin, dashed] (2,4) -- (2,0);
  \foreach \x in {(0,4), (2,4), (5,4)} \filldraw[fill=white] \x circle (3pt);
  \foreach \x in {(-1,6), (1,6),(2,6), (5,6), (0,0), (2,0),(7,0)} \filldraw[fill=black] \x circle (3pt);
  \foreach  \x in {(0,2),(3,6),(3,0), (5,0),(5,2),(7,2)} \filldraw[fill=white] \x +(-3pt,-3pt) rectangle +(3pt,3pt);
  %\foreach  \x in {(-1,6), (1,6), (6,0),(8,0),(8,2)} \filldraw[fill=black] \x +(-3pt,-3pt) rectangle +(3pt,3pt);
  \draw (0.5,5) node {$C_1$} (3.5,5) node {$C_2$} (3.5,3) node {$C_3$} (3,1) node {$C_4$} (5,1)  node {$D$};
   \draw (0,6) node[above] {\tiny $1$} (1,0) node[below] {\tiny $1$} (2.5,6) node[above] {\tiny $2$} (2.5,0) node[below] {\tiny $2$} (4,6) node[above] {\tiny $3$} (6,0) node[below] {\tiny $3$} (4,0) node[below] {\tiny $4$} (6,2) node[above] {\tiny $4$};
  \end{tikzpicture}
  \end{minipage}
  \begin{minipage}[t]{0.45\linewidth}
\centering
  \begin{tikzpicture}[scale=0.5]
  \fill[blue!30] (7,0) rectangle (9,2);
  \draw (0,0) -- (9,0) -- (9,2) -- (5,2) -- (5,6) -- (2,6) -- (2,4) -- (1,6) -- (-1,6) -- (0,4) -- cycle;
  \foreach \x in {(0,4), (2,4), (5,4)} \filldraw[fill=white] \x circle (3pt);
  \foreach \x in {(-1,6), (1,6),(2,6), (5,6), (4,0), (6,0)} \filldraw[fill=black] \x circle (3pt);
  \foreach  \x in {(0,2),(0,0), (5,2),(9,0),(9,2)} \filldraw[fill=white] \x +(-3pt,-3pt) rectangle +(3pt,3pt);
  \foreach  \x in {(2,0),(3,6),(7,2),(7,0)} \filldraw[fill=black] \x +(-3pt,-3pt) rectangle +(3pt,3pt);
  \draw (0.5,5) node {$C_1$} (3.5,5) node {$C_2$} (2.5,3) node {$C_3$} (2.5,1) node {$C_4$} (8,1)  node {$D$};
   \draw (0,6) node[above] {\tiny $1$} (5,0) node[below] {\tiny $1$} (2.5,6) node[above] {\tiny $2$} (6.5,0) node[below] {\tiny $2$} (4,6) node[above] {\tiny $3$} (3,0) node[below] {\tiny $3$} (1,0) node[below] {\tiny $4$} (6,2) node[above] {\tiny $4$} (8,2) node[above] {\tiny $5$} (8,0) node[below] {\tiny $5$};
  \end{tikzpicture}
  \end{minipage}
  \caption{Case 4.III): $C_1$ is simple and $C_2$ is semi-simple}
  \label{fig:4III:C1sim:C2semsim}
\end{figure}

If $k_4=3$, that is $M\in \cH(1^4)$, then by a careful inspection, we also have that the bottom of $C_3$ only contains one saddle connection. Hence, there are two saddle connections contained in both the top and bottom of $C_4$. Let  $D$ be the simple cylinder in $\ol{C}_4$ as shown  Figure~\ref{fig:4III:C1sim:C2semsim} (right), then one can easily show that $D$ is free and we are done.

\medskip

\noindent \ul{$\bullet$ Case $k_1+k_2=4$.} We have $k_4 \in\{1,2\}$.  By Lemma~\ref{lm:C4III:C4notS}, we know that $k_4=2$ and $M \in \cH(1^4)$.  If there is a saddle connection that is contained in both top and bottom of $C_4$, then we have a free simple cylinder $D$ in $\ol{C}_4$ whose boundary contains two distinct zeros and the lemma follows.  Assume from now on that there is no saddle connection that is contained in both top and bottom of $C_4$.

\medskip

 \ul{\bf Claim 1:} {\em The top of either $C_1$ or $C_2$ cannot be entirely contained in the bottom of $C_4$.}
\begin{proof}
 If the top of either $C_1$ or $C_2$ is contained in the bottom of $C_4$, then by the proof of Lemma~\ref{lm:sc:topC1C2:botC3} the tops of both $C_1$ and $C_2$ are contained in the bottom of $C_4$.  In this case, the bottom of $C_4$ must contain a saddle connection not in the tops of $C_1$ and $C_2$ since otherwise we would have $C_3$ and $C_4$ homologous.  Such a saddle connection must be also  contained in the top of $C_4$  which contradicts our hypothesis.
\end{proof}

\ul{\bf Claim 2:} {\em We have $k_1=k_2=2$.}
 \begin{proof}
 Assume that $k_1=1$ which means that $C_1$ is a simple cylinder. By \cite[Lem. 2.11]{AulicinoNguyenGen3TwoZeros}, the top of $C_1$ is not contained in the bottom of $C_3$, thus it is contained in the bottom of $C_4$. But this is already excluded by the previous claim.
\end{proof}

\ul{\bf Claim 3:} {\em The bottom of $C_4$ contains at most two saddle connections.}
\begin{proof}
The hypothesis that no saddle connection in the top of $C_4$ is also contained in its bottom implies that the top of $C_4$ is contained in the bottom of $C_3$. Claim 1 implies that at least one saddle connection in the top of $C_1$ (resp. $C_2$) is contained in the bottom of $C_3$. Thus the bottom of $C_3$ contains at least four saddle connections (the top of $C_4$ and at least two other saddle connections). Recall that we have 8 saddle connections,  and the bottom of $C_1$ (resp. $C_2$) contains one saddle connection. Hence the bottom of $C_4$ contains at most two saddle connections.
\end{proof}

\ul{\bf Claim 4:} {\em The top of $C_i, \, i \in \{1,2,4\}$, consists of two saddle connections between two distinct simple zeros.}
\begin{proof}
If the top of either $C_1,C_2$, or $C_4$ contains a single zero, then we have two horizontal saddle connections joining this zero to itself. Since all the zeros are simple, there are no saddle connections from this zero to another one. But this contradicts the condition that the graph $\Gr$ consisting of the horizontal saddle connections that do not contain $x_0$ is connected.
\end{proof}
Assume  that the bottom of $C_4$ consists of two saddle connections. Since the top of either $C_1$ or $C_2$ cannot be contained in the bottom of $C_4$, one saddle connection in the bottom of $C_4$ is contained in the top of $C_1$ and the other one is contained in the top of $C_2$. From the same argument as above, we see that there are two distinct zeros in the bottom of $C_4$.  Therefore, there must be a zero that is contained in both top and bottom of $C_4$. Note that this zero is  contained in the tops of $C_1, C_2$, and in the bottom of $C_3$. By an angle count, one can easily see that the total angle at this zero is at least $5\pi$,  which is a contradiction since all of the zeros are simple. We can then conclude that the bottom of $C_4$ consists of a single saddle connection, which means that $C_4$ is semi-simple. The proof of the lemma is now complete.
%\end{itemize}
\end{proof}

\section{Proof  of  Lemma~\ref{lm:C4II:preliminary}}\label{sec:proof:lm:C4II:preliminary}

\begin{lemma}
\label{Case4IIEqCls}
Let $M$ be a horizontally periodic translation surface satisfying Case 4.II) in a rank two affine manifold $\cM \subset \cH(2,1^2) \cup \cH(1^4)$.  Assume that $M$ is $\cM$-cylindrically stable. Then $C_3$ and $C_4$ are $\cM$-parallel.
\end{lemma}
\begin{proof}
Since $\cM$ has rank two, both $C_3$ and $C_4$ cannot be free otherwise we would have a Lagrangian subspace of dimension three in $p(T^\R_M\cM)$.  Hence, it suffices to consider the possibility that one of $C_3$ or $C_4$ is free.  Without loss of generality, assume by contradiction that $C_3$ is free and $C_4$ is $\cM$-parallel to $\{C_1,C_2\}$.

Let us first consider the case $\ol{C}_3$ contains a simple cylinder $D$.  Recall that  $C_3$ can be viewed as  one cylinder in a 2-cylinder decomposition of a genus two translation surface.  Hence, there are at most two saddle connections that are contained in both top and bottom borders of $C_3$. From this observation, it is not difficult to see that $D$ is free.

Assume that the boundary of $D$ contains two simple zeros, then we can collapse $D$ to get a surface $M'$ which is contained in a rank two affine submanifold $\cM'$ of either $\cH(2,2)$ or $\cH(2,1^2)$. Note that $M'$ also has a cylinder decomposition satisfying Case 4.II) in the horizontal direction.

 By assumption, we know that $M'$ has an involution $\inv$ with four fixed points. By inspection, we see that if $M' \in \cH(2,1^2)$, then $\inv$ must fix $C_3$ and $C_4$ and exchange $C_1$ and $C_2$. But this would imply that $\inv$ has at least five fixed points since the double zero must be fixed by $\inv$ and we have $4$ regular fixed points in the interiors of $C_3$ and $C_4$. So we have a contradiction in this case.
 If $M'\in \cH(2,2)$, then we have two possibilities, either $\inv$ fixes $C_1,C_2$ and exchanges $C_3,C_4$, or $\inv$ fixes $C_3, C_4$ and exchanges  $C_1$ and $C_2$. In either case, we see that $C_4$ is not $\cM'$-parallel to $C_1,C_2$. Thus in any neighborhood of $M'$ in $\cM'$, we can find a surface $M'_1$ on which $C_4$ is not parallel to $C_1,C_2$.  By the isomorphism from \cite{MirzakhaniWrightBoundary} of the tangent space of $\cM'$ with a subspace of the tangent space of $\cM$, we see that there exists in any neighborhood of $M$ in $\cM$ a surface $M_1$ on which $C_4$ is not parallel to $\{C_1,C_2\}$, which means that $C_4$ is not $\cM$-parallel to
$\{C_1,C_2\}$.

 \medskip

Suppose now that the boundary of $D$ contains a double zero, which means that $M\in \cH(2,1^2)$. The assumption means that $C_3$ is contained in a translation surface in the stratum $\cH(2)$. Since there is only one cylinder diagram for 2-cylinder decompositions of surfaces in $\cH(2)$, we see that one side of $C_1$ (resp. $C_2$) contains only one saddle connection, which means that $C_1$ and $C_2$ are semi-simple. Note also that in this case any vertical ray that crosses $C_1$ or $C_2$ must intersect $C_3$.

Consider now $C_4$. If $\ol{C}_4$ contains a simple cylinder, then we have a contradiction by \cite[Lem. 2.12]{AulicinoNguyenGen3TwoZeros}. Thus in this case $C_4$ must be simple. Note that there is only one 4-cylinder diagram satisfying all of these conditions which is shown in Figure~\ref{fig:C4II:H211:C4sim}.
By a similar argument as in \cite[Lem. 6.17]{AulicinoNguyenGen3TwoZeros}, we can twist $\{C_1,C_2,C_4\}$ and $C_3$ independently so that  $D$ is vertical, and there exists a vertical cylinder $E$ crossing each of $C_1,C_2,C_3$ once. Since $C_4$ is $\cM$-parallel to $C_1$ and $C_2$, there must exists a vertical cylinder $E'$ in the equivalence class of $E$ that crosses $C_4$. Let $h_i$ be the height of $C_i$ and $n_i$ be the number of times that a core curve of $E'$ crosses $C_i$. Note that we have $n_1=n_2=n_3=n$, and $0< n_4 \leq n$. By the Cylinder Proportion Lemma we have
$P(E,C_3)=P(E',C_3)$, which implies
$$
\frac{h_3}{h_1+h_2+h_3}=\frac{nh_3}{n(h_1+h_2+h_3) +n_4h_4} \Leftrightarrow n_4h_4=0.
$$
But this is clearly impossible. Thus we also have a contradiction.

\begin{figure}[htb]
  \centering
  \begin{tikzpicture}[scale=0.3]
  \fill[blue!30] (9,6) rectangle (12,9);
  \fill[green!30] (0,4) rectangle (2,11);
  \draw (0,11) -- (0,4) -- (2,4) -- (2,1) -- (5,1) -- (5,4) -- (9,4) -- (9,6) --  (12,6) -- (12,9) -- (9,9) -- (9,11) -- cycle; \draw (0,9) -- (9,9) (0,6) -- (9,6);
    \foreach \x in {(2,11), (2,4), (5,4), (9,4)} \filldraw[fill=white] \x circle (3pt);
    \foreach \x in {(0,11), (0,4), (2,1), (5,1), (6,11), (9,11), (9,4)} \filldraw[fill=black] \x circle (3pt);
    \foreach \x in {(0,9), (0,6), (9,9), (9,6), (12,9), (12,6)} \filldraw[fill=white] \x +(-3pt,-3pt) rectangle +(3pt,3pt);
    %{\filldraw[fill=white] \x circle (3pt); \draw \x +(-3pt,0) -- + (3pt,0) +(0,3pt) -- +(0,-3pt); }
    \draw (10.5,7.5) node {$D$};
    \draw (1,8) node {$E$};
    \draw (5,10) node {$C_1$};
    \draw (5,7.5) node {$C_3$};
    \draw (5,5) node {$C_2$};
    \draw (3.5,2.5) node {$C_4$};
 \end{tikzpicture}
 \caption{Cylinder decomposition in Case 4.II), $M\in \cH(2,1,1)$, $C_4$ is simple}
 \label{fig:C4II:H211:C4sim}
\end{figure}
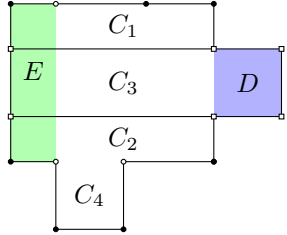

It remains to consider the case $C_3$ is simple. One can twist $C_3$ so that it contains no vertical saddle connections and perform an extended cylinder collapse from
\cite[Proof. of Lem. 4.7]{AulicinoNguyenGen3TwoZeros} to get a new cylinder, which we call $C_3$ by abuse of notation.
This new cylinder contains a simple cylinder, so we are back to the previous case. The proof of the lemma is then complete.
\end{proof}

\begin{lemma}
\label{C3CylImpC4Cyl}
The cylinder $C_3$ contains a simple cylinder if and only if $C_4$ contains a simple cylinder.\footnote{The simple cylinders in this lemma need not be parallel, but a posteriori we will see that they are.}
\end{lemma}
\begin{proof}
Recall that $C_3$ (resp. $C_4$) either is simple, or contains a simple cylinder. Since $C_3$ (resp. $C_4$) is only adjacent to $C_1$ and $C_2$, this lemma is an easy consequence of the Cylinder Proportion Lemma.
\end{proof}

\begin{lemma}\label{lm:C4II:C3C4:similar}
Let $k_3$ (resp. $k_4$) be the number of saddle connections contained in both top and bottom of $C_3$ (resp. $C_4$). Then $k_3=k_4$.
\end{lemma}
\begin{proof}
 Note that we have $0 \leq k_i \leq 2, \, i=3,4$, since $C_3$ (resp. $C_4$) can be viewed as a cylinder in a $2$-cylinder decomposition of a translation surface of genus two. If $k_3=0$, then $C_3$ is simple, and by Lemma~\ref{C3CylImpC4Cyl}, we know that $C_4$ is simple, hence $k_4=0$. Thus we can assume that $k_3$ and $k_4$ are both non-zero.  Since the roles of $C_3$ and $C_4$ can be exchanged, we only need to consider the case $k_3=1$ and $k_4=2$.

 Let us denote by $\sig_3$ the unique saddle connection which is contained in both the top and bottom of $C_3$.  The two saddle connections contained in both the top and bottom of $C_4$ are denoted by $\sig_4$ and $\sig'_4$. We can twist $C_4$ (and simultaneously $C_3$) such that there is a subdomain $R$ of $C_4$ isometric to a rectangle whose top and bottom sides are the union of $\sig_4$ and $\sig'_4$. Since $\cM$ is defined over $\Q$, we can assume that $M$ is a square-tiled surface, which means that the vertical direction is periodic.
There is a vertical cylinder $D$ whose closure  equals the closure of $R$. In particular, $D$ is contained in $\ol{C}_4$.

Since $C_3$ and $C_4$ are $\cM$-parallel, it follows that the closure of $C_3$ must contain a vertical cylinder $D'$ which is $\cM$-parallel to $D$. Note that $D'$ must be a simple cylinder and $\sig_3$ is entirely contained  in $D'$.

We now remark that the equivalence class of $D$ must be $\{D,D'\}$, since any other vertical cylinder must cross $C_1$ or $C_2$. We can ``stretch'' simultaneously $D$ and $D'$ so that their heights are very small with respect to the lengths of the horizontal saddle connections outside of $D\cup D'$. Note that as the heights of $D$ and $D'$ tend to zero, the lengths of $\sig_3,\sig_4,\sig'_4$ also decrease to zero.

Observe that we have a simple cylinder $E$ that is contained in the closure of $R$ consisting of simple closed geodesics crossing $\sig_4$ once. As the height of $D$ decreases to  zero, the direction of $E$ converges to the vertical direction.  There must exist a cylinder $E'$ that is $\cM$-parallel to $E$  which crosses $C_3$. But as the direction of $E$ is close to vertical, such a cylinder cannot be contained in the closure of $C_3$.  Hence, it must cross $C_1$ or $C_2$ from which we get a contradiction.
\end{proof}

\subsection*{Proof of Lemma~\ref{lm:C4II:preliminary}}
\begin{proof}
Lemma~\ref{lm:C4II:preliminary} follows from Lemmas \ref{Case4IIEqCls} and \ref{lm:C4II:C3C4:similar}.
\end{proof}

\section{6-Cylinder Diagrams in Genus Three}\label{sec:6cyl:diag:proof}

In this section, we give the proof of Proposition~\ref{prop:6CylDiags}.

The following lemma is well known to most of people in the field, we provide here a proof for the sake of completeness.

\begin{lemma}\label{lm:max:num:cyl:ppants}
Let $M$ be a horizontally periodic translation surface in a stratum $\cH(\kappa)$ of genus $g$, where $|\kappa|=n$. Denote by  $C_1,\dots,C_k$ the horizontal cylinders of $M$, and let  $\gamma_i$ be a core curve of $C_i$ for $i=1,\dots,k$. Then we have
\begin{itemize}
\item[(a)] $k \leq g+n-1$,

\item[(b)] If $k=g+n-1$ then complement of the curves  $\{\gamma_1,\dots,\gamma_k\}$ is the disjoint union of $n$ punctured spheres, each of which contains a unique singularity of $M$.  In particular, if $\kappa=(1,\dots,1)$, and $k=3g-3$, then each of those components is the interior of a pair of pants (or a thrice-punctured sphere).
\end{itemize}
\end{lemma}
\begin{proof}
Cut $M$ along the curves $\gamma_1,\dots,\gamma_k$, we get $m$ compact surfaces with boundary denoted by $\tilde{M}_1,\dots,\tilde{M}_m$. Since each $\tilde{M}_j$ must contain a singularity of $M$, we have $m \leq n$. Let $g_i$ and $r_i$ be respectively the genus and the number of boundary components  of $\tilde{M}_j$. Since we have
$$
\chi(M)=\chi(\tilde{M}_1)+\dots+\chi(\tilde{M}_m),
$$
it follows
$$
2-2g=\sum_{j=1}^m (2-2g_j-r_j) \leq \sum_{j=1}^m (2-r_j)=2m -\sum_{j=1}^m r_j=2m-2k \leq 2n-2k.
$$
Thus we have
$$
k \leq g+n-1.
$$
The equality occurs if and only if  $m=n$, and $g_j=0$, for $j=1,\dots,m$, which means that each $\tilde{M}_j$ is a sphere with some discs removed and contains a unique singularity of $M$.
If $\kappa=(1,\dots,1)$, each $\tilde{M}_j$ contains a cone point of angle $4\pi$. The Gauss-Bonnet Theorem then implies that we must have $\chi(\tilde{M}_j)=-1$.
\end{proof}

\medskip

\begin{proof}[Proof of Proposition~\ref{prop:6CylDiags}]
Let $C_i, \, i=1,\dots,6,$ denote the horizontal cylinders of $M$, and let $\gamma_i$ be a core curve of $C_i$. By Lemma~\ref{lm:max:num:cyl:ppants}, the family $\{\gamma_1,\dots,\gamma_6\}$ cuts $M$ into $4$ pairs of pants denoted by $\tilde{M}_1,\dots,\tilde{M}_4$. Let $x_j$ be the unique singularity of $M$ that is contained in $\tilde{M}_j$.

\begin{figure}[htb]
\centering
\begin{minipage}[t]{0.4\linewidth}
\centering
\begin{tikzpicture}[scale=0.3]
\draw (-5,8) -- (-4,4) -- (-4,0);
\draw (5,8) -- (4,4) -- (4,0);
\draw (-1,8) -- (0,4) -- (1,8);

\draw (-4,0) arc (180:360:4 and 1.5);
\draw[dashed] (4,0) arc (0:180: 4 and 1.5);

\draw (-1,8) arc (0:180: 2 and 1);
\draw (-5,8) arc (180:360:2 and 1);
%\draw (-5,8) arc (180:270: 2 and 1);
%\draw (-3,7) arc (270:360:2 and 1);

\draw (5,8) arc (0:180:2 and 1);
\draw (1,8) arc (180:360:2 and 1);
%\draw (1,8) arc (180:270: 2 and 1); \draw (3,7) .. controls (4,7) and (4,7) .. (5,8);

\draw (-4,4) arc (180:360: 2 and 1);
\draw[dashed] (0,4) arc (0:180: 2 and 1);
%\draw[dashed] (-2,5) arc (90:180: 2 and 1); \draw[dashed] (0,4) .. controls (-1,5) and (-1,5) .. (-2,5);

\draw (0,4) arc (180:360: 2 and 1);
\draw[dashed] (4,4) arc (0:180: 2 and 1);

%\draw[dashed] (4,4) arc (0:90: 2 and 1); \draw (0,4) .. controls (1,5) and (1,5) .. (2,5);

\filldraw[fill=white] (0,4) circle (4pt);
\end{tikzpicture}
\end{minipage}

\caption{A component $\tilde{M}_j$}
\label{fig:pairofpants}
\end{figure}
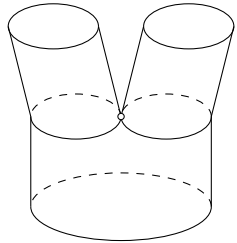
Here below, we record some properties of the cylinders $C_1,\dots,C_6$.
\begin{itemize}
\item[(a)] Each boundary component of $C_i$ has at most $2$ saddle connections, and contains a unique zero of $M$.

\item[(b)] Each zero is contained in the boundary of $3$ cylinders.

\item[(c)] Each cylinder contains two distinct zeros in its boundary.
\end{itemize}

For $i=1,\dots,6$, let $\delta_i$ be a saddle connection  in $C_i$ connecting the pair of zeros in its boundary. The union $\cup_{i=1}^6\delta_i$ is  an embedded graph   $\Gamma$ in $M$. This graph is also the dual graph of the nodal curve obtained from $M$ by pinching  $\gamma_1,\dots,\gamma_6$. By definition, $\Gamma$ has $4$ vertices and $6$ edges. There are $2$ admissible configurations for $\Gamma$ that are shown in Figure~\ref{fig:6cyl:dualgraph}. In Case 1, any pair of vertices are connected by only one edge, in Case 2 there are two pairs of vertices such that there are two edges between the vertices in each pair. We will derive the possible cylinder diagrams from the configurations of $\Gamma$.

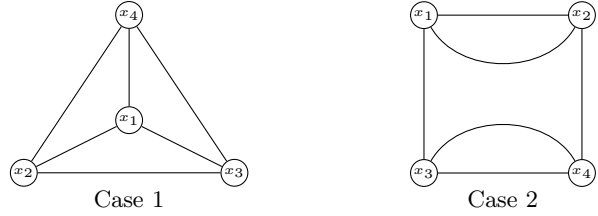
\begin{figure}[htb]
\centering
\begin{minipage}[t]{0.4\linewidth}
\centering
\begin{tikzpicture}[scale=0.35, inner sep=0.2mm, vertex/.style={circle, draw=black, fill=white, minimum size=1mm},>= stealth]
\node (center) at (0,0) [vertex] {\tiny $x_1$};
\node (above) at (0,4) [vertex] {\tiny $x_4$};
\node (left) at (-4,-2) [vertex] {\tiny $x_2$};
\node (right) at (4,-2) [vertex] {\tiny $x_3$};

\draw (above) to (right);
\draw (above) to (left);
\draw (left) to (right) ;
\draw (center) to (above);
\draw (center) to (left);
\draw (center) to  (right);

\draw (0,-3) node {\small Case 1};

\end{tikzpicture}
\end{minipage}
\begin{minipage}[t]{0.4\linewidth}
\centering
\begin{tikzpicture}[scale=0.35, inner sep=0.2mm, vertex/.style={circle, draw=black, fill=white, minimum size=1mm},>= stealth]
\node (above-left) at (-3,6) [vertex] {\tiny $x_1$};
\node (above-right) at (3,6) [vertex] {\tiny $x_2$};
\node (below-left) at (-3,0) [vertex] {\tiny $x_3$};
\node (below-right) at (3,0) [vertex] {\tiny $x_4$};

\draw (above-left) to (below-left);
\draw (above-right) to (below-right);
\draw (above-left) to (above-right) ;
\draw (below-left) to (below-right);
\draw (above-left) to[out=-60, in=240] (above-right);
\draw (below-left) to[out=60, in=120]  (below-right);

\draw (0,-1) node {\small Case 2};

\end{tikzpicture}
\end{minipage}

\caption{Configurations of the graph $\Gamma$}
\label{fig:6cyl:dualgraph}
\end{figure}

\medskip

\noindent \underline{\em Case 1:} Let $\ell_i$ be the circumference of $C_i$, and assume that $\ell_1=\max\{\ell_1,\dots,\ell_6\}$. We claim that each boundary component of $C_1$ contains two saddle connections. This is because otherwise $C_1$ is a semi-simple cylinder, and there would be another cylinder $C_i$ such that $\ell_i > \ell_1$.

We can assume that the zeros of $M$ in the top and bottom borders of $C_1$ are respectively $x_1$ and $x_2$.
We now remark that each saddle connection in the top border of $C_1$ is the bottom border of another cylinder. We can assume that the cylinders whose bottom border  is contained in the top of $C_1$ are $C_2$ and $C_3$. Similarly, there are two cylinders $C_i,C_j$ whose top border is contained in the bottom border of $C_1$. We claim that $\{i,j\}\cap\{2,3\}=\varnothing$ because otherwise there would be two edges in $\Gamma$ between $x_1$ and $x_2$. Thus we can assume that $\{i,j\}=\{4,5\}$.

Note that by the same argument we  see that the top borders of $C_2$ and $C_3$ contain two distinct zeros, which are neither $x_1$ nor $x_2$. The same is true for the bottom borders of  $C_4$ and $C_5$. Thus we can assume that the top of $C_2$ and the bottom of $C_4$ contain the same zero $x_3$. Consequently, the top of $C_3$ and the bottom of $C_5$ contain $x_4$. Without loss of generality, we can suppose that $\ell_2 < \ell_4$, which means that $C_2$ is a simple cylinder, while $C_4$ is strictly semi-simple, and the bottom border of $C_4$ contains two saddle connections.
Since we have $\ell_2+\ell_3=\ell_4+\ell_5=\ell_1$, it follows that $\ell_3 > \ell_5$. Hence $C_5$ is a simple cylinder, and the top of $C_3$ contains two saddle connections. From this we deduce that the cylinder $C_6$ must be simple, with top border contained in the bottom border of $C_4$, and bottom border contained in the top border of $C_3$. In conclusion, there is a unique cylinder diagram corresponding to this configuration of $\Gamma$. This cylinder diagram  is depicted  in Case 6.a of Figure~\ref{6CylDiagsFig} with a different labeling of the cylinders.

\medskip

\noindent \underline{\em Case 2:} We can assume that there are two edges between $x_1$ and $x_2$ and between $x_3$ and $x_4$. Observe that in this case, there are two cylinder core curves  that separate $\{x_1,x_2\}$ from $\{x_3,x_4\}$. In particular, they are homologous. Cutting $M$ along those curves and permuting the gluings, we obtain two translation surfaces in $\cH(1,1)$, each of which admits a $3$-cylinder decomposition in the horizontal direction. Therefore, one can recover the cylinder diagram of $M$ from the unique $3$-cylinder diagram for $\cH(1,1)$ and a choice of regluing. The possible diagrams are depicted by Cases 6.b, 6.c, and 6.d of Figure~\ref{6CylDiagsFig}.
\end{proof}

\bibliography{fullbibliotex_final}{}

%[EM] Zbl 06914160
%[EMM] Zbl 06731862

\end{document}